\documentclass [leqno,11pt]{amsart}
\usepackage{amssymb, amsmath,latexsym,amsfonts,amsbsy, amsthm,mathtools,graphicx,CJKutf8,CJKnumb,CJKulem,color}
\usepackage{float}
\usepackage{hyperref}

\numberwithin{equation}{section}

\allowdisplaybreaks

\let\al=\alpha

\let\f=\frac

\let\Om=\Omega

\let\na=\nabla

\let\pa=\partial

\def\R{\mathbb R}

\def\no{\noindent}

\newcommand{\beq}{\begin{equation}}
\newcommand{\eeq}{\end{equation}}
\newcommand{\ben}{\begin{eqnarray}}
\newcommand{\een}{\end{eqnarray}}
\newcommand{\beno}{\begin{eqnarray*}}
\newcommand{\eeno}{\end{eqnarray*}}

\newtheorem{theorem}{Theorem}[section]
\newtheorem{definition}[theorem]{Definition}
\newtheorem{lemma}[theorem]{Lemma}
\newtheorem{proposition}[theorem]{Proposition}

\newtheorem{remark}[theorem]{Remark}

\begin{document}

\begin{CJK*}{UTF8}{gkai}

\title[Stability of shear flows of Prandtl type]
{On the stability of shear flows of Prandtl type for the steady Navier-Stokes equations}

\author{Qi Chen}
\address{School of Mathematical Science, Peking University, 100871, Beijing, P. R. China}
\email{chenqi940224@gmail.com}

\author{Di Wu}
\address{School of Mathematics, South China University of Technology, Guangzhou, 510640,  P. R. China}
\email{wudi@scut.edu.cn}

\author{Zhifei Zhang}
\address{School of Mathematical Science, Peking University, 100871, Beijing, P. R. China}
\email{zfzhang@math.pku.edu.cn}

\date{\today}

\maketitle

\begin{abstract}
In this paper, we prove the stability of shear flows of Prandtl type as $
\big(U(y/\sqrt{\nu}),0\big)$ for the steady Navier-Stokes equations under a natural spectral assumption on the linearized NS operator.
We develop a direct energy method combined with compact method  to solve the Orr-Sommerfeld equation.

\end{abstract}

\section{Introduction}

\subsection{The problem and main result}
In this paper, we study the steady Navier-Stokes equations with small viscosity $\nu$ on the half-plane $\Om_\theta=\mathbb{T}_\theta\times\mathbb{R}_+$, where $\mathbb{T}_\theta$ is a torus with the period $2\pi\theta$:
\begin{align}\label{eq:SNS}
\left\{
\begin{aligned}
&v^{\nu}\cdot\nabla v^{\nu}-\nu\Delta v^\nu+\nabla q^\nu=g^\nu,\quad(x,y)\in\Omega_\theta,\\
&\nabla\cdot v^\nu=0,\quad(x,y)\in\Omega_\theta,\\
&v^\nu|_{y=0}=0,\quad x\in\mathbb{T}_\theta,
\end{aligned}\right.
\end{align}
where $v^\nu=(v_1^\nu,v_2^\nu)$ and $q^\nu$ are the unknown velocity and pressure of the fluid respectively, and $g^\nu$ is a given external force.

We are concerned with the stability of shear flows of Prandtl type as $U^\nu(x,y)=(U(y/\sqrt{\nu}),0)$, which is
a solution of \eqref{eq:SNS} with $g^\nu=-\nu\partial_y^2 U^\nu$ and $q^{\nu}=0$. We make the following structure assumptions on $U$:
\begin{eqnarray}\label{ass:U}
\begin{split}
   & U(0)=0,\quad U(Y)>0\ \text{in}\ Y>0,\quad \lim_{Y\rightarrow +\infty}U(Y)=1,\\
   &\partial_YU(0)>0,\quad \sum_{k=1,2}\sup_{Y\geq 0}(1+Y)^3|\partial_Y^kU(Y)|<\infty.
\end{split}
\end{eqnarray}
The assumption \eqref{ass:U}  implies that
\ben
  U(Y)\geq C^{-1}\min(1,Y)\sim Y/(1+Y).\label{eq:U-S1}
\een

In this paper, we take the external force $g^\nu=-\nu\pa_y^2U^\nu+f^\nu$ so that $u^\nu=v^\nu-U^\nu$ satisfies
\begin{align}\label{eq:perturbation-NS}
\left\{\begin{aligned}
&U(\frac{y}{\sqrt{\nu}})\partial_x u^\nu+\Big(u_2^\nu\partial_yU(\frac{y}{\sqrt{\nu}}),0\Big)-\nu\Delta u^\nu+\nabla p^\nu=-u^\nu\cdot\nabla u^\nu+f^\nu,\\
&\nabla\cdot u^\nu=0,\\
&u^\nu|_{y=0}=0.
\end{aligned}\right.
\end{align}
We denote by $\mathcal{P}_n, n\in\mathbb{Z}$ the  projection on the $n-$th Fourier mode in $x$ variable:
\begin{align*}
(\mathcal{P}_n u)(x,y)=u_n(y)e^{\mathrm{i}\tilde{n}x},\quad\tilde{n}=\frac{n}{\theta},\quad u_n(y)=\frac{1}{2\pi\theta}\int_0^{2\pi\theta}e^{-\mathrm{i}\tilde{n}x}u(x,y)dx.
\end{align*}
We denote 
$$\mathcal{Q}_0=I-\mathcal{P}_0,$$
where $I$ is the identity operator. It is easy to see that $P_0u^\nu=\big(u_{0,1}^\nu,0\big)$ due to $\na\cdot u^\nu=0$.

In a recent work \cite{GM}, D. Gerard-Varet and Y. Maekawa proved the following stability result: \\
{\it there exist small constants $\nu_0, \theta_0$ so that if $0<\nu\le \nu_0$ and $0<\theta\le \theta_0$ and $f^\nu=\mathcal{Q}_0f^\nu$ with $\|f^\nu\|_{L^2}\leq\eta \nu^{\f34}|\log\nu|^{-1}$,
then it holds that
\begin{eqnarray}\nonumber
\begin{split}
\|u^\nu_{0,1}\|_{L^\infty}+&\nu^{\f14}\|\partial_y u^\nu_{0,1}\|_{L^2}+\sum_{n\neq0}\|u^\nu_n\|_{L^\infty}\\
+&\nu^{-\f14}\|\mathcal{Q}_0 u^\nu\|_{L^2}+\nu^{\f14}\|\nabla\mathcal{Q}_0 u^\nu\|_{L^2}\leq C\nu^{-\f14}|\log\nu|^{\f12}\|f^\nu\|_{L^2}.
\end{split}
\end{eqnarray}}
The smallness assumption $\theta\le \theta_0$ plays the key role in the proof of the stability.\smallskip

The aim of this paper is twofold. The first aim of this paper is to replace the smallness assumption $\theta\le \theta_0$ by the following more natural condition: 
$$\theta\in \Sigma(U,\nu),$$
 where the set $\Sigma(U,\nu)$ consists of the positive number $\theta$ so that the following linearized Navier-Stokes equations have a unique solution $u^\nu\in H^2(\Om_\theta)\cap H^1_0(\Om_\theta)$ for $f^\nu\in L^2(\Om_\theta)$:
\begin{align}\nonumber
\left\{\begin{aligned}
&U(\frac{y}{\sqrt{\nu}})\partial_x u^\nu+\Big(u_2^\nu\partial_yU(\frac{y}{\sqrt{\nu}}),0\Big)-\nu\Delta u^\nu+\nabla p^\nu=f^\nu,\\
&\nabla\cdot u^\nu=0,\\
&u^\nu|_{y=0}=0.
\end{aligned}\right.
\end{align}
The set $\Sigma(U,\nu)$ is not empty, since small $\theta$ belongs to this set, which has been proved in \cite{GM}.
In fact, our assumption is equivalent to the spectral condition: {\it $0$ is not an eigenvalue of  the linearized operator $\mathcal{L}_\nu$ defined by
\beno
\mathcal{L}_\nu u=\mathrm{P}^\nu\Big(U(\frac{y}{\sqrt{\nu}})\partial_x u+\Big(u_2\partial_yU(\frac{y}{\sqrt{\nu}}),0\Big)\Big)-\nu\mathrm{P}^\nu \Delta u,
\eeno
where $\mathrm{P}^\nu$ is the Helmoholtz-Leray projection. } It is very interesting to determine appropriate structure assumptions on $U$ such that $0\notin \sigma(\mathcal{L}_\nu)$.

The second aim of this paper is to solve the Orr-Sommerfeld equation by developing a direct method in the sprit of our paper \cite{CWZ},
which should be independent of interest. The proof in \cite{GM} used the Rayleigh-Airy iteration method introduced in \cite{GGN, GMM} to solve the Orr-Sommerfeld equation.

Now we state our main result.

\begin{theorem}\label{th:main}
There exist positive constants $\nu_0$ and $\eta$ such that the following result holds for $0<\nu\leq\nu_0$: if $\theta\in \Sigma(U,\nu)$ and $f^\nu=\mathcal{Q}_0f^\nu$ with $\|f^\nu\|_{L^2}\leq\eta \nu^{\f34}|\log\nu|^{-1}$, then the unique solution $u^\nu$ to \eqref{eq:perturbation-NS} satisfies
\begin{eqnarray}
\begin{split}
\|u^\nu_{0,1}\|_{L^\infty}+&\nu^{\f14}\|\partial_y u^\nu_{0,1}\|_{L^2}+\sum_{n\neq0}\|u^\nu_n\|_{L^\infty}\\
+&\nu^{-\f14}\|\mathcal{Q}_0 u^\nu\|_{L^2}+\nu^{\f14}\|\nabla\mathcal{Q}_0 u^\nu\|_{L^2}\leq C\nu^{-\f14}|\log\nu|^{\f12}\|f^\nu\|_{L^2},
\end{split}
\end{eqnarray}
where $C$ is independent of $\theta$ and $\nu_0$.
\end{theorem}

\begin{remark}
When $\theta$ is small, the spectral condition holds and our result is the same as one in \cite{GM}. For $\theta$ not small but $\theta\in \Sigma(U, \nu)$, our result is new.
\end{remark}

\begin{remark}
The works \cite{GMM, CWZ} consider the stability of shear flows of Prandtl type for the unsteady Navier-Stokes equations for the perturbations in the Gevrey class.  In such case, it is enough to establish the resolvent estimates for the linearized operator when the spectral parameter lies in an unstable part of the resolvent set. While, in the steady case, one needs to establish the resolvent estimates when the spectral parameter equals to zero. However, it remains unknown whether $0$ lies in the resolvent set of $\sigma(\mathcal{L}_\nu)$ for general period $2\pi \theta$ in $x$. This is the essential reason why we impose the spectral condition $0\notin \sigma(\mathcal{L}_\nu)$.\smallskip

\end{remark}

\subsection{Some known results}

Most of earlier mathematical works are devoted to the study of the inviscid limit of the unsteady Navier-Stokes equations.
In the absence of the boundary, the inviscid limit from the Navier-Stokes equations to the Euler equations has been justified in various functional settings \cite{Kato, Swann, BM, Mas, CW, AD, Mar}. In the presence of the boundary, under the Navier-slip boundary condition, the inviscid limit has been also established in \cite{CMR, IP, MR, IS, XX, WXZ}.  

For non-slip boundary condition, due to the mismatch of the boundary condition between the Navier-Stokes equations and the Euler equations, Prandtl introduced the following boundary layer expansion:\ben\label{eq:Pran-exp}
 \left\{
 \begin{array}{l}
 u^{\nu}_1(t,x,y)=u^{e}_1(t,x,y)+ u^{BL}\big(t,x,\f{y}{\sqrt{\nu}}\big)+O(\sqrt{\nu}),\\
 u^{\nu}_2(t,x,y)=u^{e}_2(t,x,y)+\sqrt{\nu}v^{BL}\big(t,x,\f{y}{\sqrt{\nu}}\big)+O(\sqrt{\nu}),
 \end{array}\right.
 \een
where $(u_1^\nu, u_2^\nu)$ is the solution of the 2-D Navier-Stokes equations on the half space with non-slip boundary condition, $(u_1^e, u_2^e)$ is the solution of the 2-D Euler equations, and $(u^p,v^p)=\big(u^e_1(t,x,0)+u^{BL}(t,x,Y), \pa_yu^e_2(t,x,0)Y+v^{BL}(t,x,Y)\big)$ satisfies the Prandtl equation. To our knowledge, the Prandtl expansion \eqref{eq:Pran-exp} was justified only in some special cases:
 the analytic data  \cite{SC2, WWZ},  the initial vorticity vanishing near the boundary \cite{Mae, FTZ}, as well as the domain and the data with a circular symmetry  \cite{LMT, MT}.  Initiated by Kato \cite{Kato1}, there were many works devoted to the conditional convergence to the Euler equations \cite{TW, Wang, Ke}. These conditions can be confirmed for the data which is analytic near the boundary \cite{KVW, WF}.
 \smallskip

Grenier \cite{Gre} studied the Prandtl expansion of shear type flows  as
\ben\label{eq:Prandtl exp-shear}
u^\nu_s=\big(U^e(t,y),0\big)+\Big(U^{BL}\big(t,\f y {\sqrt{\nu}}\big),0\Big).
\een
When the shear flow $U^{BL}(0,Y)$ is linearly unstable for the Euler equations, he proved the $H^1$ instability of this expansion, and also $L^\infty$ instability in \cite{GNg}. In fact, even for the shear flows which are linearly stable for the Euler equations,
Guo, Grenier and Nguyen \cite{GGN} proved that it could be linearly unstable for the Navier-Stokes equations when $\nu$ is very small.
When $U^{BL}(t,Y)$ is a monotone and concave function, Gerard-Varet,  Masmoudi  and  Maekawa \cite{GMM}(see also \cite{CWZ}) proved the stability of the Prandtl expansion \eqref{eq:Prandtl exp-shear} for the perturbations in the Gevrey class(see \cite{GMM2} for general cases).

Guo and Nguyen \cite{GN}  considered the Prandtl expansion of the steady Navier-Stokes equations over a moving plate.
See \cite{Iy1, Iy2, Iy3}  for more relevant works. For the steady Navier-Stokes equations with non-slip boundary condition, Gerard-Varet and Maekawa \cite{GM} proved the stability of shear flows  $\big(U\big(\frac y {\sqrt{\nu}}\big),0\big)$ in the Sobolev space. Guo and  Iyer  \cite{GI} proved the local in time stability of the Blasius flow, and Iyer and Masmoudi \cite{IM} recently proved the global stability.
See also \cite{GZ} for  a different method.

\subsection{Sketch of the proof}

The road map of the proof of Theorem \ref{th:main} is as follows.  \smallskip

The key step is to  study the following linearized NS system for each Fourier mode:
\begin{align}\label{eq:per-n-mode}
\left\{\begin{aligned}
&\mathrm{i}\tilde{n} U(\frac{y}{\sqrt{\nu}})u_n+\Big(u_{n,2}\partial_y U(\frac{y}{\sqrt{\nu}}),0\Big)-\nu(\partial_y^2-\tilde{n}^2)u_n+(\mathrm{i}\tilde{n}p_n,\partial_y p_n)=f_n,\quad y>0,\\
&\mathrm{i}\tilde{n} u_{n,1}+\partial_y u_{n,2}=0,\quad y>0,\\
&u_n|_{y=0}=0,
\end{aligned}\right.
\end{align}
where $u_n=u_n(y)$ is the $n-$th Fourier mode of the velocity $u$. The divergence-free condition and homogeneous Dirichlet condition imply $u_0=(u_{0,1}, 0)$. The main ingredient of this paper is to establish the following estimates(see Proposition \ref{thm:main-linear}):
\begin{enumerate}
\item if $0<|\tilde{n}|\leq \delta_0\nu^{-\f34}$ and $\theta\in(0,\theta_0]$, then
\begin{align*}
|\tilde{n}|^{\f23}\|u_n\|_{L^2}+|\tilde{n}|^{\f13}\nu^{\f12}\|\partial_y u_n\|_{L^2}\leq C\|f_n\|_{L^2}.
\end{align*}
\item if $|\tilde{n}|\geq\delta_0\nu^{-\f34}$ and $\theta\in(0,\theta_0]$, then
\begin{align*}
|\tilde{n}|^{2}\nu\|u_n\|_{L^2}+|\tilde{n}|\nu\|\partial_y u_n\|_{L^2}\leq C\|f_n\|_{L^2}.
\end{align*}
\item if $\theta>\theta_0$ and $\theta\in\Sigma(U,\nu)$, then
 \begin{align*}
|\tilde{n}|\|u_n\|_{L^2}+\nu^{\f12}|\tilde{n}|\|\partial_y u_n\|_{L^2}\leq C\|f_n\|_{L^2}.
\end{align*}
\end{enumerate}
Compared with \cite{GM}, the result in the case of $\theta>\theta_0$ and $\theta\in\Sigma(U,\nu)$ is completely new.

With the help of the above estimates for the linearized system, we can prove nonlinear stability result by using a fixed point argument in the following functional space:
 \begin{align*}
X_{\nu,\varepsilon}=\big\{u:\|u\|_{X_\nu}\leq\varepsilon\nu^{\f12}|\log\nu|^{-\f12}\big\},
\end{align*}
where
\begin{align*}
\|u\|_{X_\nu}=\|u_{0,1}\|_{L^\infty}+\nu^{\f14}\|\partial_y u_{0,1}\|_{L^2}+\sum_{n\neq0}\|u_n\|_{L^\infty}+\nu^{-\f14}\|\mathcal{Q}_0u\|_{L^2}+\nu^{\f14}\|\nabla\mathcal{Q}_0 u\|_{L^2}.
\end{align*}
See section 5.3 for the details. \smallskip

Now we introduce main ideas of the proof for the estimates (1)-(3). \smallskip

For the case of $\tilde{n}\geq C\nu^{-\f34}$, the diffusion part in \eqref{eq:per-n-mode} can control the other terms.
Thus, the desired estimate can be proved via standard energy method to \eqref{eq:per-n-mode}.

For the case of  $0<\tilde{n}\leq C\nu^{-\f34}$, the situation is much more complicated.
As in \cite{GM}, we need to reformulate the system in terms of the stream function $\phi_n$ and make a change of variable.
Recall that 
\begin{align*}
u_{n,1}=\partial_y\phi_n,\quad u_{n,2}=-\mathrm{i}\tilde{n}\phi_n.
\end{align*}
Hence, the stream function $\phi_n$ satisfies the following equation
\begin{align*}
\left\{\begin{aligned}
&-\mathrm{i}\tilde{n}U\big(\frac{y}{\sqrt{\nu}}\big)(\partial_y^2-\tilde{n}^2)\phi_n+\mathrm{i}\tilde{n}\partial_y^2U\big(\frac{y}{\sqrt{\nu}}\big)\phi_n+\nu(\partial_y^2-\tilde{n}^2)^2\phi_n=\mathrm{i}\tilde{n}f_{n,2}-\partial_y f_{n,1},\quad y>0,\\
&\phi_n|_{y=0}=\partial_y\phi_n|_{y=0}=0.
\end{aligned}\right.
\end{align*}
Now we rescale the variable $y$ to $Y=\frac{y}{\sqrt{\nu}}$ and introduce 
\begin{align*}
\phi(Y)={\nu^{-\f12}}\phi_n(y),\quad f(Y)=\nu^{\f12}f_n(y).
\end{align*}
Then $\phi(Y)$ satisfies the Orr-Sommerfeld equation
\begin{align}\label{eq:OS}\left\{\begin{aligned}
&U(\partial_Y^2-\alpha^2)\phi-U''\phi+\mathrm{i}\varepsilon (\partial_Y^2-\alpha^2)^2\phi=-f_2-\dfrac{\mathrm{i}}{\alpha}\partial_Yf_1,\quad Y>0,\\
&\phi|_{Y=0}=\partial_Y\phi|_{Y=0}=0,
\end{aligned}\right.
\end{align}
where
\beno
\alpha=\tilde{n}\sqrt{\nu},\quad \varepsilon=\tilde{n}^{-1}=\frac{\theta}{n}.
\eeno

To solve \eqref{eq:OS}, we first make the following important decomposition to the solution $\phi=\phi_0+\phi_1$ with
$\phi_0$ and $\phi_1$ solving
\begin{align}\label{eq:Orr-phi0-s1}\left\{\begin{aligned}
&\partial_Y(U\partial_Y\phi_0)-\alpha^2U\phi_0+\mathrm{i}\varepsilon (\partial_Y^2-\alpha^2)^2\phi_0=-f_2-\dfrac{\mathrm{i}}{\alpha}\partial_Yf_1,\\
&\phi_0|_{Y=0}=\partial_Y\phi_0|_{Y=0}=0,
\end{aligned}\right.
\end{align}
and
\begin{align}\label{eq:Orr-phi1-s1}\left\{\begin{aligned}
&U(\partial_Y^2-\alpha^2)\phi_1+\mathrm{i}\varepsilon (\partial_Y^2-\alpha^2)^2\phi_1-U''\phi_1=\partial_Y(U'\phi_0),\\
&\phi_1|_{Y=0}=\partial_Y\phi_1|_{Y=0}=0.
\end{aligned}\right.
\end{align}
For the system \eqref{eq:Orr-phi0-s1}, we can obtain our estimate by a direct energy method due to good divergence structure of main part
(see Lemma \ref{prop:phi0}). While, for the system \eqref{eq:Orr-phi1-s1}, the source term $\partial_Y(U'\phi_0)$ has a good decay in $Y$ and takes the divergence structure. To solve \eqref{eq:Orr-phi1-s1}, we consider the following two cases: \smallskip

In the case when $\varepsilon$ is small,  we first solve the following system with artificial boundary conditions:
\begin{align}\label{eq:Orr-som-art-s1}\left\{\begin{aligned}
&U(\partial_Y^2-\alpha^2)\varphi-U''\varphi+\mathrm{i} \varepsilon(\partial_Y^2-\alpha^2)^2\varphi=f,\\
&w=(\partial_Y^2-\alpha^2)\varphi,\quad \varphi|_{Y=0}=\partial_Yw|_{Y=0}=0.
\end{aligned}\right.
\end{align}
In this case, the diffusion term  $\mathrm{i} \varepsilon(\partial_Y^2-\alpha^2)^2\varphi$ could be viewed as a perturbation
in some sense. More precisely, we can close our estimates by solving the Airy equation
\begin{align*}
U w+\mathrm{i}\varepsilon(\partial_Y^2-\alpha^2)w=U''\varphi+f,\quad \partial_Y w|_{Y=0}=0,
\end{align*}
and the Rayleigh equation
\begin{align*}
U(\partial_Y^2-\alpha^2)\varphi-U''\varphi=f-\mathrm{i} \varepsilon(\partial_Y^2-\alpha^2)^2\varphi,\quad \varphi|_{Y=0}=0.
\end{align*}
Thanks to good boundary condition on $w$ and our structure assumption on $U$, we can obtain our desired estimates by a direct energy method and using Rayleigh's trick. See section 2 for the details. In order to match the boundary condition,  we need to construct the boundary layer corrector via solving the system
\begin{align}\nonumber\left\{\begin{aligned}
&\mathrm{i}\varepsilon(\partial_Y^2-\alpha^2)W_b+UW_b-U''\Phi_b=0,\\
&(\partial_Y^2-\alpha^2)\Phi_b=W_b,\\
&\Phi_b|_{Y=0}=0,\ \partial_Y\Phi_b|_{Y=0}=1.
\end{aligned}\right.
\end{align}
For $1\leq|\alpha|<\infty$, the boundary corrector $W_b$ could be chosen as a perturbation of the Airy function.  For $0<|\alpha|\leq1$, we need to use the slow mode and fast mode to construct the boundary layer corrector as in \cite{GM}.

\smallskip

In the case when $\varepsilon$ is not small, the diffusion term can not be viewed as a perturbation. So, the above argument or Rayleigh-Airy iteration
does not work. In this case, we develop the compactness method to solve the system
\begin{align*}\left\{\begin{aligned}
&U(\partial_Y^2-\alpha^2)\varphi-U''\varphi+\mathrm{i} \varepsilon(\partial_Y^2-\alpha^2)^2\varphi=f,\\
&w=(\partial_Y^2-\alpha^2)\varphi,\quad \varphi|_{Y=0}=\partial_Y\varphi|_{Y=0}=0.
\end{aligned}\right.
\end{align*}
More precisely, when $\varepsilon\ge c$ and $|\al|\le C$,  we can use the compactness method to establish the following estimate:
 \begin{align*}
     & \|w\|_{L^2}+\|(\partial_Y\varphi,\alpha \varphi)\|_{L^2}\leq C\|(1+Y)^2f\|_{L^2}.
  \end{align*}
To obtain strong convergence of the sequence, we also need to establish the weighted estimates such as
\begin{align*}
     &\big\|\sqrt{Y}(\partial_Y\varphi,\alpha\varphi)\big\|_{L^2}^2\leq C\|(\partial_Y\varphi,\alpha\varphi)\|_{L^2}\big( \varepsilon\|(1+Y)(\partial_Yw,\al w)\|_{L^2}+\|(1+Y)^2f\|_{L^2}\big),\\
     &\varepsilon^{\f23}\|Y(\partial_Yw,\alpha w)\|_{L^2}+ \varepsilon^{\f13}\|Yw\|_{L^2}\leq C\big(\|\partial_Y\varphi\|_{L^2}+\|Yf\|_{L^2}+\varepsilon\|\partial_Yw\|_{L^2} \big).
  \end{align*}
To arrive a contradiction, we need to assume that the homogeneous system
  \begin{align}\nonumber\left\{\begin{aligned}
  &\mathrm{i}\varepsilon(\partial_Y^2-\alpha^2)w+Uw-U''\varphi=0,\\
  &(\partial_Y^2-\alpha^2)\varphi=w,\quad
  \partial_Y\varphi|_{Y=0}=\varphi|_{Y=0}=0,
  \end{aligned}\right.
  \end{align}
has only trivial solution $\varphi=0$. This assumption exactly corresponds to our spectral condition $0\notin \Sigma(U,\nu)$.
See section 4.2 for the details.

\section{Orr-sommerfeld equation with artificial boundary condition}
In this section, we study the Orr-Sommerfeld equation with artificial boundary condition
\begin{align}\label{eq:Orr-som-art}\left\{\begin{aligned}
&U(\partial_Y^2-\alpha^2)\varphi-U''\varphi+\mathrm{i} \varepsilon(\partial_Y^2-\alpha^2)^2\varphi=f,\quad f(0)=0,\\
&w=(\partial_Y^2-\alpha^2)\varphi,\quad \varphi|_{Y=0}=\partial_Yw|_{Y=0}=0.
\end{aligned}\right.
\end{align}
Compared with the non-slip condition, this kind of  boundary condition allows us to pick enough information from the Airy type structure in the Orr-Sommerfeld equation. Our main idea is as follows. 
\begin{enumerate}
\item We first treat nonlocal term as a perturbation. That is, we write \eqref{eq:Orr-som-art} as the following Airy type equation:
\begin{align*}
Airy[w]:=U w+\mathrm{i}\varepsilon(\partial_Y^2-\alpha^2)w=\tilde{f}:=U''\varphi+f,\quad \partial_Y w|_{Y=0}=0.
\end{align*}
Then we establish the estimates for general $\tilde{f}$ with $(1+Y)^2\tilde{f}\in L^2$.
\item To close the estimates, we need to control $\|\varphi\|_{L^2}$. For this purpose, we treat $\mathrm{i} \varepsilon(\partial_Y^2-\alpha^2)^2\varphi$ as a perturbation. That is, we write \eqref{eq:Orr-som-art}  as the following Rayleigh type equation:
\begin{align*}
Ray[\varphi]:=U(\partial_Y^2-\alpha^2)\varphi-U''\varphi=\bar{f}:=f-\mathrm{i} \varepsilon(\partial_Y^2-\alpha^2)^2\varphi,\quad \varphi|_{Y=0}=0,
\end{align*}
which brings us an useful estimate of $\|\varphi\|_{L^2}$.
\end{enumerate}

Now we state our main result of this section.

\begin{proposition}\label{prop:Orr-som-art}
 Assume that $U$ satisfies  the structure assumption \eqref{ass:U}. Then there exists a small positive number $c_1>0$ such that if $0<\varepsilon\leq c_1$ and $\alpha\neq0$, then for any $f\in H^1_{0}(\mathbb{R}_+)$ with $(1+Y)^2f\in L^2(\mathbb{R}_+)$, there exists a unique solution $\varphi\in H^4(\mathbb{R}_+)\cap H^1_0(\mathbb{R}_+)$ to \eqref{eq:Orr-som-art} satisfying
 \begin{align*}
    &\varepsilon^{\f13}\|(1+Y)w\|_{L^2}+ \varepsilon^{\f23}\|(1+Y)(\partial_Yw,\alpha w)\|_{L^2}+ \varepsilon\|(1+Y)(\partial_Y^2-\alpha^2)w\|_{L^2}\\
     &\quad+\|(\partial_Y\varphi,\alpha\varphi)\|_{L^2}\leq C\left(\|(1+Y)^2f\|_{L^2}+\dfrac{1}{|\alpha|^2}\bigg| \int_{0}^{+\infty}f\mathrm{d}Y\bigg|\right).
  \end{align*}
\end{proposition}

\subsection{Estimates for the Airy equation}
We consider the following system with different boundary conditions
\begin{align}\label{eq:toy0}\left\{\begin{aligned}
&U(\partial_Y^2-\alpha^2)\varphi+\mathrm{i}\varepsilon(\partial_Y^2 -\alpha^2)^2\varphi=f,\\
&w=(\partial_Y^2-\alpha^2)\varphi,\quad\varphi|_{Y=0}=w|_{Y=0}=0,
\end{aligned}
\right.\end{align}
and
\begin{align}\label{eq:toy0-1}\left\{\begin{aligned}
&U(\partial_Y^2-\alpha^2)\varphi+\mathrm{i}\varepsilon(\partial_Y^2 -\alpha^2)^2\varphi=f,\\
&w=(\partial_Y^2-\alpha^2)\varphi,\quad\varphi|_{Y=0}=\partial_Yw|_{Y=0}=0.
\end{aligned}
\right.\end{align}

These two boundary conditions are compatible with the dissipative structure involved in these two systems, which allows us to obtain the following estimates for both systems.

\begin{lemma}\label{prop:nonlocal}
Let $0<\varepsilon\leq 1$. Let $\varphi$ be the solution to \eqref{eq:toy0} or \eqref{eq:toy0-1} with  $f\in L^2(\mathbb{R}_+)$. Then we have\\
 $\bullet$ If $f\in L^2(\mathbb{R}_+)$, then it holds that
\begin{align*}
     &\|Uw\|_{L^2}+ \varepsilon^{\f16}\big\|\sqrt{U}w\big\|_{L^2}+\varepsilon^{\f13} \|w\|_{L^2} +\varepsilon^{\f23}\|(\partial_Yw,\alpha w)\|_{L^2}+\varepsilon\|(\partial_Y^2-\alpha^2)w\|_{L^2}\leq C\|f\|_{L^2}.
\end{align*}
 $\bullet$ If $(1+Y)f\in L^2(\mathbb{R}_+)$, then it holds that
\begin{align*}
     & \|(\partial_Y\varphi,\alpha\varphi)\|_{L^2}+ \varepsilon^{\f13}\|Yw\|_{L^2}+\varepsilon^{\f16} \big\|Y\sqrt{U}w\big\|_{L^2}+\varepsilon^{\f23} \|Y(\partial_Yw,\alpha w)\|_{L^2}\\
     &\qquad+\varepsilon\|Y(\partial_Y^2-\alpha^2)w\|_{L^2}\leq C\|(1+Y)f\|_{L^2}.
  \end{align*}
  {
 $\bullet$ If $(1+Y)f\in L^2(\mathbb{R}_+)$. If $f$ is replaced by $\partial_Y f$ or $f/Y$, then the solution $\varphi$ to \eqref{eq:toy0} satisfies
 \begin{align*}
 	\varepsilon^{\f12}\|\sqrt{U}w\|_{L^2}+\varepsilon^{\f23}\|w\|_{L^2}+\varepsilon\|(\partial_Y w,\alpha w)\|_{L^2}\leq C\|f\|_{L^2}.
 \end{align*}}
\end{lemma}
\begin{remark}
The existence of \eqref{eq:toy0} and \eqref{eq:toy0-1} can be obtained by the method of continuity. In fact, after replacing $\mathrm{i}\varepsilon(\partial_Y^2 -\alpha^2)^2\varphi$ by $\mathrm{i}\varepsilon(\partial_Y^2 -\alpha^2)^2\varphi-il(\partial_Y^2-\alpha^2)\varphi$ with $l$ large enough, one can obtain the existence of the solution, which along with the fact that the above estimates hold for any $l>0$ implies the existence of \eqref{eq:toy0} and \eqref{eq:toy0-1}.
\end{remark}
{
\begin{remark}\label{re:airy-part}
Without the assumption $\varphi|_{Y=0}=0$, we can also obtain the following estimate
	\begin{align*}
     &\|Uw\|_{L^2}+ \varepsilon^{\f16}\big\|\sqrt{U}w\big\|_{L^2}+\varepsilon^{\f13} \|w\|_{L^2} +\varepsilon^{\f23}\|(\partial_Yw,\alpha w)\|_{L^2}+\varepsilon\|(\partial_Y^2-\alpha^2)w\|_{L^2}\leq C\|f\|_{L^2}.
\end{align*}
\end{remark}

}

\begin{proof}
We first consider the case  of $f\in L^2(\mathbb{R}_+)$. Let $\varphi$ be the solution to \eqref{eq:toy0} or \eqref{eq:toy0-1}. Taking inner product with $w$, we have
\begin{align*}
\langle f,w\rangle_{L^2(\mathbb{R}_+)}=&\langle  Uw+\mathrm{i}\varepsilon(\partial_Y^2-\alpha^2)w ,w\rangle_{L^2(\mathbb{R}_+)}\\
=&\big\|\sqrt{U}w\big\|_{L^2}^2-\mathrm{i} \varepsilon\|(\partial_Yw,\alpha w)\|_{L^2}^2,
\end{align*}
which implies that
\begin{align}\label{est:nonl-w1}
\big\|\sqrt{U}w\big\|_{L^2}^2+ \varepsilon\|(\partial_Yw,\alpha w)\|_{L^2}^2\leq
     C\|f\|_{L^2}\|w\|_{L^2}.
\end{align}
By the structure assumption \eqref{ass:U} on $U$, there exists $Y_0>0$ such that
   \begin{align}\label{est:U-pro}
      & U(Y)\geq C^{-1}Y,\quad \text{if}\quad Y\leq Y_0;\qquad U(Y)\geq C^{-1},\quad \text{if}\quad Y\geq Y_0,
   \end{align}
then we have the following interpolation
   \begin{align}\label{est:interpolation-1}
   \begin{split}
      \|w\|_{L^2}&\leq \|w\|_{L^2\big([0,\epsilon^{\f13}Y_0]\big)}+ \|w\|_{L^2\big([\varepsilon^{\f13}Y_0,+\infty)\big)}\\
      &\leq \varepsilon^{\f16}Y_0^{\f12}\|w\|_{L^\infty}+
      C\max(\varepsilon^{-\f16}Y_0^{-\f12},1)\|\sqrt{U}w\|_{L^2\big([ \varepsilon^{\f13}Y_0,+\infty )\big)}\\
      &\leq C\varepsilon^{\f16}\|\partial_Yw\|_{L^2}^{\f12}\|w\|_{L^2}^{\f12}+ C\varepsilon^{-\f16}\|\sqrt{U}w\|_{L^2}.
      \end{split}
   \end{align}

   Combining with \eqref{est:nonl-w1} and using the fact that $\varepsilon\leq 1$, we get
   \begin{align*}
      & \|w\|_{L^2}\leq C\varepsilon^{-\f{1}{12}}\|f\|_{L^2}^{\f14}\|w\|_{L^2}^{\f34}+ C\varepsilon^{-\f16}\|f\|_{L^2}^{\f12}\|w\|_{L^2}^{\f12},
   \end{align*}
   which implies
   \begin{align*}
   \|w\|_{L^2}\leq C\varepsilon^{-\f13}\|f\|_{L^2}.
   \end{align*}
   Applying \eqref{est:nonl-w1} again, we obtain
   \begin{align}\label{est:nonl-res1}
      &\varepsilon^{\f16}\big\|\sqrt{U}w\big\|_{L^2}+\varepsilon^{\f13} \|w\|_{L^2}+\varepsilon^{\f23}\|(\partial_Yw,\alpha w)\|_{L^2}\leq C\|f\|_{L^2}.
   \end{align}

  Multiplying $f$ on both sides of the first equation in \eqref{eq:toy0} or \eqref{eq:toy0-1} and integrating by parts,  we get
  \begin{align}\label{est:nonl-dirL2}
     &\|Uw\|_{L^2}^2+\varepsilon^2\|(\partial_Y^2-\alpha^2)w\|_{L^2}^2 +2\varepsilon\mathbf{Re}\big(\mathrm{i}\langle Uw,(\partial_Y^2-\alpha^2)w\rangle
     \big)= \|f\|_{L^2}^2.
  \end{align}
  Noticing that
  \begin{align*}
     & \big|\mathbf{Re}\big(\mathrm{i}\langle Uw,(\partial_Y^2-\alpha^2)w\rangle
     \big)\big|=\big|\mathbf{Im}\big(\langle U'w,\partial_Yw\rangle
     +\|\sqrt{U}(\partial_Yw,\alpha w)\|_{L^2}^2\big)\big|\\
     &\leq \|U'w\|_{L^2}\|\partial_Yw\|_{L^2}\leq C\|w\|_{L^2}\|\partial_Yw\|_{L^2}\leq C\varepsilon^{-1}\|f\|_{L^2}^2,
  \end{align*}
where we used \eqref{est:nonl-res1} in the last inequality.   Then combining with \eqref{est:nonl-dirL2} and \eqref{est:nonl-res1}, the above inequality implies
  \begin{align}\label{est:nonl-res1-comp}
     & \|Uw\|_{L^2}+ \varepsilon^{\f16}\big\|\sqrt{U}w\big\|_{L^2}+\varepsilon^{\f13} \|w\|_{L^2} +\varepsilon^{\f23}\|(\partial_Yw,\alpha w)\|_{L^2}+\varepsilon\|(\partial_Y^2-\alpha^2)w\|_{L^2}\leq C\|f\|_{L^2}.
  \end{align}
This proves the first statement of this lemma.

Now we turn to prove the second statement. Let $(1+Y)f\in L^2(\mathbb{R}_+)$.
 For this purpose, we introduce the following cut-off function. Let $\chi\geq0$ be a $C^3(\mathbb{R}_+)$ cut-off function, such that
\begin{align*}
   &\chi(Y)=1,\ \text{if}\ Y\in[0,1);\quad\chi(Y)=0,\ \text{if}\ Y\in[2,+\infty),
\end{align*}let $\chi_R(Y)=\chi(Y/R),\ R>1$, then
\begin{align*}
   \chi_R(Y)=1,\ &\text{if}\ Y\in[0,R);\quad\chi(Y)=0,\ \text{if}\ Y\in[2R,+\infty),\\
   &R|\partial_Y\chi_R|+R|\partial_Y^2\chi_R|\leq C,
\end{align*}
which implies that
\begin{align}\label{eq:cut-off bound}
	\ |\partial_Y(Y\chi_R(Y))|+|\partial_Y^2(Y\chi_R(Y))|\leq C.
\end{align}
Hence, $Y\chi_{R} w$ satisfies \eqref{eq:toy0} with the source term $Y\chi_Rf+2\mathrm{i}\varepsilon\partial_Y(Y\chi_R) \partial_Yw+ \mathrm{i}\varepsilon\partial_Y^2(Y\chi_R)w$. That is,
\begin{align*}\left\{\begin{aligned}
   &\mathrm{i}\varepsilon(\partial_Y^2-\alpha^2)(Y\chi_Rw)+ U(Y\chi_Rw)=Y\chi_Rf+2\mathrm{i}\varepsilon\partial_Y(Y\chi_R) \partial_Yw+ \mathrm{i}\varepsilon\partial_Y^2(Y\chi_R)w,\\
   &(Y\chi_Rw)|_{Y=0}=\lim_{Y\rightarrow+\infty}(Y\chi_Rw)=0.
\end{aligned}\right.
\end{align*}
Then by replacing $w$ and $f$ by $Y\chi_R w$ and $Y\chi_Rf+2\mathrm{i}\varepsilon\partial_Y(Y\chi_R) \partial_Yw+ \mathrm{i}\varepsilon\partial_Y^2(Y\chi_R)w$ respectively in \eqref{est:nonl-res1-comp}, we get
\begin{align*}
   &\varepsilon^{\f13}\big\|Y\chi_Rw\big\|_{L^2}+ \varepsilon^{\f16}\big\|\sqrt{U}Y\chi_Rw\big\|_{L^2}+\varepsilon^{\f23} \big\|\big(\partial_Y(Y\chi_Rw),\alpha Y\chi_Rw\big)\big\|_{L^2}
   \\&\quad+\varepsilon \big\|(\partial_Y^2-\alpha^2)\big(Y\chi_Rw\big)\big\|_{L^2}\\
   &\leq C\big\|Y\chi_Rf+2\mathrm{i}\varepsilon\partial_Y(Y\chi_R) \partial_Yw+ \mathrm{i}\varepsilon\partial_Y^2(Y\chi_R)w\big\|_{L^2}\\
   &\leq C\|Y\chi_Rf\|_{L^2}+C\varepsilon\|\partial_Y(Y\chi_R) \|_{L^\infty}\|\partial_Yw\|_{L^2} +C\varepsilon\|\partial_Y^2(Y\chi_R) \|_{L^\infty}\|w\|_{L^2}\\
    &\leq C\|Yf\|_{L^2}+C\varepsilon\|\partial_Yw\|_{L^2} +C\varepsilon\|w\|_{L^2}.
\end{align*}
Then by \eqref{est:nonl-res1-comp} again, we get
\begin{align*}
  &\varepsilon^{\f13}\big\|Y\chi_Rw\big\|_{L^2} +\varepsilon^{\f16}\big\|\sqrt{U}Y\chi_Rw\big\|_{L^2}+\varepsilon^{\f23} \big\|\big(\partial_Y(Y\chi_Rw),\alpha Y\chi_Rw\big)\big\|_{L^2} \\&+\varepsilon \big\|(\partial_Y^2-\alpha^2)\big(Y\chi_Rw\big)\big\|_{L^2}\\
  \leq& C\|Yf\|_{L^2}+C(\varepsilon^{\f13}+\varepsilon^{\f23})\|f\|_{L^2}\leq C\|(1+Y)f\|_{L^2}.
\end{align*}
On the other hand, by \eqref{est:nonl-res1-comp} and \eqref{eq:cut-off bound}, we notice that
\begin{align*}
	\varepsilon^{\f23}\|\partial_Y(Y\chi_R w)\|_{L^2}\geq\varepsilon^{\f23}\|Y\chi_R\partial_Y w\|_{L^2}-C\varepsilon^{\f23} \|w\|_{L^2}\geq \varepsilon^{\f23}\|Y\chi_R\partial_Y w\|_{L^2}-C\varepsilon^{\f13}\|f\|_{L^2},
\end{align*}
and
\begin{align*}
	\varepsilon\|\partial_Y^2(Y\chi_R w)\|_{L^2}\geq \varepsilon\|Y\chi_R\partial_Y^2 w\|_{L^2}-C\varepsilon \|\partial_Y w\|_{L^2}-C\varepsilon\|w\|_{L^2}\geq \varepsilon\|Y\chi_R\partial_Y^2 w\|_{L^2}-C\varepsilon^{\f13}\|f\|_{L^2}.
\end{align*}
Therefore, we obtain
\begin{align*}
	\varepsilon^{\f13}\big\|Y\chi_Rw\big\|_{L^2} +&\varepsilon^{\f16}\big\|\sqrt{U}Y\chi_Rw\big\|_{L^2}+\varepsilon^{\f23} \big\|\big(Y\chi_R\partial_Y w,\alpha Y\chi_Rw\big)\big\|_{L^2}+\varepsilon \big\|Y\chi_R(\partial_Y^2-\alpha^2)w\big\|_{L^2}\\
	&\leq C\|(1+Y)f\|_{L^2}.
\end{align*}
Leting $R\rightarrow+\infty$,  we conclude
\begin{eqnarray}\label{est:nonl-Yw}
  \begin{split}
     &\varepsilon^{\f13}\big\|Yw\big\|_{L^2} +\varepsilon^{\f16}\big\|\sqrt{U}Yw\big\|_{L^2}+\varepsilon^{\f23} \|Y(\partial_Yw,\alpha w)\|_{L^2}\\
   &+\varepsilon \|Y(\partial_Y^2-\alpha^2)w\|_{L^2}\leq C\|(1+Y)f\|_{L^2}.
  \end{split}
\end{eqnarray}

Now we are left with the estimate of $\|(\partial_Y\varphi,\alpha\varphi)\|_{L^2}$.
Taking inner product with $-\varphi/(U+\varepsilon^{\f13})$, we get
  \begin{align}\label{est:nonl1}
     &\big\langle Uw+\mathrm{i}\varepsilon(\partial_Y^2-\alpha^2)w ,-\varphi/(U+\varepsilon^{\f13}) \big\rangle= \big\langle f,-\varphi/(U+\varepsilon^{\f13})\big\rangle.
  \end{align}
We first notice that
\begin{align*}
 \big\langle Uw,-\varphi/(U+\varepsilon^{\f13})\big\rangle =& \langle w,-\varphi\rangle-\varepsilon^{\f13}\big\langle w,-\varphi/(U+\varepsilon^{\f13})\big\rangle\\
     \geq &\|(\partial_Y\varphi,\alpha\varphi)\|_{L^2}^2-\varepsilon^{\f13} \big\|w\bar{\varphi}/U\big\|_{L^1},
\end{align*}
which implies that
\begin{eqnarray}\label{est:nonl2}
\begin{split}
\|(\partial_Y\varphi,\alpha\varphi)\|_{L^2}^2\leq& \varepsilon^{\f13} \big\|w\bar{\varphi}/U\big\|_{L^1}+\big|\big\langle \mathrm{i}\varepsilon(\partial_Y^2-\alpha^2)w ,-\varphi/(U+\varepsilon^{\f13}) \big\rangle\big|\\
&+\big|\big\langle f,\varphi/(U+\varepsilon^{\f13})\big\rangle\big|.
\end{split}
\end{eqnarray}
To deal with $\big\|w\bar{\varphi}/U\big\|_{L^1}$, we first notice that $1/U\leq C(1/Y+\sqrt{U})$ from \eqref{est:U-pro}. Hence, we have \begin{eqnarray}\label{eq:non-L1}
\begin{split}
  \big\|w\bar{\varphi}/U\big\|_{L^1}&\leq C\|w\|_{L^2}\big\|\varphi/Y\big\|_{L^2}+C\big\|\sqrt{U}Yw\big\|_{L^2} \big\|\varphi/Y\big\|_{L^2}\\
      &\leq C\|w\|_{L^2}\|\partial_Y\varphi\|_{L^2}+C \big\|\sqrt{U}Yw\big\|_{L^2}\|\partial_Y\varphi\|_{L^2}.
\end{split}
\end{eqnarray}

On the other hand, by the fact $|U'Y/U|\leq C$, we have
  \begin{align*}
     \big\|\partial_Y\big( \varphi/(U+\varepsilon^{\f13})\big)\big\|_{L^2} &\leq \varepsilon^{-\f13}\|\partial_Y \varphi\|_{L^2}+ \varepsilon^{-\f13}\|\varphi/Y\|_{L^2} \|U'Y/U\|_{L^\infty}\\
     &\leq C\varepsilon^{-\f13}\|\partial_Y\varphi\|_{L^2}.
  \end{align*}
 Since $\varphi/(U+\varepsilon^{\f13})|_{Y=0}=0$, we obtain
  \begin{align}\label{est:nonl3}
     &\big|\big\langle \mathrm{i}\varepsilon(\partial_Y^2-\alpha^2)w ,-\varphi/(U+\varepsilon^{\f13}) \big\rangle\big|\nonumber\\
     &\leq \varepsilon \big|\big\langle\partial_Yw, \partial_Y\big(\varphi/(U+\varepsilon^{\f13})\big)\big\rangle\big| +\varepsilon\big|\langle \alpha^2w,-\varphi/(U+\varepsilon^{\f13})\rangle\big|\nonumber\\
     &\leq \varepsilon\|\partial_Yw\|_{L^2}\big\|\partial_Y\big( \varphi/(U+\varepsilon^{\f13})\big)\big\|_{L^2} +\varepsilon^{\f23}\|\alpha w\|_{L^2}\|\alpha \varphi\|_{L^2}\nonumber\\
     &\leq C\varepsilon^{\f23}\|(\partial_Yw,\alpha w)\|_{L^2}\|(\partial_Y \varphi,\alpha\varphi)\|_{L^2}.
  \end{align}
  By the fact that $\big|Y/\big((1+Y)U\big)\big|\leq C$ and $(1+Y)f\in L^2(\mathbb{R}_+)$, we have
  \begin{eqnarray}\label{est:nonl4}
  \begin{split}
      \big|\big\langle f,\varphi/(U+\varepsilon^{\f13})\big\rangle\big|\leq& \|(1+Y)f\|_{L^2}\big\|Y/\big((1+Y)U\big)\big\|_{L^\infty}\|\varphi/Y\|_{L^2}\\
     \leq&  C\|(1+Y)f\|_{L^2}\|\partial_Y\varphi\|_{L^2}.
  \end{split}
  \end{eqnarray}

Summing up \eqref{est:nonl2}, \eqref{eq:non-L1}, \eqref{est:nonl3} and \eqref{est:nonl4}, we conclude
\begin{eqnarray}\nonumber
   \begin{split}
       \|(\partial_Y\varphi,\alpha\varphi)\|_{L^2}^2\leq& C\varepsilon^{\f23}\|(\partial_Yw,\alpha w)\|_{L^2}\|(\partial_Y\varphi,\alpha\varphi)\|_{L^2}+ C\varepsilon^{\f13}\|w\|_{L^2}\|\partial_Y\varphi\|_{L^2}\\
       &+C\varepsilon^{\f13} \big\|\sqrt{U}Yw\big\|_{L^2}\|\partial_Y\varphi\|_{L^2}
   +C\|(1+Y)f\|_{L^2}\|\partial_Y\varphi\|_{L^2},
   \end{split}
\end{eqnarray}
which along with \eqref{est:nonl-res1} and \eqref{est:nonl-Yw} gives
   \begin{eqnarray}\nonumber
   \begin{split}
       \|(\partial_Y\varphi,\alpha\varphi)\|_{L^2} \leq C\|(1+Y)f\|_{L^2}.
   \end{split}
   \end{eqnarray}
   
  Finally, let us turn to show the last statement of this lemma. Since the case of $\partial_Y f$ and $\frac{f} {Y}$ are similar, here we only consider the case of $\frac{f}{Y}$. We first notice that
   \begin{align*}
   	\big\|\sqrt{U}w\big\|_{L^2}^2+ \varepsilon\|(\partial_Yw,\alpha w)\|_{L^2}^2\leq
     C\|f\|_{L^2}\|w/Y\|_{L^2}\leq C\|f\|_{L^2}\|\partial_Y w\|_{L^2}.
   \end{align*}
   Then we have
   \begin{align*}
    	\varepsilon^\f12 \big\|\sqrt{U}w\big\|_{L^2}+ \varepsilon\|(\partial_Yw,\alpha w)\|_{L^2}\leq C\|f\|_{L^2}.
   \end{align*}
   From \eqref{est:interpolation-1}, we know that
   \begin{align*}
   	\|w\|_{L^2}\leq C\varepsilon^\f16\|\partial_Yw\|_{L^2}^\f12\|w\|_{L^2}^\f12+C\varepsilon^{-\f16}\|\sqrt{U}w\|_{L^2}\leq C\varepsilon^{-\f13}\|f\|_{L^2}^\f12\|w\|_{L^2}^\f12+C\varepsilon^{-\f23}\|f\|_{L^2},
   \end{align*}
   which gives   
   \begin{align*}
   	\|w\|_{L^2}\leq C\varepsilon^{-\f23}\|f\|_{L^2}.
   \end{align*}
   
   This completes the proof of this lemma.
\end{proof}

\subsection{Estimates for the Rayleigh equation}
We consider the Rayleigh equation
\begin{align}\label{eq:Rayleigh}
\left\{\begin{aligned}
   &U(\partial_Y^2-\alpha^2)\varphi-U''\varphi= -\partial_Yf_1+\alpha^2f_2,\\
   &\varphi|_{Y=0},\quad \partial_Yf_1|_{Y=0}=\alpha^2f_2|_{Y=0}.
   \end{aligned}\right.
\end{align}
\begin{lemma}\label{prop:Rayleigh-res}
  Let $(\varphi,f_1,f_2)$ solve \eqref{eq:Rayleigh}. Then we have
  \begin{align*}
     &\|\varphi\|_{L^2} \leq C\|(1+Y)f_1\|_{L^2}+C(1+|\alpha|^2)(|\alpha|^{-2}|f_1(0)|+\|f_2\|_{H^1}).
  \end{align*}
\end{lemma}
\begin{proof}
  We first construct two cut-off functions to correct $\partial_Yf_1(0)$ and $f_1(0)$. Let $\rho_{0}$ be a cut-off function such that
\begin{align*}
   & \int_{0}^{+\infty}\rho_0\mathrm{d}Y=1,\quad \rho_0\geq0,\quad \rho_0(0)=0,\quad \left\|\dfrac{\rho_0}{U}\right\|_{L^2}+\|(1+Y)\sigma[\rho_0]\|_{L^2}\leq C,
\end{align*}
where $\sigma[\rho_0]:=\int_Y^\infty\rho_0(Y_1)dY_1$.
Let $\rho_1$ be a cut-off function such that
\begin{align*}
   & \rho_1(0)=1,\quad \int_{0}^{+\infty}\rho_1\mathrm{d}Y=0,\quad \|\rho_1\|_{H^1}+\|(1+Y)\sigma[\rho_1]\|_{L^2}\leq C.
\end{align*}
Then we decompose $-\partial_Y f_1$ as follow
\begin{align}
-\partial_Y f_1=F_{1,1}+F_{1,2},
\end{align}
where
\begin{align*}
   &F_{1,1}(Y)=-\partial_Yf_1-f_1(0)\rho_0(Y)+\partial_Y f_1(0)\rho_1(Y),\\
   &F_{1,2}= f_1(0)\rho_0(Y)-\partial_Y f_1(0)\rho_1(Y).
\end{align*}
Clearly, $\int_{0}^{+\infty}F_{1,1}\mathrm{d}Y=0$, $F_{1,1}|_{Y=0}=0$.
We decompose $\varphi=\varphi_1+\varphi_2$ as
\begin{align}\label{eq:Ray-phi1}\left\{\begin{aligned}
&U\partial_Y^2\varphi_1-U''\varphi_1=F_{1,1},\\
&\varphi_1|_{Y=0}=0,
\end{aligned}\right.
\end{align}
and
\begin{align}\label{eq:Ray-phi2}\left\{\begin{aligned}
&U(\partial_Y^2-\alpha^2)\varphi_2-U''\varphi_2=\alpha^2U\varphi_1 + F_{1,2}+\alpha^2f_2,\\
&\varphi_2|_{Y=0}=0,
\end{aligned}\right.
\end{align}

We first show the estimate of $\varphi_1$. We notice that equation \eqref{eq:Ray-phi1} can be written as
\begin{align*}
   &\partial_Y\bigg[U^2\partial_Y\big(\dfrac{\varphi_1}{U}\big)\bigg]=F_{1,1},\quad \varphi_1|_{Y=0}=0.
\end{align*}
Since $F_{1,1}$ satisfies the compatible condition of \eqref{eq:Ray-phi1}, the solution $\varphi_1$ to \eqref{eq:Ray-phi1} can be represented by the following formula
\begin{align*}
   & \varphi_1(Y)=U(Y)\int_{Y}^{+\infty}\dfrac{\int_{Y_1}^{+\infty} F_{1,1}(Y_2)\mathrm{d}Y_2}{U^2(Y_1)}\mathrm{d}Y_1= \mathcal{L}[\sigma[F_{1,1}]](Y),
\end{align*}
where the linear operator $\mathcal{L}[\cdot]$ is defined as
\begin{align*}
\mathcal{L}[f](Y)=U(Y )\int_Y^\infty\frac{f(Y_1)}{U^2(Y_1)}dY_1.
\end{align*}
According to Lemma \ref{lem:GM's-hardy} and $F_{1,1}=-\partial_Y f_1-F_{1,2}$, we deduce that
\begin{align*}
   \|\varphi_1\|_{L^2}\leq& \big\|\mathcal{L}[\sigma[F_{1,1}]]\big\|_{L^2} \leq C \big\|(1+Y)\sigma[F_{1,1}]\big\|_{L^2} \\
  \leq& C \big\|(1+Y)\sigma[\partial_Yf_1]\big\|_{L^2}+ C\big\|(1+Y)\sigma[F_{1,2}]\big\|_{L^2}\\
  \leq &C\|(1+Y)f_1\|_{L^2}+C\big\|(1+Y)\sigma[F_{1,2}]\big\|_{L^2}.
\end{align*}
By the definition of $F_{1,2}$ and the condition $\partial_Y f_1(0)=\alpha^2f_2(0)$, we have that
\begin{align*}
 \big\|(1+Y)\sigma[F_{1,2}]\big\|_{L^2}& \leq |f_1(0)|\big\|(1+Y)\sigma[\rho_0]\big\|_{L^2}+\alpha^2 |f_2(0)|\big\|(1+Y)\sigma[\rho_1]\big\|_{L^2},
\end{align*}
which along with the fact $\|(1+Y)\sigma[\rho_0]\|_{L^2}+\|(1+Y)\sigma[\rho_1]\|_{L^2}\leq C$ implies that
\begin{align*}
\big\|(1+Y)\sigma[F_{1,2}]\big\|_{L^2}\leq C(|f_1(0)|+|\al^2 f_2(0)|)\leq C(|f_1(0)|+\alpha^2\|f_2\|_{H^1}).
\end{align*}
This shows that
\begin{align}\label{est:Ray-phi1-L2}
   &\|\varphi_1\|_{L^2}\leq C\big(\|(1+Y)f_1\|_{L^2}+|f_1(0)|+|\alpha|^2\|f_2\|_{H^1} \big).
\end{align}

Now we turn to consider the estimate of $\varphi_2$. We also rewrite the equation as
\begin{align*}\left\{\begin{aligned}
&\partial_Y\bigg[U^2\partial_Y\big(\dfrac{\varphi_2}{U}\big)\bigg] -\alpha^2U\varphi_2=\alpha^2U\varphi_1 + F_{1,2}+\alpha^2f_2,\\
&\varphi_2|_{Y=0}=0.
\end{aligned}\right.
\end{align*}
Multiplying both sides of the first equation by $-\varphi_2/U$, we  get by integration by parts that
\begin{align}\label{est:Ray-phi2-inner}
   & \left\|U\partial_Y\big(\dfrac{\varphi_2}{U}\big)\right\|_{L^2}^2+ \alpha^2\|\varphi_2\|_{L^2}^2=-\alpha^2\langle \varphi_1,\varphi_2\rangle -\langle  F_{1,2}+\alpha^2f_2,\varphi_2/U\rangle.
\end{align}
For the second term on the right side, we have
\begin{align}\label{est:Ray-phi2-inner1}
   &\langle  F_{1,2}+\alpha^2f_2,\varphi_2/U\rangle= f_1(0)\langle \rho_0/U,\varphi_2\rangle +\big\langle (-\partial_Y f_1(0)\rho_1+\alpha^2f_2)/U,\varphi_2\big\rangle.
\end{align}
From the definition of $\rho_0$, we infer that
\begin{align}\label{est:Ray-phi2-inner2}
   & \big|f_1(0)\langle \rho_0/U,\varphi_2\rangle\big|\leq |f_1(0)|\left\|\rho_0/U\right\|_{L^2}\|\varphi_2\|_{L^2}\leq C |f_1(0)|\|\varphi_2\|_{L^2}.
\end{align}
According to the boundary condition $(-\partial_Y f_1(0)\rho_1+\alpha^2f_2)|_{Y=0}=0$ and the fact that $U\geq C^{-1}Y/(1+Y)$, we get by  Hardy's inequality that
\begin{align*}
   &\big|\big\langle (-\partial_Y f_1(0)\rho_1+\alpha^2f_2)/U,\varphi_2\big\rangle\big| \leq C\big\| \big((-\partial_Y f_1(0)\rho_1+\alpha^2f_2)(1+Y)\big)/Y\big\|_{L^2}\|\varphi_2\|_{L^2}\\
   &\leq C\|-\partial_Y f_1(0)\rho_1+\alpha^2f_2\|_{H^1}\|\varphi_{2}\|_{L^2}\leq C\big(|f_2(0)|+\|f_2\|_{H^1}\big)\alpha^2\|\varphi_2\|_{L^2}\\
   &\leq C\alpha^2\|f_2\|_{H^1}\|\varphi_2\|_{L^2}.
\end{align*}
Summing up with \eqref{est:Ray-phi2-inner}, \eqref{est:Ray-phi2-inner1} and \eqref{est:Ray-phi2-inner2}, we conclude that
\begin{align*}
   &  \left\|U\partial_Y\big(\dfrac{\varphi_2}{U}\big)\right\|_{L^2}^2+ \alpha^2\|\varphi_2\|_{L^2}^2\leq \alpha^2\|\varphi_1\|_{L^2}\|\varphi_{2}\|_{L^2}+ C|f_1(0)|\|\varphi_2\|_{L^2}+ C\alpha^2\|f_2\|_{H^1}\|\varphi_2\|_{L^2}.
\end{align*}
Therefore, we obtain
\begin{align*}
   & \|\varphi_2\|_{L^2}\leq C\big(\|\varphi_1\|_{L^2}+|\alpha|^{-2}|f_1(0)|+\|f_2\|_{H^1}\big),
\end{align*}
which along with \eqref{est:Ray-phi1-L2} gives
\begin{align*}
   &\|\varphi\|_{L^2}\leq \|\varphi_1\|_{L^2}+\|\varphi_2\|_{L^2}\\
   &\leq C\|(1+Y)f_1\|_{L^2}+C(1+|\alpha|^2)(|\alpha|^{-2}|f_1(0)|+\|f_2\|_{H^1})
\end{align*}

This completes the proof of the lemma.
\end{proof}

\subsection{Proof of Proposition \ref{prop:Orr-som-art}}
In this part, we prove Proposition \ref{prop:Orr-som-art}. In fact, Proposition \ref{prop:Orr-som-art} is a consequence of the following two lemmas, which give the estimates of the solution $\varphi$ to \eqref{eq:Orr-som-art} for small $|\alpha|$ and large $|\alpha|$ respectively.

\begin{lemma}\label{prop:Orr-som-art-alpha-small}
 Let $(\varphi,w,f)$ solve \eqref{eq:Orr-som-art}. Then for any fixed $M\geq 0$, there exist $c_0^{(0)}=c_0^{(0)}(M)>0$ such that for any $|\alpha|\leq M$, $0<\varepsilon\leq c_0^{(0)}$, there holds
 \begin{align*}
    &\varepsilon^{\f13}\|(1+Y)w\|_{L^2}+ \varepsilon^{\f23}\|(1+Y)(\partial_Yw,\alpha w)\|_{L^2}+ \varepsilon\|(1+Y)(\partial_Y^2-\alpha^2)w\|_{L^2}\\
     &\quad +\|(\partial_Y\varphi,\alpha\varphi)\|_{L^2}+\|\varphi\|_{L^2}\leq C\left(\|(1+Y)^2f\|_{L^2}+\dfrac{1+|\alpha|^2}{|\alpha|^2}\bigg| \int_{0}^{+\infty}f\mathrm{d}Y\bigg|\right).
  \end{align*}
\end{lemma}

\begin{proof}
We first view the nonlocal term $U''\varphi$ as a perturbation. Hence,  we rewrite the equation as
  \begin{align*}\left\{\begin{aligned}
     &U(\partial_Y^2-\alpha^2)\varphi+\mathrm{i} \varepsilon(\partial_Y^2-\alpha^2)^2\varphi=f+U''\varphi:=g,\\
     &w=(\partial_Y^2-\alpha^2)\varphi,\quad \varphi|_{Y=0}=\partial_Yw|_{Y=0}=0.
     \end{aligned}\right.
  \end{align*}
By the fact that $g(0)=0,$ and $|U''|\leq C(1+Y)^{-3}$, we obtain
  \begin{align}
     &\|(1+Y)g\|_{L^2}\leq \|(1+Y)f\|_{L^2}+C\|\varphi\|_{L^2}\label{est:Orr-som-art-gL2}.
  \end{align}
  By Lemma \ref{prop:nonlocal} and \eqref{est:Orr-som-art-gL2}, we get
  \begin{eqnarray}\label{est:Orr-som-art-wH1}
  \begin{split}
    &\varepsilon^{\f13}\|(1+Y)w\|_{L^2}+ \varepsilon^{\f23}\|(1+Y)(\partial_Yw,\alpha w)\|_{L^2}+ \varepsilon\|(1+Y)(\partial_Y^2-\alpha^2)w\|_{L^2}\\
    &+\|(\partial_Y\varphi,\alpha\varphi)\|_{L^2}\leq C\big( \|(1+Y)f\|_{L^2}+\|\varphi\|_{L^2}\big).
   \end{split}
  \end{eqnarray}

Next we view $\mathrm{i}\varepsilon(\partial_Y^2-\alpha^2)^2\varphi$ as a perturbation. So, we rewrite the equation as
  \begin{align*}\left\{\begin{aligned}
     &U(\partial_Y^2-\alpha^2)\varphi-U''\varphi=-\partial_Y\left(\sigma[f] -\mathrm{i}\varepsilon\partial_Yw\right)-\mathrm{i}\varepsilon\alpha^2 w,\\
     &w=(\partial_Y^2-\alpha^2)\varphi,\quad \varphi|_{Y=0}=\partial_Yw|_{Y=0}=0,
     \end{aligned}\right.
  \end{align*}
  here $\sigma[f](Y)=\int_{Y}^{+\infty}f(Y_1)\mathrm{d}Y_1.$
  Then by Lemma \ref{prop:Rayleigh-res} and $\big(\sigma[f] -\mathrm{i}\varepsilon\partial_Yw)|_{Y=0}=\sigma[f](0)$, we get
  \begin{align*}
     &\|\varphi\|_{L^2}\leq C\left(\big\|(1+Y)\big(\sigma[f] -\mathrm{i}\varepsilon\partial_Yw)\big\|_{L^2}+(1+ |\alpha|^{-2})|\sigma[f](0)|+\varepsilon(1+|\alpha|^2)\|w\|_{H^1}\right).
  \end{align*}
By the definition of $\sigma[\cdot]$ and Lemma \ref{lem:GM's-hardy}, we have
  \begin{align*}
     &\big\|(1+Y)\sigma[f]\big\|_{L^2}\leq C\|(1+Y)^2f\|_{L^2},\quad |\sigma[f](0)|=\bigg|\int_{0}^{+\infty}f\mathrm{d}Y\bigg|.
  \end{align*}
Hence, we conclude that
  \begin{align*}
     &\|\varphi\|_{L^2}\leq C\left(\|(1+Y)^2f\|_{L^2}+ \dfrac{1+|\alpha|^2}{|\alpha|^2}\bigg|\int_{0}^{+\infty}f\mathrm{d}Y\bigg| +\varepsilon(1+M^2)\|(1+Y)w\|_{H^1}\right),
  \end{align*}
which along with \eqref{est:Orr-som-art-wH1} implies that
  \begin{align*}
     &\varepsilon^{\f13}\|(1+Y)w\|_{L^2}+ \varepsilon^{\f23}\|(1+Y)(\partial_Yw,\alpha w)\|_{L^2}+ \varepsilon\|(1+Y)(\partial_Y^2-\alpha^2)w\|_{L^2}\\
    &+\|(\partial_Y\varphi,\alpha\varphi)\|_{L^2}+\|\varphi\|_{L^2}\\
    \leq& C\left( \|(1+Y)^2f\|_{L^2}+ (1+|\alpha|^{-2})\bigg|\int_{0}^{+\infty}f\mathrm{d}Y\bigg| +\varepsilon(1+M^2)\|(1+Y)w\|_{H^1}\right).
  \end{align*}
  Then we obtain
  \begin{align*}
     &\varepsilon^{\f13}\big(1-C(1+M^2)\varepsilon^{\f23}\big)\|(1+Y)w\|_{L^2}+ \varepsilon^{\f23}\big(1-C(1+M^2)\varepsilon^{\f13}\big)\|(1+Y)(\partial_Yw,\alpha w)\|_{L^2}\\
    &+ \varepsilon\|(1+Y)(\partial_Y^2-\alpha^2)w\|_{L^2}+\|(\partial_Y\varphi,\alpha\varphi)\|_{L^2}+\|\varphi\|_{L^2}\\
    \leq& C\left( \|(1+Y)^2f\|_{L^2}+ \dfrac{1+|\alpha|^2}{|\alpha|^2}\bigg|\int_{0}^{+\infty}f\mathrm{d}Y\bigg| \right).
  \end{align*}
  Choosing $c_0^{(0)}$ sufficiently small, such that $\big(1-C(1+M^2)(c_0^{(0)})^{\f23}\big)\geq 1/2$ and$ \big(1-C(1+M^2)(c_0^{(0)})^{\f13}\big)\geq 1/2$, we  arrive at
    \begin{align*}
    &\varepsilon^{\f13}\|(1+Y)w\|_{L^2}+ \varepsilon^{\f23}\|(1+Y)(\partial_Yw,\alpha w)\|_{L^2}+ \varepsilon\|(1+Y)(\partial_Y^2-\alpha^2)w\|_{L^2}\\
     & \quad+\|(\partial_Y\varphi,\alpha\varphi)\|_{L^2}+\|\varphi\|_{L^2}\leq C\left(\|(1+Y)^2f\|_{L^2}+\dfrac{1+|\alpha|^2}{|\alpha|^2}\bigg| \int_{0}^{+\infty}f\mathrm{d}Y\bigg|\right).
  \end{align*}
  
The proof is completed. 
\end{proof}

The following lemma gives the estimate for $|\alpha|$ sufficiently large, where we can regard the nonlocal term as a perturbation term.

\begin{lemma}\label{prop:Orr-som-art-alpha-large}
 Let $(\varphi,w,f)$ solve \eqref{eq:Orr-som-art}. There exist $M_0>0$ such that if $|\alpha|\geq M_0,\ 0\leq \varepsilon\leq 1$, then we have
 \begin{align*}
    &\varepsilon^{\f13}\|(1+Y)w\|_{L^2}+ \varepsilon^{\f23}\|(1+Y)(\partial_Yw,\alpha w)\|_{L^2}+ \varepsilon\|(1+Y)(\partial_Y^2-\alpha^2)w\|_{L^2}\\
     & \quad+\|(\partial_Y\varphi,\alpha\varphi)\|_{L^2}\leq C\|(1+Y)f\|_{L^2}.
  \end{align*}
\end{lemma}

\begin{proof}
  We rewrite the equation as
  \begin{align*}
   U(\partial_Y^2-\alpha^2)\varphi+\mathrm{i} \varepsilon(\partial_Y^2-\alpha^2)^2\varphi=f+U''\varphi=g.
  \end{align*}
 It follows from Lemma \ref{prop:nonlocal} that
  \begin{align*}
     &\varepsilon^{\f13}\|(1+Y)w\|_{L^2}+ \varepsilon^{\f23}\|(1+Y)(\partial_Yw,\alpha w)\|_{L^2} +\varepsilon\|(1+Y)(\partial_Y^2-\alpha^2)w\|_{L^2}+\|(\partial_Y\varphi,\alpha \varphi)\|_{L^2}\\
     &\leq C\|(1+Y)g\|_{L^2}\leq C\|(1+Y)f\|_{L^2}+C\|(1+Y)U''\varphi\|_{L^2}\\
     &\leq C\|(1+Y)f\|_{L^2}+C|\alpha|^{-1}\|\alpha\varphi\|_{L^2}\leq C\|(1+Y)f\|_{L^2}+CM_0^{-1}\|\alpha\varphi\|_{L^2},
  \end{align*}
which gives our result by choosing $M_0$ sufficiently large so that $CM_0^{-1}\leq 1/2$.
\end{proof}

Now we prove Proposition \ref{eq:Orr-som-art}.

\begin{proof}[Proof of Proposition \ref{eq:Orr-som-art}]
  Let $M_0$ be the constant in Lemma \ref{prop:Orr-som-art-alpha-large}. Then we take $M=M_0$, $c_1=c_0^{(0)}(M)$ in Lemma \ref{prop:Orr-som-art-alpha-small}. It follows from Lemma  \ref{prop:Orr-som-art-alpha-large} and Lemma \ref{prop:Orr-som-art-alpha-small} that
   \begin{align*}
    &\varepsilon^{\f13}\|(1+Y)w\|_{L^2}+ \varepsilon^{\f23}\|(1+Y)(\partial_Yw,\alpha w)\|_{L^2}+ \varepsilon\|(1+Y)(\partial_Y^2-\alpha^2)w\|_{L^2}\\
     &\quad+\|(\partial_Y\varphi,\alpha\varphi)\|_{L^2}\leq C\left(\|(1+Y)^2f\|_{L^2}+\dfrac{1}{|\alpha|^2}\bigg| \int_{0}^{+\infty}f\mathrm{d}Y\bigg|\right).
  \end{align*}
  The uniqueness is a direct consequence of the above estimate.

Next we prove the existence via the method of continuity. We consider the following system
\begin{align*}\left\{\begin{aligned}
&U(\partial_Y^2-\alpha^2)\varphi_{\lambda}+\mathrm{i} \varepsilon(\partial_Y^2-\alpha^2)^2\varphi_{\lambda}+\lambda(\partial_Y^2-\alpha^2)\varphi_\lambda-  U''\varphi_{\lambda}=f,\quad f(0)=0,\\
&w_\lambda=(\partial_Y^2-\alpha^2)\varphi_{\lambda},\quad \varphi_{\lambda}|_{Y=0}=\partial_Yw_\lambda|_{Y=0}=0,
\end{aligned}\right.
\end{align*}
where $\lambda\in[0,+\infty)$. It is obvious that $\varphi_{\lambda}$ is the solution to \eqref{eq:Orr-som-art} if $\lambda=0$. It is also easy to check that for large enough $\lambda$, there exists a unique solution $\varphi_\lambda\in H^4\cap H^1_0$ to the above system. Hence to prove the existence of \eqref{eq:Orr-som-art}, we only need to obtain a uniform estimates of $\varphi_{\lambda}$ with respected to $\lambda$. In fact, we can prove this by a similar argument in the proof of Lemma \ref{prop:Orr-som-art-alpha-small}, which is guaranteed by the following facts:
\begin{enumerate}
\item if we replace $\mathrm{i}\varepsilon(\partial_Y^2-\alpha^2)^2$ by $\mathrm{i}\varepsilon(\partial_Y^2-\alpha^2)^2+\lambda(\partial_Y^2-\alpha^2)$ with $\lambda\geq 0$, the estimates in Lemma \ref{prop:nonlocal} still hold true and is independent of $\lambda$.
\item if we replace $U(\partial_Y^2-\alpha^2)$ by $(U+\lambda)(\partial_Y^2-\alpha^2)$, the estimates in Lemma \ref{prop:Rayleigh-res} still hold true and is independent of $\lambda$. In fact, we just need to modify \eqref{eq:Ray-phi1} and \eqref{eq:Ray-phi2} to
\begin{align*}
(U+\lambda)\partial_Y^2\varphi_1-U''\varphi_1=F_{1,1},\varphi_1|_{Y=0}=0,
\end{align*}
and
\begin{align*}
(U+\lambda)(\partial_Y^2-\alpha^2)\varphi_2-U''\varphi_2=\alpha^2U\varphi_1 + F_{1,2}+\alpha^2f_2,\varphi_2|_{Y=0}=0.
\end{align*}
Along with that fact that $\mathcal{L}_{\lambda}[\cdot]$ also  satisfies the results in Lemma \ref{lem:GM's-hardy} for any $\lambda>0$, we get this conclusion. Here
\begin{align*}
\mathcal{L}_{\lambda}[f](Y):=(U(Y)+\lambda)\int_Y^\infty\frac{f(Y_1)}{(U(Y_1)+\lambda)^2}dY_1.
\end{align*}
\end{enumerate}
Therefore, we can obtain the existence of \eqref{eq:Orr-som-art}.
\end{proof}

\section{Boundary layer corrector}

This section is devoted to constructing the boundary layer corrector. That is, we need to construct the solution to the following homogeneous system
\begin{align}\label{eq:boundary-corr-parphi}\left\{\begin{aligned}
&\mathrm{i}\varepsilon(\partial_Y^2-\alpha^2)W_b+UW_b-U''\Phi_b=0,\\
&(\partial_Y^2-\alpha^2)\Phi_b=W_b,\\
&\Phi_b|_{Y=0}=0,\ \partial_Y\Phi_b|_{Y=0}=1.
\end{aligned}\right.
\end{align}
We construct the solution $W_b$ by finding two special solution (fast and slow mode) to the homogeneous Orr-Sommerfeld equation to match the boundary condition in \eqref{eq:boundary-corr-parphi}.  The fast mode is a solution to the homogeneous Orr-Sommerfeld equation built around the Airy function, and the slow mode is built around a solution to the Rayleigh equation.

We shall see that for $1\leq|\alpha|<\infty$, the boundary corrector $W_b$ is a perturbation of the Airy function. In this case, the ``stream fucntion'' $\Phi_a$ of Airy function is equipped with zero boundary condition. While, for $0<|\alpha|\leq1$, we choose the fast decay part $\Phi_{a,f}$ of $\Phi_a$ as the ``stream function'' of Airy function. Since $\Phi_{a,f}(0)\neq0$, we need to use the slow mode to correct it.

We always assume $\varepsilon|\alpha|^3\ll1$ throughout this section.
\subsection{Construction of the fast mode}
We construct the fast mode $W$  around the Airy function. We first define the  main part of $W$.
Let $Ai(y)$ be the Airy function defined in the appendix, which satisfies $Ai''(y)-yAi(y)=0$. Let $d=-\mathrm{i}\varepsilon\alpha^2/(U'(0))$ and denote 
\begin{align}\label{formula:Wa-construct}
  &W_a(Y)=\dfrac{Ai\big(\mathrm{e}^{\mathrm{i}\f{\pi}{6}} \big(U'(0)/\varepsilon\big)^{\f13}(Y+d) \big)}{\mathrm{e}^{\mathrm{i}\f{\pi}{6}} \big(U'(0)/\varepsilon\big)^{\f13}Ai'(0)}.
\end{align}
Then we find that 
\begin{align}\label{formula:Wa-eq}
   &\mathrm{i}\varepsilon(\partial_Y^2-\alpha^2)W_a+U'(0)YW_a=0,\quad \partial_YW_a(0)=\dfrac{Ai'\big(\mathrm{e}^{\mathrm{i}\f{\pi}{6}} \big(U'(0)/\varepsilon\big)^{\f13}d \big)}{Ai'(0)}.
\end{align}
We first notice that
\begin{align*}
	|(U'(0)/\varepsilon)^{\f13}d|=U'(0)^{-\f23}\varepsilon^\f23\alpha^2,
\end{align*}
which along with the facts $A_i(0)=\dfrac{1}{3^{2/3}\Gamma(2/3)}$, $Ai'(0)=-\dfrac{1}{3^{1/3}\Gamma(1/3)}$ and the smoothness of $Ai(y)$ implies that if $\varepsilon|\alpha|^3$ is small enough, we have
\begin{align}\label{eq:Ai(ke)-Ai(0)}
	\dfrac{Ai\big(\mathrm{e}^{\mathrm{i}\f{\pi}{6}} \big(U'(0)/\varepsilon\big)^{\f13}d \big)}{Ai(0)}= 1+\mathcal O(\varepsilon^{\f23}\alpha^2 ) \quad\mathrm{and} \quad\dfrac{Ai'\big(\mathrm{e}^{\mathrm{i}\f{\pi}{6}} \big(U'(0)/\varepsilon\big)^{\f13}d \big)}{Ai'(0)}= 1+\mathcal O(\varepsilon^{\f23}\alpha^2 ).
\end{align}
Let $\Phi_a$ solve $(\partial_Y^2-\alpha^2)\Phi_a=W_a,\ \Phi_a|_{Y=0}=0$. By Lemma \ref{lem:ham-bound}, we know that
\begin{align*}
	\Phi_a(Y)=-\alpha^{-1} \mathrm{e}^{-\alpha Y}\int_{0}^{+\infty}W_a(Z)\sinh(\alpha Z)\mathrm{d}Z+\alpha^{-1} \int_{Y}^{+\infty}W_a(Z)\sinh\big(\alpha(Z-Y)\big)\mathrm{d}Z.
\end{align*}
We denote by $\Phi_{a,f}$ the fast decay part of $\Phi_a$, i.e., 
\begin{align*}
  \Phi_{a,f} &= \dfrac{1}{\alpha}\int_{Y}^{+\infty} W_a(Z)\sinh\big(\alpha (Z-Y)\big)\mathrm{d}Z.
\end{align*}

We first establish estimates about $W_a$ and $\Phi_a$ before going further. Thanks to $\mathbf{Im}(d)<0$, we take
 \begin{align*}
    & \kappa= \big(U'(0)/\varepsilon\big)^{\f13},\quad \eta=d,\quad \tilde{A}(Y)=Ai(\mathrm{e}^{\mathrm{i}\f{\pi}{6}}\kappa(Y+\eta)) /Ai(\mathrm{e}^{\mathrm{i}\f{\pi}{6}}\kappa\eta).
 \end{align*}
Then we have
\begin{align}\label{formula:Wa-tildeA}
   W_a(Y)&= \dfrac{Ai(\mathrm{e}^{\mathrm{i}\frac{\pi}{6}}\kappa\eta )}{\mathrm{e}^{\mathrm{i}\f{\pi}{6}} \big(U'(0)/\varepsilon\big)^{\f13}Ai'(0)}\times \dfrac{Ai\big(\mathrm{e}^{\mathrm{i}\f{\pi}{6}} \big(U'(0)/\varepsilon\big)^{\f13}(Y+d) \big)}{Ai(\mathrm{e}^{\mathrm{i}\frac{\pi}{6}}\kappa\eta)}= \dfrac{Ai(\mathrm{e}^{\mathrm{i}\frac{\pi}{6}}\kappa\eta)\tilde{A}(Y)}{\mathrm{e}^{\mathrm{i}\f{\pi}{6}} \big(U'(0)/\varepsilon\big)^{\f13}Ai'(0)},
\end{align}
and
\begin{align*}
	\Phi_{a,f}(Y)=\dfrac{Ai(\mathrm{e}^{\mathrm{i}\frac{\pi}{6}}\kappa\eta)\tilde{\Phi}_f(Y)}{\mathrm{e}^{\mathrm{i}\f{\pi}{6}} \big(U'(0)/\varepsilon\big)^{\f13}Ai'(0)}.
\end{align*}
Then we get by Lemma \ref{lem:Airy-bound} that
\begin{align*}
	&\Phi_{a,f}(0)=\dfrac{e^{\mathrm i\frac{\pi}{3}}\varepsilon^{\f43}Ai(e^{\mathrm i\frac{\pi}{6}}\kappa\eta)}{(U'(0))^{\f43}Ai'(0)}+\mathcal O(\varepsilon^{\f43}\alpha^2),\\
	&\partial_Y\Phi_{a,f}(0)=-\dfrac{\varepsilon^\f23}{3(U'(0))^\f23Ai'(0)}+\mathcal O(\varepsilon\alpha^2),
\end{align*}
which along with \eqref{eq:Ai(ke)-Ai(0)} implies that
\begin{align}\label{eq:phif(0)}
\begin{split}
	&\Phi_{a,f}(0)=\dfrac{e^{\mathrm i\frac{\pi}{3}}\varepsilon^{\f43}Ai(0)}{(U'(0))^{\f43}Ai'(0)}+\mathcal O(\varepsilon^{\f43}\alpha^2),\\
	&\partial_Y\Phi_{a,f}(0)=-\dfrac{\varepsilon^\f23}{3(U'(0))^\f23Ai'(0)}+\mathcal O(\varepsilon\alpha^2).
	\end{split}
\end{align}
We also notice that by \eqref{eq:Ai(ke)-Ai(0)},
\begin{align*}
	|W_a(Y)|=\Big|\dfrac{Ai(0)\tilde{A}(Y)}{\mathrm{e}^{\mathrm{i}\f{\pi}{6}} \big(U'(0)/\varepsilon\big)^{\f13}Ai'(0)}\times\dfrac{Ai(\mathrm{e}^{\mathrm{i}\frac{\pi}{6}}\kappa\eta)}{Ai(0)}\Big|\sim \varepsilon^{\f13}|\tilde{A}(Y)|.
\end{align*}
According to the fact $1\leq 1+|\kappa\eta|\leq C(1+\varepsilon|\alpha|^3)^{\f23}$ and applying Lemma \ref{lem:Airy-w} and \ref{lem:Airy-bound}, we then have
\begin{align}\label{est:Wa-1}
   & \varepsilon^{-\f16}\|W_a\|_{L^2} +
   \varepsilon^{-\f12}\|(\partial_Y\Phi_a,\alpha\Phi_a)\|_{L^2} +\varepsilon^{-\f56}\|Y^2W_a\|_{L^2}+\varepsilon^{-\f32}\|Y^4W_a\|_{L^2}\nonumber\\
   &+ \varepsilon^{-\f23}|\alpha|^{\f12}\|\Phi_a\|_{L^2}+\varepsilon^{-\f23}
   |\alpha|^{\f72}\|Y^3\Phi_a\|_{L^2} \leq C\varepsilon^{\f13},
\end{align}
and
\begin{align}\label{est:tildePhi-1}
  & |\partial_Y\Phi_a(0)|\geq C^{-1}\varepsilon^{\f23}(1+\varepsilon|\alpha|^3)^{-\f13}.
\end{align}
For the fast decay part, we also have
\begin{align}\label{est:Wa-1-fast}
   &
   \varepsilon^{-\f12}\|(\partial_Y\Phi_{a,f},\alpha\Phi_{a,f})\|_{L^2} +\varepsilon^{-\f56}\|\Phi_{a,f}\|_{L^2}+\varepsilon^{-\frac{11}{6}}\|Y^3\Phi_{a,f}\|_{L^2} \leq C\varepsilon^{\f13},
\end{align}
and
\begin{align}\label{est:tildePhi-1-fast}
  & |\partial_Y\Phi_{a,f}(0)|\geq C^{-1}\varepsilon^{\f23}(1+\varepsilon|\alpha|^3)^{-\f13}.
\end{align}

As we mentioned before, we construct the fast mode $(W,\Phi)$ around $(W_a,\Phi_a)$ for the case of $1\leq|\alpha|<\infty$ and $(W_f,\Phi_f)$ around $(W_a,\Phi_{a,f})$ for the case of $0<|\alpha|\leq1$. More precisely, we define $W=W_a+W_e$ with $W_e$ satisfying
\begin{align}\label{formula:Wa-eq}\left\{\begin{aligned}
   &\mathrm{i}\varepsilon(\partial_Y^2-\alpha^2)W_e+UW_e-U''\Phi_e= -\big(U-U'(0)Y\big)W_a+U''\Phi_a,\\
   &(\partial_Y^2-\alpha^2)\Phi_e= W_e,\\
   &\Phi_e|_{Y=0}=\partial_YW_e|_{Y=0}=0,
\end{aligned}\right.
\end{align}
and $W_f=W_a+W_{e,f}$ with $W_{e,f}$ being the solution to
\begin{align}\label{formula:Wa-eq-fast}\left\{\begin{aligned}
   &\mathrm{i}\varepsilon(\partial_Y^2-\alpha^2)W_{e,f}+UW_{e,f}-U''\Phi_{e,f}= -\big(U-U'(0)Y\big)W_a+U''\Phi_{a,f},\\
   &(\partial_Y^2-\alpha^2)\Phi_e= W_{e,f},\\
   &\Phi_e|_{Y=0}=\partial_YW_e|_{Y=0}=0.
\end{aligned}\right.
\end{align}
\begin{lemma}
Let $c_1$ be the positive number in Proposition \ref{prop:Orr-som-art}. Then for any $0<\varepsilon\leq c_1$, the following statements hold.
\begin{itemize}
	\item Let $(W_e,\Phi_e)$ be the solution to \eqref{formula:Wa-eq}. Then for any $1\leq|\alpha|<+\infty$, we have
        \begin{align}\label{est:We-1}
 &\varepsilon^{\f13}\|(1+Y)W_e\|_{L^2}+\varepsilon^{\f23}\|(1+Y)(\partial_YW_e,\alpha W_e)\|_{L^2}+ \|(\partial_Y\Phi_e,\alpha\Phi_e)\|_{L^2}\leq C\varepsilon.
\end{align}
    \item Let $(W_{e,f},\Phi_{e,f})$ be the solution to \eqref{formula:Wa-eq-fast}. Then for any $\alpha>0$, we have
    \begin{align}\label{est:We-1-fast}
 &\varepsilon^{\f13}\|(1+Y)W_{e,f}\|_{L^2} +\varepsilon^{\f23}\|(1+Y)(\partial_YW_{e,f},\alpha W_{e,f})\|_{L^2}+ \|(\partial_Y\Phi_{e,f},\alpha\Phi_{e,f})\|_{L^2}\leq C\varepsilon^{\f76}.
\end{align}
\end{itemize}
\end{lemma}
\begin{remark}
The existence of $(W_e, \Phi_e)$ and $(W_{e,f},\Phi_{e,f})$ are also deduced by Proposition \ref{prop:Orr-som-art}.
\end{remark}
\begin{proof}
We start our proof of \eqref{est:We-1}.
We first notice that the source term in \eqref{formula:Wa-eq} can be written as
\begin{eqnarray}\label{formula:force}
\begin{split}
    &-(U-U'(0)Y)W_a+U''\Phi_a=-\partial_Y\big[(U-U'(0)Y)\partial_Y\Phi_a\big] \\
   &\quad +\partial_Y\big[(U'-U'(0))\Phi_a\big]+\alpha^2(U-U'(0)Y)\Phi_a\\
   &=\partial_YR_1+R_2,
\end{split}
\end{eqnarray}
where
\begin{align}
   &R_1= -(U-U'(0)Y)\partial_Y\Phi_a+(U'-U'(0))\Phi_a,\label{formula:R1}\\
    &R_2=\alpha^2(U-U'(0)Y)\Phi_a.\label{formula:R2}
\end{align}
Moreover, $\lim_{Y\rightarrow+\infty}R_1(Y)=R_1|_{Y=0}=0$. Then
\begin{align*}
   & \dfrac{1}{|\alpha|^2}\left|\int_{0}^{+\infty}(\partial_YR_1+R_2)\mathrm{d}Y\right|= \dfrac{1}{|\alpha|^2}\left|\int_{0}^{+\infty}R_2\mathrm{d}Y\right|\leq C|\alpha|^{-2}\|(1+Y)R_2\|_{L^2}.
\end{align*}
 Then by Proposition \ref{prop:Orr-som-art}, for any $0<\varepsilon\leq c_1$, we get
\begin{align*}
   &\varepsilon^{\f13}\|(1+Y)W_e\|_{L^2}+\varepsilon^{\f23}\|(1+Y)(\partial_YW_e,\alpha W_e)\|_{L^2}+ \|(\partial_Y\Phi_e,\alpha\Phi_e)\|_{L^2}\\
   &\leq C\bigg(\|(1+Y)^2(\partial_YR_1+R_2)\|_{L^2}+ \dfrac{1}{|\alpha|^2}\left|\int_{0}^{+\infty}(\partial_YR_1+R_2)\mathrm{d}Y\right| \bigg)\\
   &\leq C\big(\|(1+Y)^2(\partial_YR_1+R_2)\|_{L^2}+ |\alpha|^{-2}\|(1+Y)R_2\|_{L^2}\big).
\end{align*}
On the other hand, thanks to $|(U-U'(0)Y)W_a|\leq CY^2|W_a|$ , we have
\begin{align*}
  &\|(1+Y)^2(\partial_YR_1+R_2)\|_{L^2}\leq C\big( \|(1+Y)^2Y^2W_a\|_{L^2}+\|\Phi_a\|_{L^2}\big),
\end{align*}
and
\begin{align*}
|\alpha|^{-2}\|(1+Y)R_2\|_{L^2}\leq \|(1+Y)^3\Phi_a\|_{L^2}.
\end{align*}
From \eqref{est:Wa-1} and $\varepsilon\leq 1$, $|\alpha|\geq 1$, we infer that
\begin{align*}
\|(1+Y)^2Y^2W_a\|_{L^2}+\|\Phi_a\|_{L^2}+ \|(1+Y)^3\Phi_a\|_{L^2}\leq C{\color{blue}\varepsilon}.
\end{align*}
Therefore, we obtain
\begin{align*}
\|(1+Y)^2(\partial_YR_1+R_2)\|_{L^2}+ |\alpha|^{-2}\|(1+Y)R_2\|_{L^2}\leq  C\varepsilon.
\end{align*}
Thus, we conclude 
\begin{align*}
   &\varepsilon^{\f13}\|(1+Y)W_e\|_{L^2}+\varepsilon^{\f23}\|(1+Y)(\partial_YW_e,\alpha W_e)\|_{L^2}+ \|(\partial_Y\Phi_e,\alpha\Phi_e)\|_{L^2}\leq  C\varepsilon.
\end{align*}

 Now we turn to the proof of \eqref{est:We-1-fast}. Again by Proposition \ref{prop:Orr-som-art}, for any $0<\varepsilon\leq c_1$ and a similar argument as above, we have 
\begin{align*}
	   &\varepsilon^{\f13}\|(1+Y)W_{e,f}\|_{L^2}+\varepsilon^{\f23}\|(1+Y)(\partial_YW_{e,f},\alpha W_{e,f})\|_{L^2}+ \|(\partial_Y\Phi_{e,f},\alpha\Phi_{e,f})\|_{L^2}\\
	   &\leq C\big(\|(1+Y)^2Y^2W_{a,f}\|_{L^2}+\|\Phi_{a,f}\|_{L^2}+ \|(1+Y)^3\Phi_{a,f}\|_{L^2}\big),
\end{align*}
which along with \eqref{est:Wa-1-fast} implies that
\begin{align*}
	&\varepsilon^{\f13}\|(1+Y)W_{e,f}\|_{L^2} +\varepsilon^{\f23}\|(1+Y)(\partial_YW_{e,f},\alpha W_{e,f})\|_{L^2}+ \|(\partial_Y\Phi_{e,f},\alpha\Phi_{e,f})\|_{L^2}\leq C\varepsilon^{\f76}.
\end{align*}

\end{proof}

\begin{lemma}\label{lem:W-est}
Let $c_1$ be the constant in Proposition \ref{prop:Orr-som-art}. There exist $c_2\in(0,c_1]$, $\delta_{*}\in(0,1],$ such that if $\varepsilon|\alpha|^3\leq \delta_{*},\ 0<\varepsilon \leq c_2$, the following statements hold true:
\begin{enumerate}
	\item For any $1\leq|\alpha|<\infty$, there exists a unique solution $\Phi\in H^4$ to the homogeneous Orr-Sommerfeld equation satisfying 
	   \begin{align*}
     & \varepsilon^{-\f16}\|W\|_{L^2}+ \varepsilon^{-\f12}\|(\partial_Y\Phi,\alpha\Phi)\|_{L^2} +\varepsilon^{-\f13}\|(\partial_Y\Phi,\alpha\Phi)\|_{L^\infty}\leq C\varepsilon^{\f13}.
  \end{align*}
  Moreover, we have
  \begin{align*}
     & |\partial_Y\Phi(0)|\geq C^{-1}\varepsilon^{\f23},\quad\Phi(0)=0.
  \end{align*}
  \item For any $\alpha>0$, there exists a unique solution $\Phi_f\in H^4$ to the homogeneous Orr-Sommerfeld equation satisfying 
   \begin{align*}
     & \varepsilon^{-\f16}\|W_f\|_{L^2}+ \varepsilon^{-\f12}\|(\partial_Y\Phi_f,\alpha\Phi_f)\|_{L^2} +\varepsilon^{-\f13}\|(\partial_Y\Phi_f,\alpha\Phi_f)\|_{L^\infty}\leq C\varepsilon^{\f13}.
  \end{align*}
  Moreover, we have
  \begin{align*}
     \partial_Y\Phi_f(0)=-\dfrac{\varepsilon^\f23}{3(U'(0))^\f23Ai'(0)}+\mathcal O(\varepsilon(1+\alpha^2)) ,\quad \Phi_f(0)=\dfrac{e^{\mathrm i\frac{\pi}{3}}\varepsilon^{\f43}Ai(0)}{(U'(0))^{\f43}Ai'(0)}+\mathcal O(\varepsilon^{\f43}\alpha^2).
  \end{align*}

\end{enumerate}
\end{lemma}

\begin{proof}
We first show the first statement of this lemma.
  By \eqref{est:Wa-1} and \eqref{est:We-1}, if $\varepsilon|\alpha|^3\leq \delta_*\leq 1,\ 0<\varepsilon\leq c_2\leq c_1$, $|\alpha|\geq 1$, we have
  \begin{eqnarray}\label{est:tildeW1}
  \begin{split}
     & \varepsilon^{-\f16}\|W\|_{L^2}+ \varepsilon^{-\f12}\|(\partial_Y\Phi,\alpha\Phi)\|_{L^2}\\
     &\leq C\big( \varepsilon^{-\f16}\|(1+Y)W_a\|_{L^2}+ \varepsilon^{-\f12}\|(\partial_Y\Phi_a,\alpha\Phi_a)\|_{L^2}+ \varepsilon^{-\f16}\|(1+Y)W_e\|_{L^2}\\
     &\quad+ \varepsilon^{-\f12}\|(\partial_Y\Phi_e,\alpha\Phi_e)\|_{L^2}\big)\\
     &\leq C\varepsilon^{\f13}\big(1+\varepsilon^{\f16} \big)\leq C\varepsilon^{\f13}(1+c_2^{\f16})\leq C\varepsilon^{\f13}.
  \end{split}
  \end{eqnarray}
By the interpolation, we get
  \begin{align}\label{est:tildeW2}
     & \|(\partial_Y\Phi,\alpha\Phi)\|_{L^\infty}\leq \|W\|_{L^2}^{\f12} \|(\partial_Y\Phi,\alpha\Phi)\|_{L^2}^{\f12} \leq C\varepsilon^{\f23}.
  \end{align}

Again by \eqref{est:We-1} and the interpolation, we have
\begin{align*}
   & \|\partial_Y\Phi_e\|_{L^\infty}\leq \|W_e\|_{L^2}^{\f12} \|(\partial_Y\Phi_e,\alpha\Phi_e)\|_{L^2}^{\f12}\leq C\varepsilon^{\f56},
\end{align*}
which along with \eqref{est:tildePhi-1} implies that
\begin{align*}
   |\partial_Y\Phi(0)|&\geq |\partial_Y\Phi_a(0)|-\|\partial_Y\Phi_e\|_{L^\infty} \geq C^{-1}\varepsilon^{\f23}(1-C\varepsilon^{\f16}).
\end{align*}
Then taking $c_2$ sufficient small so that $1-Cc_2^{\f16}\geq 1/2$,  we get
\begin{align}\label{est:tildePhi0}
 |\partial_Y\Phi(0)|\geq C^{-1}\varepsilon^{\f23}.
\end{align}

 Now we turn to the proof of the second statement. By a similar argument as \eqref{est:tildeW1}, we have
\begin{align*}
     & \varepsilon^{-\f16}\|W_f\|_{L^2}+ \varepsilon^{-\f12}\|(\partial_Y\Phi_f,\alpha\Phi_f)\|_{L^2} +\varepsilon^{-\f13}\|(\partial_Y\Phi_f,\alpha\Phi_f)\|_{L^\infty}\leq C\varepsilon^{\f13}.
  \end{align*}
  We also notice that
  \begin{align*}
  	|\partial_Y \Phi_{f}(0)-\partial_Y\Phi_{a,f}(0)|=|\Phi_{e,f}(0)|\leq\|\Phi_{e,f}\|_{L^\infty}\leq \|W_{e,f}\|_{L^2}^{\f12}\|\partial_Y\Phi_{e,f}\|_{L^2}^{\f12}\leq C\varepsilon
  \end{align*}
  and $\Phi_{f}(0)=\Phi_{a,f}(0)$ which along with \eqref{eq:phif(0)} implies that
  \begin{align*}
     \partial_Y\Phi_f(0)=-\dfrac{\varepsilon^\f23}{3(U'(0))^\f23Ai'(0)}+\mathcal O(\varepsilon(1+\alpha^2)) ,\quad \Phi_f(0)=\dfrac{e^{\mathrm i\frac{\pi}{3}}\varepsilon^{\f43}Ai(0)}{(U'(0))^{\f43}Ai'(0)}+\mathcal O(\varepsilon^{\f43}\alpha^2).
  \end{align*}
  
\end{proof}

\subsection{Construction of the slow mode}
In this part, we construct a solution to the homogeneous Orr-Sommerfeld equation around a solution $\varphi_{Ray}$ to the homogeneous Rayleigh when $0<\alpha\leq 1$. Let $\varphi=\varphi_0+\varphi_1+\varphi_2$ be the solution to the homogeneous Rayleigh equation constructed in Proposition \ref{Prop:hom-ray}. For $0<\alpha\leq1$, we define $\varphi_{Ray}$ as follow
\begin{align}
	\varphi_{Ray}=\frac{c_E}{\alpha}\varphi,\quad c_E=\frac{\alpha}{\varphi_1(0)}=U'(0)+\mathcal{O}(\alpha).
\end{align}
Then from Proposition \ref{Prop:hom-ray}, we directly have
\begin{lemma}\label{Lem:slow-ray}
	For any $0<\alpha\leq 1$, there exits a solution $\varphi_{Ray}\in H^1(\mathbb R_+)$ to the homogeneous Rayleigh equation satisfying the following properties: $\varphi_{Ray}=\varphi_{Ray,0}+\varphi_{Ray,1}+\varphi_{Ray,2}$ with
	\begin{eqnarray}\label{eq:hom-ray}
		\begin{split}
			&\varphi_{Ray,0}=\frac{c_E}{\alpha}Ue^{-\alpha Y},\quad\varphi_{Ray,1}(0)=1,\\
			&\|\partial_Y\varphi_{Ray,1}\|_{L^2}+\|\varphi_{Ray,1}\|_{L^2}\leq C,\\
			&\|\partial_Y\varphi_{Ray,2}\|_{L^2}+\alpha\|\varphi_{Ray,2}\|_{L^2}\leq C\alpha^\f12.
		\end{split}
	\end{eqnarray}
	In particular, we have
	\begin{align}
		\varphi_{Ray}(0)=1.
	\end{align}
	If $\frac{U''}{U}\in L^2(\mathbb R_+)$ in addition, then $\varphi_{Ray,1},\varphi_{Ray,2}\in H^2(\mathbb R_+)$.
\end{lemma}

We denote $OS=\mathrm i \varepsilon(\partial_Y^2-\alpha^2)^2+U(\partial_Y^2-\alpha^2)-U''$. We observe that
\begin{align*}
	OS[\varphi_{Ray}]=\mathrm i \varepsilon(\partial_Y^2-\alpha^2)^2\varphi_{Ray},
\end{align*}
whose source term contains too much singularity. Hence, we introduce $\psi$ being the solution to the following system:
\begin{align}\label{eq:airy-slow}
	\left\{\begin{aligned}
		&\mathrm i \varepsilon(\partial_Y^2-\alpha^2)\psi+U\psi=\mathrm i \varepsilon(\partial_Y^2-\alpha^2)\varphi_{Ray},\quad Y>0,\\
		&\psi(0)=0.
	\end{aligned}\right.
\end{align}

\begin{lemma}\label{Lem:slow-airy}
	Let $\alpha\in(0,1]$ and $\psi$ be the solution  of \eqref{eq:airy-slow}. Then $\psi=\psi_0+\psi_1$ with
	\begin{align*}
		 &\varepsilon^{\f13} \|\partial_Y\psi_0\|_{L^2}+\|\psi_0\|_{L^2}\leq C\alpha^{-1}\varepsilon^{\f23},\\
		 &\varepsilon^{\f13}\|\partial_Y\psi_1\|_{L^2}+\alpha\|\psi_1\|_{L^2}\leq C\varepsilon^{\f13}.
	\end{align*}
\end{lemma}
\begin{proof}
    We first notice that
    \begin{align*}
    	(\partial_Y^2-\alpha^2)\varphi_{Ray}=&\frac{U''}{U}\varphi_{Ray}=\frac{U''}{U}\varphi_{Ray,0}+\frac{U''}{U}(\varphi_{Ray,1}+\varphi_{Ray,2})\\
    	=&\frac{c_E}{\alpha}U''e^{-\alpha Y}+\frac{U''}{U}(\varphi_{Ray,1}+\varphi_{Ray,2}).
    \end{align*}
	Then we decompose $\psi$ as $\psi=\psi_0+\psi_1$, where $\psi_0$ is the solution to
	\begin{align}
		\mathrm i \varepsilon(\partial_Y^2-\alpha^2)\psi_0+U\psi_0=\mathrm i \varepsilon\frac{c_E}{\alpha}U''e^{-\alpha Y},\quad\psi_0(0)=0,
	\end{align}
	and $\psi_1$ is the solution to
	\begin{align}
		\mathrm i \varepsilon(\partial_Y^2-\alpha^2)\psi_1+U\psi_1=\mathrm i \varepsilon \frac{U''}{U}(\varphi_{Ray,1}+\varphi_{Ray,2}),\quad \psi_2(0)=0.
	\end{align}
	By the first statement in  Lemma \ref{prop:nonlocal}, we have
	\begin{align*}
		&\|\partial_Y\psi_0\|_{L^2}\leq C\varepsilon^{\f13}\alpha^{-1}\|U''e^{-\alpha Y}\|_{L^2}\leq C\varepsilon^{\f13}\alpha^{-1},\\
&\|\psi_0\|_{L^2}\leq C\varepsilon^{\f23}\alpha^{-1}\|U''e^{-\alpha Y}\|_{L^2}\leq C\varepsilon^{\f23}\alpha^{-1}.
	\end{align*}
According to the last statement in Lemma \ref{prop:nonlocal} and \eqref{eq:hom-ray}, we have
\begin{align*}
		&\|\partial_Y\psi_1\|_{L^2}\leq C\|\frac{YU''}{U}(\varphi_{Ray,1}+\varphi_{Ray,2})\|_{L^2}\leq C,\\
		&\alpha\|\psi_1\|_{L^2}\leq C\alpha\varepsilon^{\f13}\|\frac{YU''}{U}(\varphi_{Ray,1}+\varphi_{Ray,2})\|_{L^2}\leq C\alpha\varepsilon^{\f13}.
\end{align*}

\end{proof}

Now we are ready to construct the slow mode. We define $\phi_s=\varphi_{Ray}+\psi+\tilde\phi$, where $\tilde\phi $ is the solution to
\begin{align}\label{eq:per-slow}
\left\{\begin{aligned}
   &\mathrm{i}\varepsilon(\partial_Y^2-\alpha^2)\tilde w +U\tilde w -U''\tilde\phi = -2\partial_Y(U'\psi) ,\\
   &(\partial_Y^2-\alpha^2)\tilde\phi =\tilde w ,\\
   &\tilde\phi |_{Y=0}=\partial_Y\tilde w|_{Y=0}=0.
\end{aligned}\right.
\end{align}
\begin{lemma}\label{Lem:slow-pertu}
	Let $\alpha\in(0,1]$ and $0<\varepsilon\leq c_1$, where $c_1$ is the small positive constant in Proposition \ref{prop:Orr-som-art}. Suppose that   $\tilde\phi$ is the solution to \eqref{eq:per-slow}. Then we have
	 \begin{align*}
    &\varepsilon^{\f13}\|(1+Y)\tilde w\|_{L^2}+ \varepsilon^{\f23}\|(1+Y)(\partial_Y\tilde w,\alpha\tilde w)\|_{L^2}+ \varepsilon\|(1+Y)(\partial_Y^2-\alpha^2)\tilde w\|_{L^2}\\
     &\quad+\|(\partial_Y\tilde\phi,\alpha\tilde\phi)\|_{L^2}\leq C(\alpha^{-1}\varepsilon^{\f13}+1).
     \end{align*}
	
\end{lemma}
\begin{proof}
	From Proposition \ref{prop:Orr-som-art} and the fact $U'\psi|_{Y=0}=0$, we have
	\begin{align*}
    &\varepsilon^{\f13}\|(1+Y)\tilde w\|_{L^2}+ \varepsilon^{\f23}\|(1+Y)(\partial_Y\tilde w,\alpha\tilde w)\|_{L^2}+ \varepsilon\|(1+Y)(\partial_Y^2-\alpha^2)\tilde w\|_{L^2}\\
     &\quad+\|(\partial_Y\tilde\phi,\alpha\tilde\phi)\|_{L^2}\leq C\|(1+Y)^2\partial_Y(U'\psi) \|_{L^2}.
  \end{align*}
	On the other hand, we get by Lemma \ref{Lem:slow-airy}  that
	\begin{align*}
		\|(1+Y)^2\partial_Y(U'\psi) \|_{L^2}\leq& \|(1+Y)^2U''\psi\|_{L^2}+\|(1+Y)^2U'\partial_Y\psi\|_{L^2}\\
		\leq&\|(1+Y)^2YU''\|_{L^\infty}\|\psi/Y\|_{L^2}+\|(1+Y)^2U'\|_{L^\infty}\|\partial_Y\psi\|_{L^2}\\
		\leq&C\|\partial_Y\psi\|_{L^2}\leq C(\alpha^{-1}\varepsilon^{\f13}+1).
	\end{align*}
   Therefore, we obtain 
  \begin{align*}
    &\varepsilon^{\f13}\|(1+Y)\tilde w\|_{L^2}+ \varepsilon^{\f23}\|(1+Y)(\partial_Y\tilde w,\alpha\tilde w)\|_{L^2}+ \varepsilon\|(1+Y)(\partial_Y^2-\alpha^2)\tilde w\|_{L^2}\\
     &\quad+\|(\partial_Y\tilde\phi,\alpha\tilde\phi)\|_{L^2}\leq C(\alpha^{-1}\varepsilon^{\f13}+1).
     \end{align*}
 \end{proof}

\begin{proposition}\label{prop:slow mode}
	Let $\alpha\in(0,1]$ and $0<\varepsilon\leq c_1$, where $c_1$ is the small positive constant in Proposition \ref{prop:Orr-som-art}. Then there exists a solution $\phi_s\in H^4$ to the homogeneous Orr-Sommerfeld equation such that
	\begin{align*}
		\phi_s=\frac{c_E}{\alpha}Ue^{-\alpha Y}+\phi_{s,re},\quad\phi_{s,re}(0)=1,
	\end{align*}
	where
	\begin{align*}
		&\|(\partial_Y\phi_{s,re},\alpha\phi_{s,re})\|_{L^2}\leq C(1+\alpha^{-1}\varepsilon^{\f13}),\\
		&\|\partial_Y\phi_{s,re}\|_{L^\infty}\leq C(\varepsilon^{-\f14}+\alpha^{-1}\varepsilon^{\frac{1}{12}}),\\
		&\|(\partial_Y^2-\alpha^2)\phi_{s,re}\|_{L^2}\leq C(\varepsilon^{-\f13}+\alpha^{-1}\varepsilon^{-\f16}).
	\end{align*}
	In particular,
	\begin{align*}
		\partial_Y\phi_s(0)=\frac{c_EU'(0)}{\alpha}+\mathcal O(\varepsilon^{-\f14}+\alpha^{-1}\varepsilon^{\frac{1}{12}}).
	\end{align*}
\end{proposition}

\begin{proof}
   Recall that $\phi_s=\varphi_{Ray}+\psi+\tilde\phi$. First of all, we show $\phi_s\in H^2$. For this moment, we assume that $U''/U\in L^2(\mathbb R_+)$ and later we recover the $H^2$ regularity of $\phi_s$ without the assumption $U''/U\in L^2(\mathbb R_+)$.
	 By Lemma \ref{Lem:slow-ray}-\ref{Lem:slow-pertu},  we obtain that $\phi_s\in H^2$. Hence, $\phi_s$ satisfies the homogeneous Orr-Sommerfeld equation in the following weak sense
	 \begin{align*}
	 	&\big\langle U(\partial_Y^2-\alpha^2)\phi_s-U''\phi_s,h\big\rangle+\mathrm i\varepsilon\big\langle(\partial_Y^2-\alpha^2)\phi_s,(\partial_Y^2-\alpha^2)h\big\rangle=0,\quad h\in H^2,~\partial_Y h(0)=0,
	 \end{align*}
	which implies that $w_s=(\partial_Y^2-\alpha^2)\phi_s$ is a weak solution to the Poisson equation $\mathrm i\varepsilon(\partial_Y^2-\alpha^2)w_s=-U w_s+U''\phi_s$ with the Neumann boundary condition $\partial_Y w_s|_{Y=0}=0$. Therefore, we have $w_s\in H^2$ and
	\begin{align}\label{eq:phis-energy}
		\big\langle Uw_s-U''\phi_s,w_s\big\rangle-\mathrm i\varepsilon(\|\partial_Y w_s\|_{L^2}^2+\alpha^2\|w_s\|_{L^2}^2)=0.
	\end{align}
	According to Lemma \ref{prop:nonlocal} and Remark \ref{re:airy-part}, we have
	\begin{align*}
		\|w_s\|_{L^2}\leq C\varepsilon^{-\f13} \|U''\phi_s\|_{L^2}.
	\end{align*}
	Hence, the $H^2$ regularity of $\phi_s$ is controlled by  $\|U''\phi_s\|_{L^2}$, which means that the assumption $U''/U\in L^2$ is not required. Now we complete the proof of $\phi_s\in H^4$.
	
	By Lemma \ref{Lem:slow-ray}-Lemma \ref{Lem:slow-pertu}, we directly have
	\begin{align*}
		\|(\partial_Y\phi_{s,re},\alpha\phi_{s,re})\|_{L^2}\leq C(1+\alpha^{-1}\varepsilon^{\f13}).
	\end{align*}
	Now we show the estimates about $\|(\partial_Y^2-\alpha^2)\phi_{s,re}\|_{L^2}$. Recall that
	\begin{align*}
		\phi_{s,re}=\phi_s-\frac{c_E}{\alpha}Ue^{-\alpha Y},\quad w_{s,re}=(\partial_Y^2-\alpha^2)\phi_{s,re}
	\end{align*}
	which along with \eqref{eq:phis-energy} implies
	\begin{eqnarray}\label{eq:slow-vort}
		\begin{split}
			\big\langle-&2c_EUU'e^{-\alpha Y}+Uw_{s,re}-U''\phi_{s,re},c_E\alpha^{-1}(U''e^{-\alpha Y}-2\alpha U'e^{-\alpha Y})+w_{s,re}\big\rangle\\
		&-\mathrm i\varepsilon(\|\partial_Yw_{s}\|_{L^2}^2+\alpha^2\|w_{s}\|_{L^2}^2)=0
		\end{split}
	\end{eqnarray}
By taking the real part of the above equality, we obtain 
	\begin{align*}
        \|\sqrt{U}w_{s,re}\|_{L^2}^2=&\mathbf{Re}\langle2c_EUU'e^{-\alpha Y}+U''\phi_{s,re},w_{s,re}\rangle-\mathbf{Re}\big\langle Uw_{s,re},\frac{c_E}{\alpha}(U''e^{-\alpha Y}-2\alpha U'e^{-\alpha Y})\big\rangle\\
        &+\mathbf{Re}\big\langle2c_EUU'e^{-\alpha Y}+U''\phi_{s,re},\frac{c_E}{\alpha}(U''e^{-\alpha Y}-2\alpha U'e^{-\alpha Y})\big\rangle,
	\end{align*}
	which shows that
	\begin{align*}
		\|\sqrt{U}w_{s,re}\|_{L^2}^2\leq& C\|U^{\f12}U'e^{-\alpha Y}\|_{L^2}^2+\|U''\phi_{s,re}\|_{L^2}\|w_{s,re}\|_{L^2}\\
		&+C\alpha^{-2}\|U^{\f12}(U''e^{-\alpha Y}-2\alpha U'e^{-\alpha Y})\|_{L^2}^2+C\alpha^{-1}(1+\|U''\phi_{s,re}\|_{L^2})\\
		\leq&C\alpha^{-2}+\|U''\phi_{s,re}\|_{L^2}\|w_{s,re}\|_{L^2}+C\alpha^{-1}(1+\|U''\phi_{s,re}\|_{L^2}).
	\end{align*}
	On the other hand, notice that $\phi_{s,re}=\varphi_{Ray,1}+\varphi_{Ray,2}+\psi+\tilde \phi$, and then 
	\begin{align*}
		\|U''\phi_{s,re}\|_{L^2}\leq \|U''\varphi_{Ray,1}\|_{L^2}+\|U''(\phi_{s,re}-\varphi_{Ray,1})\|_{L^2},
	\end{align*}
	which along with the fact $(\phi_{s,re}-\varphi_{Ray,1})|_{Y=0}=0$ implies
	\begin{align*}
		\|U''\phi_{s,re}\|_{L^2}&\leq C\|\varphi_{Ray,1}\|_{L^2}+C\|\frac{\phi_{s,re}-\varphi_{Ray,1}}{Y}\|_{L^2}\\
&\leq C\|\varphi_{Ray,1}\|_{L^2}+C(\|\partial_Y\phi_{s,re}\|_{L^2}+\|\partial_Y\varphi_{Ray,1}\|_{L^2})\leq C\Big(1+\frac{\varepsilon^{\f13}}{\alpha}\Big).
	\end{align*}
	Hence, we have
	\begin{align}\label{est:wsre1}
		\|\sqrt{U}w_{s,re}\|_{L^2}^2\leq C\alpha^{-2}+C\Big(1+\frac{\varepsilon^{\f13}}{\alpha}\Big)\|w_{s,re}\|_{L^2}.
	\end{align}
	
	From the imaginary part of \eqref{eq:slow-vort}, we have
	\begin{eqnarray}
		\begin{split}
			&\varepsilon(\|\partial_Y w_{s}\|_{L^2}^2+\alpha^2\|w_{s}\|_{L^2}^2)\\
			&\leq C\alpha^{-1}\|\sqrt{U}w_{s,re}\|_{L^2}+C\alpha^{-1}\|U''\phi_{s,re}\|_{L^2}+C\|w_{s,re}\|_{L^2}\\
			&\qquad+C\|U''\phi_{s,re}\|_{L^2}\|w_{s,re}\|_{L^2}\\
			&\leq C\alpha^{-2}+C\Big(1+\frac{\varepsilon^{\f13}}{\alpha}\Big)\|w_{s,re}\|_{L^2}.
		\end{split}
	\end{eqnarray}
Thanks to $w_{s,re}=w_s-c_E\alpha^{-1}U''\mathrm{e}^{-\alpha Y}+2c_EU'\mathrm{e}^{-\alpha Y}$, we have
\begin{align*}
   & \|\partial_Y w_{s,re}\|_{L^2}^2+\alpha^2\|w_{s,re}\|_{L^2}^2\leq \|\partial_Y w_{s}\|_{L^2}^2+\alpha^2\|w_{s}\|_{L^2}+C\alpha^{-1}.
\end{align*}  
Thus, we have 
\begin{align}\label{est:wsre2}
   & \varepsilon(\|\partial_Y w_{s,re}\|_{L^2}^2+\alpha^2\|w_{s,re}\|_{L^2}^2) \leq C\alpha^{-2}+C\Big(1+\frac{\varepsilon^{\f13}}{\alpha}\Big)\|w_{s,re}\|_{L^2}.
\end{align}
By a similar interpolation inequality as \eqref{est:interpolation-1}, we get
\begin{align*}
	\|w_{s,re}\|_{L^2}\leq C\varepsilon^\f16\|\partial_Y w_{s,re}\|_{L^2}^\f12\|w_{s,re}\|_{L^2}^\f12+C\varepsilon^{-\f16}\|\sqrt{U}w_{s,re}\|_{L^2},
\end{align*}
which along with \eqref{est:wsre1} and \eqref{est:wsre2}  implies that 
\begin{align*}
	\|w_{s,re}\|_{L^2}\leq& C\varepsilon^\f16\Big(\alpha^{-\f12}\varepsilon^{-\f14}+(\varepsilon^\f14+\varepsilon^{-\f16}\alpha^{-\f14})\|w_{s,re}\|_{L^2}^{\f14}\Big)\|w_{s,re}\|_{L^2}^\f12\\
	&+C\varepsilon^{-\f16}\Big(\alpha^{-1}+(1+\varepsilon^{\f16}\alpha^{-\f12})\|w_{s,re}\|_{L^2}^\f12\Big)\\
	\leq&C\big(\varepsilon^{-\f16}\alpha^{-1}+\varepsilon^{-\f13}\big).
\end{align*}
Then we finally have 
	\begin{align*}
		\|(\partial_Y^2-\alpha^2)\phi_{s,re}\|_{L^2}=\|w_{s,re}\|_{L^2}\leq C\big(\varepsilon^{-\f16}\alpha^{-1}+\varepsilon^{-\f13}\big),
	\end{align*}
and by the interpolation,
	\begin{align*}
		\|\partial_Y\phi_{s,re}\|_{L^\infty}\leq C\big(\varepsilon^{-\f13}+\alpha^{-1}\varepsilon^{-\f16}\big)^{\f12}\big(1+\alpha^{-1}\varepsilon^{\f13}\big)^{\f12}
\leq C\big(\varepsilon^{-\f14}+\alpha^{-1}\varepsilon^{\f{1}{12}}\big).
	\end{align*}
	
This completes the proof of the proposition.
\end{proof}

\subsection{Boundary corrector}
Here we will construct the boundary corrector $\Phi_b$.
Since $|\partial_Y\Phi(0)|\geq C^{-1}\varepsilon^{\f23}$, we know that the solution $W_b$ to \eqref{eq:boundary-corr-parphi} can be given by $W_b=W/\partial_Y\Phi(0)$ for the case of $1\leq|\alpha|<\infty$. For the case of $0<|\alpha|\leq1$, we need to use the slow mode $\phi_s$ to modify the boundary value of $\Phi_f$.

\begin{proposition}\label{prop:Wb-est}
Let $c_1$ be the constant in Proposition \ref{prop:Orr-som-art}. There exist $c_2\in(0,c_1]$, $\delta_{*}\in(0,1],$ such that if $\varepsilon|\alpha|^3\leq \delta_{*},\ 0<\varepsilon \leq c_2$, there exists a unique solution $\Phi_b\in H^4\cap H^1_0$ to the homogeneous Orr-Sommerfeld equation \eqref{eq:boundary-corr-parphi} satisfying the following properties:
\begin{itemize}
	\item If $1\leq|\alpha|<\infty$, then we have
    \begin{align*}
     & \varepsilon^{\f16}\|W_b\|_{L^2}+ \varepsilon^{-\f16}\|(\partial_Y\Phi_b,\alpha\Phi_b)\|_{L^2} +\|(\partial_Y\Phi_b,\alpha\Phi_b)\|_{L^\infty}\leq C.
  \end{align*}
  \item If $0<|\alpha|\leq 1$, then we have
        \begin{align*}
     & \varepsilon^{\f16}\|W_b\|_{L^2}+\frac{\alpha+\varepsilon^\f23}{\alpha\varepsilon^\f16+\varepsilon^\f23} \|(\partial_Y\Phi_b,\alpha\Phi_b)\|_{L^2} +\|(\partial_Y\Phi_b,\alpha\Phi_b)\|_{L^\infty}\leq C.
  \end{align*}
\end{itemize}
\end{proposition}
\begin{proof}
	We first consider the case of $1\leq|\alpha|<\infty$. We define $\Phi_b(Y)=\Phi(Y)/\partial_Y\Phi(0)$. By Lemma \ref{lem:W-est}, we have 
	\begin{align*}
     & \varepsilon^{\f16}\|W_b\|_{L^2}+ \varepsilon^{-\f16}\|(\partial_Y\Phi_b,\alpha\Phi_b)\|_{L^2} +\|(\partial_Y\Phi_b,\alpha\Phi_b)\|_{L^\infty}\leq C.
  \end{align*}
  Now we turn to the case of $0<|\alpha|\leq1$. We may assume  $0<\alpha\leq1$. Recall that
  \begin{align*}
  	\Phi_f(0)=\Phi_{a,f}(0)=\alpha^{-1}\int_{0}^{+\infty}W_a(Z)\sinh(\alpha Z)\mathrm{d}Z.
  \end{align*}
 We define 
  \begin{align*}
  	\Phi_b=A\Phi_f+B\phi_s,
  \end{align*}
  where
  \begin{align*}
  	A=\frac{1}{\partial_Y\Phi_f(0)-\Phi_f(0)\partial_Y\phi_s(0)},\quad B=-\frac{\Phi_f(0)}{\partial_Y\Phi_f(0)-\Phi_f(0)\partial_Y\phi_s(0)}.
  \end{align*}
  It is easy to check that $\Phi_b(0)=0,\partial_Y\Phi_{b}(0)=1$.
  From Lemma \ref{lem:W-est} and Proposition \ref{prop:slow mode}, we infer that
  \begin{align*}
  	\partial_Y\Phi_f(0)-\Phi_f(0)\partial_Y\phi_s(0)=-\dfrac{\varepsilon^\f23}{3(U'(0))^\f23Ai'(0)}+\dfrac{e^{\mathrm i\frac{\pi}{3}}\varepsilon^{\f43}Ai(0) (U'(0))^{\f23}}{Ai'(0)\alpha}+\mathcal O(\varepsilon+\varepsilon^{\frac{17}{12}}\alpha^{-1} ).
  \end{align*}
  On the other hand, notice that
  \begin{align*}
  	\left|-\dfrac{\varepsilon^\f23}{3(U'(0))^\f23Ai'(0)}+\dfrac{e^{\mathrm i\frac{\pi}{3}}\varepsilon^{\f43}Ai(0) (U'(0))^{\f23}}{Ai'(0)\alpha}\right|\geq C\left(\varepsilon^\f23+\frac{\varepsilon^\f43}{\alpha} \right),
  \end{align*}
  Therefore, we obtain 
  \begin{align*}
  	\left|\partial_Y\Phi_f(0)-\Phi_f(0)\partial_Y\phi_s(0)\right|\geq C\Big(\varepsilon^\f23+\frac{\varepsilon^\f43}{\alpha} \Big).
  \end{align*}
 Thus,
   \begin{align*}
  	|A|\leq C\frac{\alpha}{\varepsilon^\f23(\alpha+\varepsilon^\f23)} ,\quad |B|\leq C\frac{\alpha\varepsilon^\f23}{\alpha+\varepsilon^\f23} ,
  \end{align*}
  which along with Lemma \ref{lem:W-est} and Proposition \ref{prop:slow mode} imply that
  \begin{align*}
    \|(\partial_Y\Phi_{b},\alpha\Phi_b)\|_{L^2}\leq C\frac{\varepsilon^\f23+\alpha\varepsilon^\f16}{\varepsilon^\f23+\alpha},\quad\|(\partial_Y^2-\alpha^2)\Phi_b\|_{L^2}\leq C\varepsilon^{-\f16},\quad\|(\partial_Y\Phi_b,\alpha\Phi_b)\|_{L^\infty}\leq C.
    \end{align*}
    
 This finishes the proof of the proposition.
\end{proof}

%

\section{Orr-Sommerfeld equation with non-slip boundary condition}

This section is devoted to solve the Orr-Sommerfeld equation \eqref{eq:OS}. Since the source term of \eqref{eq:OS} belongs to $H^{-1}$, we decompose the solution $\phi$ as $\phi=\phi_0+\phi_1$, where $\phi_0$ and $\phi_1$ solve the following system respectively,
\begin{align}\label{eq:Orr-phi0}\left\{\begin{aligned}
&\partial_Y(U\partial_Y\phi_0)-\alpha^2U\phi_0+\mathrm{i}\varepsilon (\partial_Y^2-\alpha^2)^2\phi_0=-f_2-\dfrac{\mathrm{i}}{\alpha}\partial_Yf_1,\\
&\phi_0|_{Y=0}=\partial_Y\phi_0|_{Y=0}=0,
\end{aligned}\right.
\end{align}
and
\begin{align}\label{eq:Orr-phi1}\left\{\begin{aligned}
&U(\partial_Y^2-\alpha^2)\phi_1+\mathrm{i}\varepsilon (\partial_Y^2-\alpha^2)^2\phi_1-U''\phi_1=\partial_Y(U'\phi_0),\\
&\phi_1|_{Y=0}=\partial_Y\phi_1|_{Y=0}=0.
\end{aligned}\right.
\end{align}

\begin{remark}
The main reason that we use this decomposition is as follows.
\begin{itemize}
	\item All of terms on the left hand side of \eqref{eq:Orr-phi0} are the divergence form, so that we can obtain good estimates by simple energy method.
	
	\item The source term in \eqref{eq:Orr-phi1} satisfies $\int_0^\infty \partial_Y(U'\phi_0)dY'=0$, which reduces $|\alpha|^{-2}$ order loss compared with general  $L^2$ source term.
\end{itemize}
\end{remark}

Therefore, our main task of this section is to solve \eqref{eq:Orr-phi1}. For this purpose, we first focus on the Orr-Sommerfeld equation with general type source term:
\begin{align}\label{eq:Orr-som-nonslip}\left\{\begin{aligned}
&U(\partial_Y^2-\alpha^2)\varphi-U''\varphi+\mathrm{i} \varepsilon(\partial_Y^2-\alpha^2)^2\varphi=f,\quad f(0)=0,\\
&w=(\partial_Y^2-\alpha^2)\varphi,\quad \varphi|_{Y=0}=\partial_Y\varphi|_{Y=0}=0.
\end{aligned}\right.
\end{align}
For small $\varepsilon$, we could obtain estimates of the solution to \eqref{eq:Orr-som-nonslip} by applying the results in previous section. However, for general $\varepsilon>0$, we apply compactness argument under our spectral assumption $0\notin \sigma(\mathcal{L}_\nu)$.

\subsection{Estimates for small viscosity $\varepsilon$}

\begin{lemma}\label{prop:Orr-som-nonslip}
Let $c_2,\ \delta_{*}>0$ be the constants in Proposition \ref{prop:Wb-est}. Then if $\varepsilon|\alpha|^3\leq \delta_{*},\ 0<\varepsilon \leq c_2$, then for any $f\in H^1$ with $f(0)=0$, there exists a unique solution $\varphi\in H^4\cap H^2_{0}$ to the system \eqref{eq:Orr-som-nonslip} satisfying
 \begin{align*}
    &\varepsilon^{\f13}\|w\|_{L^2} +\varepsilon^{\f16}\|(\partial_Y\varphi,\alpha\varphi)\|_{L^\infty}+ \|(\partial_Y\varphi,\alpha\varphi)\|_{L^2} \leq C\left(\|(1+Y)^2f\|_{L^2}+\dfrac{1}{|\alpha|^2}\bigg| \int_{0}^{+\infty}f\mathrm{d}Y\bigg|\right).
  \end{align*}
\end{lemma}
\begin{proof}
Let $(w_{ar},\varphi_{ar})$ solve the following OS equation with Neumann's type boundary condition
\begin{align*}\left\{\begin{aligned}
&U(\partial_Y^2-\alpha^2)\varphi_{ar}-U''\varphi_{ar} +\mathrm{i}\varepsilon(\partial_Y^2-\alpha^2)^2\varphi_{ar}=f,\quad f(0)=0,\\
&w_{ar}=(\partial_Y^2-\alpha^2)\varphi_{ar},\quad \varphi_{ar}|_{Y=0}=\partial_Yw_{ar}|_{Y=0}=0.
\end{aligned}\right.
\end{align*}
We take  $(W_b,\Phi_b)$ to be the boundary layer corrector satisfying  \eqref{eq:boundary-corr-parphi}.
By comparing the boundary conditions, we can decompose $w$ as
\begin{align}\label{est:orr-som-non-0}
   &w=w_{ar}-\partial_Y\varphi_{ar}(0) W_b,\qquad \varphi=\varphi_{ar}-\partial_Y\varphi_{ar}(0)\Phi_b.
\end{align}
We point out that the existence of $(w_{ar},\varphi_{ar})$ and $(W_b,\Phi_b)$ is guaranteed by Proposition \ref{prop:Orr-som-art} and Proposition \ref{prop:Wb-est} respectively.
By Proposition \ref{prop:Orr-som-art}, we have
\begin{align}\label{est:orr-som-non-1}
   & \varepsilon^{\f13}\|w_{ar}\|_{L^2} +\|(\partial_Y\varphi_{ar},\alpha\varphi_{ar})\|_{L^2}\leq C\left(\|(1+Y)^2f\|_{L^2}+\dfrac{1}{|\alpha|^2}\bigg| \int_{0}^{+\infty}f\mathrm{d}Y\bigg|\right).
\end{align}
By the interpolation, we get
\begin{eqnarray}\label{est:orr-som-non-2}
\begin{split}
   |\partial_Y\phi_{ar}(0)| \leq& \|(\partial_Y\varphi_{ar},\alpha\varphi_{ar})\|_{L^\infty}\leq C \|w_{ar}\|_{L^2}^{\f12}\|(\partial_Y\varphi_{ar},\alpha\varphi_{ar})\|_{L^2}^{\f12}\\
   \leq& C\varepsilon^{-\f16}\left(\|(1+Y)^2f\|_{L^2}+\dfrac{1}{|\alpha|^2}\bigg| \int_{0}^{+\infty}f\mathrm{d}Y\bigg|\right).
\end{split}
\end{eqnarray}
For $(W_b, \Phi_b)$, we get by Proposition \ref{prop:Wb-est} that if $1\leq |\alpha|<+\infty$,
\begin{align}\label{est:orr-som-non-3}
   &  \varepsilon^{\f16}\|W_b\|_{L^2}+ \varepsilon^{-\f16}\|(\partial_Y\Phi_b,\alpha\Phi_b)\|_{L^2} +\|(\partial_Y\Phi_b,\alpha\Phi_b)\|_{L^\infty}\leq C,
\end{align}
and if $0<|\alpha|\leq 1$,
     \begin{align}\label{eq:Wb-alpha-small}
     & \varepsilon^{\f16}\|W_b\|_{L^2}+\frac{\alpha+\varepsilon^\f23}{\alpha\varepsilon^\f16+\varepsilon^\f23} \|(\partial_Y\Phi_b,\alpha\Phi_b)\|_{L^2} +\|(\partial_Y\Phi_b,\alpha\Phi_b)\|_{L^\infty}\leq C.
  \end{align}

Summing up \eqref{est:orr-som-non-0}-\eqref{eq:Wb-alpha-small} , we conclude that for $1\leq|\alpha|<\infty$,
\begin{align*}
   & \varepsilon^{\f13}\|w\|_{L^2} +\varepsilon^{\f16}\|(\partial_Y\varphi,\alpha\varphi)\|_{L^\infty}+ \|(\partial_Y\varphi,\alpha\varphi)\|_{L^2} \leq C\left(\|(1+Y)^2f\|_{L^2}+\dfrac{1}{|\alpha|^2}\bigg| \int_{0}^{+\infty}f\mathrm{d}Y\bigg|\right),
\end{align*}
and for $0<|\alpha|\leq1$,
\begin{align*}
		 & \varepsilon^{\f13}\|w\|_{L^2}+\frac{\alpha+\varepsilon^\f23}{\alpha+\varepsilon^\f12} \|(\partial_Y\varphi ,\alpha\varphi )\|_{L^2} +\varepsilon^\f16 \|(\partial_Y\varphi ,\alpha\varphi )\|_{L^\infty}\\
	 &\leq C\left(\|(1+Y)^2f\|_{L^2}+\dfrac{1}{|\alpha|^2}\bigg| \int_{0}^{+\infty}f\mathrm{d}Y\bigg|\right).
\end{align*}
In particular, we have
\begin{align*}
	\|(\partial_Y\varphi,\alpha\varphi)\|_{L^2}\leq& C\frac{\alpha+\varepsilon^\f12}{\alpha+\varepsilon^\f23}\left(\|(1+Y)^2f\|_{L^2}+\dfrac{1}{|\alpha|^2}\bigg| \int_{0}^{+\infty}f\mathrm{d}Y\bigg|\right)\\
	&\leq C(1+\varepsilon^{-\f16})\left(\|(1+Y)^2f\|_{L^2}+\dfrac{1}{|\alpha|^2}\bigg| \int_{0}^{+\infty}f\mathrm{d}Y\bigg|\right)\\
	&\leq C\varepsilon^{-\f16}\left(\|(1+Y)^2f\|_{L^2}+\dfrac{1}{|\alpha|^2}\bigg| \int_{0}^{+\infty}f\mathrm{d}Y\bigg|\right).
\end{align*}

This finishes the proof of the proposition.
\end{proof}

\begin{lemma}\label{prop:phi0}
Let $0<\varepsilon\leq 1$. Then for any $f\in L^2$, there exists a unique solution to \eqref{eq:Orr-phi0} satisfying
  \begin{align*}
     & \varepsilon^{\f13}\|(\partial_Y\phi_0,\alpha\phi_0)\|_{L^2}+ \varepsilon^{\f23}\|(\partial_Y^2-\alpha^2)\phi_0\|_{L^2} +\varepsilon^{\f16}\|\sqrt{U}(\partial_Y\phi_0,\alpha\phi_0)\|_{L^2}\leq C|\alpha|^{-1}\|(f_1,f_2)\|_{L^2}.
  \end{align*}
\end{lemma}

\begin{proof}
This elliptic problem is uniquely solvable. Hence, it suffices to provide a priori estimate.
  Taking the inner product with $\phi_0$, we get
  \begin{align*}
     &\int_{0}^{+\infty}-U\big(|\partial_Y\phi_0|^{2}+\alpha^2|\phi_0|^2\big) \mathrm{d}Y+ \mathrm{i}\varepsilon\|(\partial_Y^2-\alpha^2)\phi_0\|_{L^2}^2= -\int_{0}^{+\infty}(f_2+\dfrac{\mathrm{i}}{\alpha}\partial_Yf_1)\phi_0\mathrm{d}Y.
  \end{align*}
Taking the real part and imaginary part of the above equation, we obtain
  \begin{eqnarray}\label{est:phi0-1}
  \begin{split}
     &\big\|\sqrt{U}(\partial_Y\phi_0,\alpha\phi_0)\big\|_{L^2}^2+ \varepsilon\|(\partial_Y^2-\alpha^2)\phi_0\|_{L^2}^2 \leq \left| \int_{0}^{+\infty}(f_2+\dfrac{\mathrm{i}}{\alpha}\partial_Yf_1)\phi_0\mathrm{d}Y\right|\\
     & \leq C|\alpha|^{-1}\|(f_1,f_2)\|_{L^2} \|(\partial_Y\phi_0,\alpha\phi_0)\|_{L^2}.
  \end{split}
  \end{eqnarray}
  
On the other hand, using the structure assumptions on $U$, we can actually obtain the control of $\|(\partial_Y\phi_0,\alpha\phi_0\|_{L^2}$ from $\|\sqrt{U}(\partial_Y\phi_0,\alpha\phi_0)\|_{L^2}$ via the interpolation. Indeed, we notice that  $1\leq C(U/Y)^{\f13}+C \sqrt{U(Y)}$ and $\partial_Y\phi_0|_{Y=0}=\phi_0|_{Y=0}$, which implies that
  \begin{align*}
     &\|(\partial_Y\phi_0,\alpha\phi_0)\|_{L^2}\leq C \left\|(U/Y)^{\f13}(\partial_Y\phi_0,\alpha\phi_0)\right\|_{L^2} +C\|\sqrt{U}(\partial_Y\phi_0,\alpha\phi_0)\|_{L^2}\\
     &\leq C\Big(\int_0^\infty(\frac{|\partial_Y\phi_0|}{Y})^{\f23}(U^{\f23}|\partial_Y\phi_0|^{\f43})dY\Big)^{\f12}+C\alpha\Big(\int_0^\infty(\frac{|\phi_0|}{Y})^{\f23}(U^{\f23}|\phi_0|^{\f43})dY\Big)^{\f12}\\
     &\quad+C\|\sqrt{U}(\partial_Y\phi_0,\alpha\phi_0)\|_{L^2}\\
     &\leq C\left\|\sqrt{U}(\partial_Y\phi_0,\alpha\phi_0)\right\|_{L^2}^{\f23} \left\|(\partial_Y\phi_0,\alpha\phi_0)/Y\right\|_{L^2}^{\f13} +C\|\sqrt{U}(\partial_Y\phi_0,\alpha\phi_0)\|_{L^2}\\
    &\leq C\left\|\sqrt{U}(\partial_Y\phi_0,\alpha\phi_0)\right\|_{L^2}^{\f23} \left\|(\partial_Y^2-\alpha^2)\phi_0\right\|_{L^2}^{\f13} +C\|\sqrt{U}(\partial_Y\phi_0,\alpha\phi_0)\|_{L^2},
  \end{align*}
 which along with \eqref{est:phi0-1} gives
  \begin{align*}
     &\|(\partial_Y\phi_0,\alpha\phi_0)\|_{L^2} \leq C\big(\varepsilon^{-\f16}+1\big)|\alpha|^{-\f12}\|(f_1,f_2)\|_{L^2}^{\f12}\|(\partial_Y\phi_0,\alpha\phi_0)\|_{L^2}^{\f12}.
  \end{align*}
  Applying \eqref{est:phi0-1} again and using $\varepsilon\leq 1$, we conclude
  \begin{align*}
     & \varepsilon^{\f13}\|(\partial_Y\phi_0,\alpha\phi_0)\|_{L^2}+ \varepsilon^{\f23}\|(\partial_Y^2-\alpha^2)\phi_0\|_{L^2} +\varepsilon^{\f16}\|\sqrt{U}(\partial_Y\phi_0,\alpha\phi_0)\|_{L^2}\leq C|\alpha|^{-1}\|(f_1,f_2)\|_{L^2}.
  \end{align*}
 
\end{proof}

\begin{proposition}\label{prop:os-nonslip-small}
  Let $f=(f_1,f_2)\in L^2(\mathbb{R}_{+}^2)$. Then there exist small positive number $\delta_*, c_*$ such that for any $\varepsilon|\alpha|^3\leq \delta_*$, $0<\varepsilon\leq c_*$, there exists a unique solution $\phi\in H^4\cap H^2_0(\mathbb{R}_+)$ to \eqref{eq:OS} satisfying
 \begin{align*}
    &\varepsilon^{\f23}|\alpha|\|(\partial_Y^2-\alpha^2)\phi\|_{L^2} +\varepsilon^{\f12}|\alpha|\|(\partial_Y\varphi,\alpha\varphi)\|_{L^\infty}+ \varepsilon^{\f13}|\alpha|\|(\partial_Y\phi,\alpha\phi)\|_{L^2} \leq C\|(f_1,f_2)\|_{L^2}.
  \end{align*}
\end{proposition}

\begin{proof}
  We decompose $\phi$ as $\phi=\phi_0+\phi_1$, where $\phi_0,\ \phi_1$ solve \eqref{eq:Orr-phi0} and \eqref{eq:Orr-phi1} respectively. By Lemma \ref{prop:Orr-som-nonslip}, we get
   \begin{align*}
      & \varepsilon^{\f13}\|(\partial_Y^2-\alpha^2)\phi_1\|_{L^2} +\varepsilon^{\f16}\|(\partial_Y\phi_1,\alpha\phi_1)\|_{L^\infty}+ \|(\partial_Y\phi_1,\alpha\phi_1)\|_{L^2} \\
   &\leq C\left(\|(1+Y)^2\partial_Y(U'\phi_0)\|_{L^2}+\dfrac{1}{|\alpha|^2}\bigg| \int_{0}^{+\infty}\partial_Y(U'\phi_0)\mathrm{d}Y\bigg|\right)\\
   &\leq C\big\|\big(\partial_Y\phi_0,\phi_0/(1+Y)\big)\big\|_{L^2}\leq C\|\partial_Y\phi_0\|_{L^2}.
   \end{align*}
   We get by Lemma \ref{prop:phi0} and the interpolation that
   \begin{align*}
      & \varepsilon^{\f23}\|(\partial_Y^2-\alpha^2)\phi_0\|_{L^2} +\varepsilon^{\f12}\|(\partial_Y\phi_0,\alpha\phi_0)\|_{L^\infty} +\varepsilon^{\f13}\|(\partial_Y\phi_0,\alpha\phi_0)\|_{L^2}\\
      &\leq C\big(\varepsilon^{\f13}\|(\partial_Y\phi_0,\alpha\phi_0)\|_{L^2} +  +\varepsilon^{\f23}\|(\partial_Y^2-\alpha^2)\phi_0\|_{L^2}\big)\\
      &\leq C|\alpha|^{-1}\|(f_1,f_2)\|_{L^2}.
   \end{align*}
 Summing up the above two inequalities, we finish the proof.
\end{proof}

\subsection{Estimates when $c_2\leq\varepsilon\leq1$}

\begin{definition}
We say that $\varphi$ is a  weak solution to \eqref{eq:Orr-som-nonslip} if
  $\varphi\in H^2_0(\mathbb{R_+}),\ \partial_Y\varphi\in H^1_0(\mathbb{R_+})$, and for any $\xi\in C_{c}^{\infty}(\mathbb{R_+})$, it holds that
  \begin{align*}
     & \mathrm{i}\varepsilon \langle w, (\partial_Y^2-\alpha^2)\xi\rangle+ \langle Uw,\xi\rangle-\langle U''\varphi,\xi\rangle=\langle f,\xi\rangle,\quad w=(\partial_Y^2-\alpha^2)\varphi.
  \end{align*}
\end{definition}

In the rest of this part, we always assume the following solvability. \smallskip

\no{\it Solvalbility Assumption (S-A)}: for  $c_2\leq \varepsilon\leq1$ and $\alpha\neq0$, there is no nontrivial weak solution to the homogeneous equation
  \begin{align}\label{eq:OS-002}\left\{\begin{aligned}
  &\mathrm{i}\varepsilon(\partial_Y^2-\alpha^2)w+Uw-U''\varphi=0,\\
  &(\partial_Y^2-\alpha^2)\varphi=w,\quad
  \partial_Y\varphi|_{Y=0}=\varphi|_{Y=0}=0.
  \end{aligned}\right.
  \end{align}

\begin{remark}
The solvalbility assumption is ensured by our spectral condition.
For small $0<\varepsilon\leq c_2$, Lemma \ref{prop:Orr-som-nonslip}  can ensure that there only exists zero solution to \eqref{eq:OS-002}. When $\al=0$, it is easy to find by the energy method that for any $0<\varepsilon\leq 1$, there only exists zero solution to \eqref{eq:OS-002} .
\end{remark}

\begin{lemma}\label{lem:comp-Yinfty}
 Let $0<\varepsilon\leq 1,\ \alpha\neq 0$. Let $(1+Y)^2f\in L^2(\mathbb{R}_+)$ and $f(0)=0$ be given. Suppose that $\varphi$ is the solution to   \eqref{eq:Orr-som-nonslip} with source term $f$ and $\varphi\in H^2(\mathbb{R}_+)$. Then we have
  \begin{align*}
     &\big\|\sqrt{Y}(\partial_Y\varphi,\alpha\varphi)\big\|_{L^2}^2\leq C\|(\partial_Y\varphi,\alpha\varphi)\|_{L^2}\big( \varepsilon\|(1+Y)(\partial_Yw,\al w)\|_{L^2}+\|(1+Y)^2f\|_{L^2}\big),\\
     &\varepsilon^{\f23}\|Y(\partial_Yw,\alpha w)\|_{L^2}+ \varepsilon^{\f13}\|Yw\|_{L^2}\leq C\big(\|\partial_Y\varphi\|_{L^2}+\|Yf\|_{L^2}+\varepsilon\|\partial_Yw\|_{L^2} \big).
  \end{align*}
\end{lemma}

\begin{proof}
First of all, we construct a $C^2(\mathbb{R}_+)$ cut-off function $\chi\geq 0,$ such that
\begin{align*}
   &\chi(Y)=1,\ \text{if}\ Y\in[0,1);\quad\chi(Y)=0,\ \text{if}\ Y\in[2,+\infty).
\end{align*}
Let $\chi_R(Y)=\chi(Y/R),\ R>1$. Then
\begin{align*}
   &\chi_R(Y)=1,\ \text{if}\ Y\in[0,R);\quad\chi(Y)=0,\ \text{if}\ Y\in[2R,+\infty),\\
   &\sum_{k=1,2}|R^k\partial_Y^k\chi_R|\leq C,
\end{align*}
which implies that
\begin{align*}
	 \sum_{k=1,2}|(1+Y)^{k}\partial_Y^k\chi_R|\leq C.
\end{align*}

Taking the inner production with $-Y\chi_R\varphi/U\in L^2$, we get
\begin{align}\label{est:nonl-psi/u}
   & \mathrm{i}\varepsilon\big\langle(\partial_Y^2-\alpha^2)w, -Y\chi_R\varphi/U\big\rangle+ \langle w,-Y\chi_R\varphi\rangle+\Big\|\sqrt{\frac{U''Y\chi_R}{U}}\varphi\Big\|_{L^2}^2=\big\langle f,-Y\chi_R\varphi/U\big\rangle.
\end{align}

Thanks to $|\partial_Y^2(Y\chi_R)|\leq CR^{-1}$, we get
\begin{align*}
   \mathbf{Re}\left(\langle w,-Y\chi_R\varphi\rangle\right)&= \mathbf{Re} \left(\langle \partial_Y\varphi,\partial_Y(Y\chi_R)\varphi \rangle + \langle \partial_Y\varphi,Y\chi_R\partial_Y\varphi \rangle +\alpha^2\langle \varphi,Y\chi_R\varphi\rangle\right)\\
   &=\big\langle \partial_Y|\varphi|^2/2,\partial_Y(Y\chi_R) \big\rangle+ \left\|\sqrt{Y\chi_R}(\partial_Y\varphi,\alpha\varphi)\right\|_{L^2}^2\\
   &=-\dfrac{1}{2}\int_{0}^{+\infty} \partial_Y^2(Y\chi_R)|\varphi|^2\mathrm{d}Y+ \left\|\sqrt{Y\chi_R}(\partial_Y\varphi,\alpha\varphi)\right\|_{L^2}^2\\
   &\geq -CR^{-1}\|\varphi\|_{L^2}^2+ \left\|\sqrt{Y\chi_R}(\partial_Y\varphi,\alpha\varphi)\right\|_{L^2}^2.
\end{align*}
Hence, we obtain
\begin{eqnarray}\label{est:nonl-wphi}
\begin{split}
 \left\|\sqrt{Y\chi_R}(\partial_Y\varphi,\alpha\varphi)\right\|_{L^2}^2 \leq&  \mathbf{Re}\left(\langle w,-Y\chi_R\varphi\rangle\right)+CR^{-1}|\alpha|^{-2}\|\alpha\varphi\|_{L^2}^2\\
 \leq&\varepsilon\left|\big\langle(\partial_Y^2-\alpha^2)w, -Y\chi_R\varphi/U\big\rangle \right|+\left|\langle f,-Y\chi_R\varphi/U\rangle\right|\\
 &-\Big\|\sqrt{\frac{U''Y\chi_R}{U}}\varphi\Big\|_{L^2}^2+CR^{-1}|\alpha|^{-2}\|\alpha\varphi\|_{L^2}^2\\
 \leq&\varepsilon\left|\big\langle(\partial_Y^2-\alpha^2)w, -Y\chi_R\varphi/U\big\rangle \right|+\left|\langle f,-Y\chi_R\varphi/U\rangle\right|\\
 &+CR^{-1}|\alpha|^{-2}\|\alpha\varphi\|_{L^2}^2.
\end{split}
\end{eqnarray}

We get by integration by parts that
\begin{align*}
   &\left|\big\langle(\partial_Y^2-\alpha^2)w, -Y\chi_R\varphi/U\big\rangle \right| = \left|\big\langle\partial_Yw, \partial_Y\big(Y\chi_R\varphi/U\big)\big\rangle+ \alpha^2\langle w, Y\chi_R\varphi/U\rangle\right|\\
   &\leq \left|\big\langle(1+Y)\partial_Yw, \big(Y\chi_R/U(1+Y)\big)\partial_Y\varphi\big\rangle+ \langle\alpha(1+Y)w, \big(Y\chi_R/U(1+Y)\big)\alpha\varphi\rangle\right|\\
   &\quad+ \left|\langle\partial_Yw, \partial_Y(Y\chi_R/U)\varphi\rangle\right|\\
   &\leq \left\|\dfrac{Y\chi_R}{(1+Y)U}\right\|_{L^\infty} \bigg(\|(1+Y)\partial_Yw\|_{L^2} \|\partial_Y\varphi\|_{L^2}+\|(1+Y)\alpha w\|_{L^2} \|\alpha\varphi\|_{L^2}\bigg)\\
   &\quad+ \left\|\dfrac{Y\partial_Y(Y\chi_R/U)}{(1+Y)}\right\|_{L^\infty} \|(1+Y)\partial_Yw\|_{L^2}\|\varphi/Y\|_{L^2},
\end{align*}
which along with the fact that $U\geq C^{-1}Y/(1+Y)$, $\big|\partial_Y(Y\chi_R/U)\big|\leq C(1+Y)/Y$, implies
\begin{align}\label{est:nonl-par2wphiu}
  \left|\big\langle(\partial_Y^2-\alpha^2)w, -Y\chi_R\varphi/U\big\rangle \right|\leq C\|(1+Y)(\partial_Yw,\alpha w)\|_{L^2} \|(\partial_Y\varphi,\alpha\varphi)\|_{L^2}.
\end{align}
Using the fact $U\geq C^{-1}Y/(1+Y)$ again, we have
\begin{eqnarray}\label{est:nonl-fYvarphi}
  \begin{split}
     \left|\langle f,-Y\chi_R\varphi/U\rangle\right| &\leq \left\|\dfrac{Y\chi_R}{(1+Y)U}\right\|_{L^\infty}\|(1+Y)Yf\|_{L^2} \|\varphi/Y\|_{L^2}\\
   &\leq C\|(1+Y)^2f\|_{L^2}\|\partial_Y\varphi\|_{L^2}.
  \end{split}
\end{eqnarray}
Summing up \eqref{est:nonl-wphi}, \eqref{est:nonl-par2wphiu} and \eqref{est:nonl-fYvarphi}, we arrive at
\begin{align*}
   &\left\|\sqrt{Y\chi_R}(\partial_Y\varphi,\alpha\varphi)\right\|_{L^2}^2\\
   &\leq C\|(\partial_Y\varphi,\alpha\varphi)\|_{L^2} \big(R^{-1}|\alpha|^{-2}\|\alpha\varphi\|_{L^2}+\varepsilon\|(1+Y)(\partial_Yw,\al w)\|_{L^2} +\|(1+Y)^2f\|_{L^2}\big),
\end{align*}
which gives by taking $R\rightarrow +\infty$ that
\begin{align*}
   &\big\|\sqrt{Y}(\partial_Y\varphi,\alpha\varphi)\big\|_{L^2}^2\leq C\|(\partial_Y\varphi,\alpha\varphi)\|_{L^2}\big( \varepsilon\|(1+Y)(\partial_Yw,\al w)\|_{L^2}+\|(1+Y)^2f\|_{L^2}\big).
\end{align*}
This shows the first inequality of the lemma.

On the other hand, $Y\chi_R w$ satisfies
\begin{align*}
   & \mathrm{i}\varepsilon(\partial_Y^2-\alpha^2)(Y\chi_Rw)+U(Y\chi_Rw)= Y\chi_RU''\varphi+Y\chi_Rf+ 2\mathrm{i}\varepsilon\partial_Y(Y\chi_R)\partial_Yw+\mathrm{i}\varepsilon \partial_Y^2(Y\chi_R)w.
\end{align*}
For the source term on the right hand side, we have
\begin{align*}
   & \left\|Y\chi_RU''\varphi+Y\chi_Rf+ 2\mathrm{i}\varepsilon\partial_Y(Y\chi_R)\partial_Yw+\mathrm{i}\varepsilon \partial_Y^2(Y\chi_R)w \right\|_{L^2}\\
   \leq& C\big(\|\partial_Y\varphi\|_{L^2}+\|Yf\|_{L^2}+\varepsilon\|\partial_Yw\|_{L^2} +R^{-1}\|w\|_{L^2}\big).
\end{align*}
Thanks to $(Y\chi_Rw)|_{Y=0}=0$, we get by Lemma \ref{prop:nonlocal}  that
\begin{align*}
   & \varepsilon^{\f23}\|(\partial_Y,\alpha)(Y\chi_Rw)\|_{L^2} +\varepsilon^{\f13}\|Y\chi_Rw\|_{L^2}\leq C\big(\|\partial_Y\varphi\|_{L^2}+\|Yf\|_{L^2}+\varepsilon\|\partial_Yw\|_{L^2} +R^{-1}\|w\|_{L^2}\big).
\end{align*}
Taking $R\rightarrow +\infty$,  we obtain
\begin{align*}
   & \varepsilon^{\f23}\|Y(\partial_Yw,\alpha w)\|_{L^2}+ \varepsilon^{\f13}\|Yw\|_{L^2}\leq C\big(\|\partial_Y\varphi\|_{L^2}+\|Yf\|_{L^2}+\varepsilon\|\partial_Yw\|_{L^2} \big).
\end{align*}

\end{proof}

Now we consider the case when $|\alpha|$ has upper bound and $\varepsilon$ has lower bound.

\begin{lemma}\label{prop:epsilonbig-alphasmall}
  Let $M_1>0,\ 0<c_3\leq 1$ and $0\neq |\alpha|\leq M_1$.  Assume that the solvalbility assumption (S-A) holds. Let $\varphi$ be the solution to  \eqref{eq:Orr-som-nonslip} with $(1+Y)^2f\in L^2(\mathbb{R}_+)$ and $f\in H_0^1(\mathbb{R}_{+})$. Then it holds that for any $1\geq\varepsilon\geq c_3$,
  \begin{align*}
     & \|w\|_{L^2}+\|(\partial_Y\varphi,\alpha \varphi)\|_{L^2}\leq C\|(1+Y)^2f\|_{L^2},
  \end{align*}
  where constant $C$ depends on $c_3$ and $ M_1$.
\end{lemma}

\begin{proof}
Without loss of generality, we may assume $\alpha\geq0$. We prove this proposition by a contradiction argument.
 Assume that this proposition is not true. Then there exists a sequence $\{\varepsilon_n,\ \alpha_n,\ \varphi^{(n)},\ w^{(n)},\ f^{(n)}\}$ satisfying
  \begin{align}\label{eq:comp-1}\left\{\begin{aligned}
    &c_3\leq \varepsilon_n\leq 1,\quad 0\neq|\alpha_n|\leq M_1,\\
    &\mathrm{i}\varepsilon_n(\partial_Y^2-\alpha_n^2)w^{(n)}+ Uw^{(n)}-U''\varphi^{(n)}=f^{(n)},\\
    &(\partial_Y^2-\alpha^2_n)\varphi^{(n)}=w^{(n)},\quad \varphi^{(n)}|_{Y=0}=\partial_Y\varphi^{(n)}|_{Y=0}=0,\\
    &\|w^{(n)}\|_{L^2}+ \|(\partial_Y\varphi^{(n)},\alpha_n\varphi^{(n)})\|_{L^2} =1,\quad \|(1+Y)^2f^{(n)}\|_{L^2}\rightarrow 0,
  \end{aligned}\right.
  \end{align}
  such that we can take a subsequence $\{\varepsilon_n,\ \alpha_n,\ \varphi^{(n)},\ w^{(n)},\ f^{(n)}\}$( denoted by the same index) satisfying
  \begin{align}\label{rela:epsilon-n}
     & \varepsilon_n\rightarrow \varepsilon\geq c_3,\quad \alpha_n\rightarrow \alpha\in [0,M_1],\quad\mathrm{as}\quad n\to\infty.
  \end{align}
Then we obtain the following weak convergence results when $n\to\infty$:
  \begin{itemize}
    \item If $\alpha=0$, we have
    \begin{align*}
       & \|\partial_Y\varphi^{(n)}\|_{H^1}+\|\alpha_n\varphi^{(n)}\|_{L^2}\leq C\big(\|w^{(n)}\|_{L^2}+ \|(\partial_Y\varphi^{(n)},\alpha_n\varphi^{(n)})\|_{L^2})\leq C.
    \end{align*}
    Then
    \begin{align*}
       &\partial_Y\varphi^{(n)}\rightharpoonup u_1\quad\text{in}\ H^1_0(\mathbb{R}_+)\quad \text{with}\ u_1\in H^1_0(\mathbb{R}_+),\\
    &\alpha_n\varphi^{(n)}\rightharpoonup \mathrm{i}u_2\quad\text{in}\ H^1_0(\mathbb{R}_+)\quad \text{with}\ \mathrm{i}u_2\in H^1_0(\mathbb{R}_+),
    \end{align*}
    hence,
    \begin{align}\label{rela:weakcon-a=0}
       &(w^{(n)},\partial_Y\varphi^{(n)},\alpha_n\varphi^{(n)}) \rightharpoonup (\partial_Yu_1,u_1,\mathrm{i}u_2)\quad\text{in}\ L^2(\mathbb{R}_+)\quad \text{with}\ (u_1,u_2)\in H^1_0(\mathbb{R}_+).
    \end{align}

    \item If $\alpha\neq 0,$ we may assume that $\alpha_n>0,$ and thanks to $\alpha_n\geq \alpha/2$ for large $n$, we can deduce that
  \begin{align*}
     &\|\varphi^{(n)}\|_{H^2}\leq  (|\alpha_n|^{-1}+1)\big( \|w^{(n)}\|_{L^2}+ \|(\partial_Y\varphi^{(n)},\alpha_n\varphi^{(n)})\|_{L^2}\big)\leq (2\alpha^{-1}+1).
  \end{align*}
  Then $\|\varphi^{(n)}\|_{H^2}$ is uniformly bounded, hence there exists $\varphi\in H^2(\mathbb{R}_+)$, such that $\varphi^{(n)}\rightharpoonup \varphi$ in $H^2(\mathbb{R}_+)$. We get
  \begin{align}\label{rela:weakcon-aneq0}
     &(w^{(n)},\partial_Y\varphi^{(n)},\alpha_n\varphi^{(n)})\rightharpoonup ((\partial_Y^2-\alpha^2)\varphi,\partial_Y\varphi,\alpha\varphi) := (w,u_1,\mathrm{i}u_2) \quad\text{in}\ L^2(\mathbb{R}_+),\\
     & \text{with}\ \varphi\in H^2(\mathbb{R}_+)\,\textrm{and}\, (w,u_1,\mathrm{i}u_2)\in L^2(\mathbb{R}_+).\nonumber
  \end{align}
  \end{itemize}

 \no\textbf{Step 1. Strong convergence}.

  Since $w,\varphi$ and $f$ all belong to $L^2(\mathbb{R}_+)$, we know that $(\partial_y^2-\alpha^2)w$  is also in $L^2(\mathbb{R}_+)$. Moreover, the following bound holds
  \begin{align*}
     \varepsilon_n\|(\partial_Y^2-\alpha_n^2)w^{(n)}\|_{L^2}\leq & \|Uw^{(n)}\|_{L^2}+\|U''\varphi^{(n)}\|_{L^2}+\|f^{(n)}\|_{L^2}\\
     \leq& C(\|w^{(n)}\|_{L^2}+\|\partial_Y\varphi^{(n)}\|_{L^2})+\|f^{(n)}\|_{L^2}.
  \end{align*}
Thanks to $\varepsilon_n\geq c_3>0$, we deduce that $\|(\partial_Y^2-\alpha_n^2)w^{(n)}\|_{L^2}$ are uniformly bounded with respect to $n$. We get by Lemma \ref{lemma:nonslip-elliptic} that
  \begin{align*}
     &\|\partial_Yw^{(n)}\|_{L^2}\leq C\|(\partial_Y^2-\alpha^2_n)w^{(n)}\|_{L^2}^{\f12}\|w^{(n)}\|_{L^2}^{\f12}.
   \end{align*}
   Then we have
   \begin{align}
    \|\partial_Yw^{(n)}\|_{L^2}\ \text{are uniformly bounded}.\label{est:comp-parw-bounded}
  \end{align}
  Thanks to $\|\partial_Y(\partial_Y\varphi^{(n)},\alpha_n\varphi^{(n)})\|_{L^2}\leq C\|w^{(n)}\|_{L^2}\leq C$, we get
  \begin{align}
     & \|\partial_Y(\partial_Y\varphi^{(n)},\alpha_n\varphi^{(n)})\|_{L^2}\ \text{are uniformly bounded}.\label{est:comp-par(parphi)-bounded}
  \end{align}
  Thanks to Lemma \ref{lem:comp-Yinfty}, we get
  \begin{align*}
     &\big\|\sqrt{Y}(\partial_Y\varphi^{(n)},\alpha_n\varphi^{(n)})\big\|_{L^2}^2\\
     &\leq C\|(\partial_Y\varphi^{(n)},\alpha_n\varphi^{(n)})\|_{L^2}\big( \varepsilon_n\|(1+Y)(\partial_Yw^{(n)},\al_n w^{(n)})\|_{L^2}+\|(1+Y)^2f^{(n)}\|_{L^2}\big),
  \end{align*}
  and
  \begin{align*}
     &\varepsilon_n\|Y(\partial_Y^2-\alpha_n^2) w^{(n)})\|_{L^2}+ \varepsilon^{\f13}_n\|Yw^{(n)}\|_{L^2}\leq C\big(\|\partial_Y\varphi^{(n)}\|_{L^2}+\|Yf^{(n)}\|_{L^2}+ \varepsilon_n\|\partial_Yw^{(n)}\|_{L^2} \big).
  \end{align*}
  Summing up, we arrive at
  \begin{align}\label{est:Y-bounded}
     &\|(1+Y)(\partial_Y^2-\alpha_n^2)w^{n}\|_{L^2}+\big\|\sqrt{Y}(\partial_Y\varphi^{(n)},\alpha_n\varphi^{(n)})\big\|_{L^2} +\|Yw^{(n)}\|_{L^2}
  \end{align}
are uniformly bounded. Then
  \begin{align}\label{est:comp-parvarphi-w}
     &\lim_{R\rightarrow +\infty}\sup_{n}\big(\|(w^{(n)},\partial_Y\varphi^{(n)}, \alpha_n\varphi^{(n)})\|_{L^2[R,+\infty)}\big)=0.
  \end{align}
  Along with \eqref{est:comp-parw-bounded}, \eqref{est:comp-par(parphi)-bounded} and \eqref{est:comp-parvarphi-w}, we deduce that as $n\to+\infty$,
  \begin{align}\label{est:comp-strong-a=0}
     & (w^{(n)},\partial_Y\varphi^{(n)},\alpha_n\varphi^{(n)})\rightarrow (\partial_Yu_1,u_1,\mathrm{i}u_2)\quad \text{strongly in}\ L^2(\mathbb{R}_+).
  \end{align}
  Moreover if $\alpha\neq0$, we have
  \begin{align}\label{est:comp-strong-aneq0}
     & \varphi^{(n)}\rightarrow \varphi\quad \text{strongly in}\ H^2(\mathbb{R}_+).
  \end{align}

 \no\textbf{Step 2. The limit  equation.}

  Let $\xi$ be any $C_c^{\infty}(\mathbb{R}_+)$ test function. Then we have
  \begin{align*}
     & \mathrm{i}\varepsilon_n \langle w^{(n)}, (\partial_Y^2-\alpha^2_n)\xi\rangle+ \langle Uw^{(n)},\xi\rangle-\langle U''\varphi^{(n)},\xi\rangle=\langle f^{(n)},\xi\rangle,\quad w^{(n)}=(\partial_Y^2-\alpha^2_n)\varphi^{(n)}.
  \end{align*}
\begin{itemize}
    \item
  If $\alpha\neq 0$, by \eqref{est:comp-strong-aneq0} and passing to the limit for $n\to\infty$, we get
  \begin{align*}
     & \mathrm{i}\varepsilon \langle w, (\partial_Y^2-\alpha^2)\xi\rangle+ \langle Uw,\xi\rangle-\langle U''\varphi,\xi\rangle=0,\quad w=(\partial_Y^2-\alpha^2)\varphi.
  \end{align*}
For the boundary condition, thanks to $\varphi^{(n)}\in H^2_{0}(\mathbb{R}_+),\ \partial_Y\varphi^{(n)}\in H^1_{0}(\mathbb{R}_+)$ and \eqref{est:comp-strong-aneq0}, we deduce that $\varphi\in H^2_{0}(\mathbb{R}_+),\ \partial_Y\varphi\in H^1_{0}(\mathbb{R}_+)$.
This means that $\varphi$ is a  weak solution to \eqref{eq:OS-002}. By the solvability assumption (S-A), we get $\varphi=0.$

\item
If $\alpha=0,$ we define $\varphi(Y)=\int_{0}^{Y}u_1(Z)\mathrm{d}Z$.  By \eqref{est:comp-strong-a=0}, we get
\begin{align*}
   & \mathrm{i}\varepsilon_n \langle w^{(n)}, (\partial_Y^2-\alpha^2_n)\xi\rangle+ \langle Uw^{(n)},\xi\rangle\longrightarrow \mathrm{i}\varepsilon \langle \partial_Yu_1,\partial_Y^2\xi\rangle+ \langle U\partial_Yu_1,\xi\rangle.
\end{align*}
On the other hand, we notice that
\begin{align*}
    &\big\|U''\varphi^{(n)}-U''\int_{0}^{Y}u_1(Z)\mathrm{d}Z\big\|_{L^2} \leq C\bigg\|\dfrac{\varphi^{(n)}-\int_{0}^{Y}u_1(Z)\mathrm{d}Z}{Y}\bigg\|_{L^2} \\
   &\leq C\|\partial_Y\varphi^{(n)}-u_1\|_{L^2}\longrightarrow 0,
\end{align*}
which implies that
\begin{align*}
   & \langle U''\varphi^{n},\xi\rangle \longrightarrow \Big\langle U''\int_{0}^{Y}u_1(Z)\mathrm{d}Z,\xi\Big\rangle,\quad \langle f^{n},\xi\rangle \longrightarrow 0.
\end{align*}
Thus, we have $w=\partial_Y^2\varphi$, $\varphi\in H^2_0(\mathbb{R}_+)$, $\partial_Y\varphi\in H^1_0(\mathbb{R}_+)$ and
\begin{align*}
        &  \mathrm{i}\varepsilon \langle w, \partial_Y^2\xi\rangle+ \langle Uw,\xi\rangle-\langle U''\varphi,\xi\rangle=0.
     \end{align*}
Then $\varphi$ is a weak solution to \ref{eq:OS-002} with $\alpha=0$, and $\varphi=0$ by (S-A).
\end{itemize}

\no\textbf{Step 3. Contradiction.}

\begin{itemize}
  \item If $\alpha\neq0$, we have
  \begin{align*}
     & (1+\alpha^2_n)\|\varphi^{(n)}\|_{H^2}\geq C^{-1}\big(\|w^{(n)}\|_{L^2}+ \|(\partial_Y\varphi^{(n)},\alpha_n\varphi^{(n)})\|_{L^2}\big)\geq C^{-1},
  \end{align*}
  which along with  \eqref{est:comp-strong-aneq0} and $\varphi=0$ implies
  \begin{align*}
     &0=(1+\alpha^2)\|\varphi\|_{H^2}=\lim_{n}(1+\alpha^2_n)\|\varphi^{(n)}\|_{H^2} \geq C^{-1}>0,
  \end{align*}
 which leads to a contradiction.

  \item If $\alpha=0$, by \eqref{est:comp-strong-a=0} and $\varphi=0$,
  \begin{align*}
     &0=\|(\partial^2_Y\varphi,\partial_Y\varphi)\|_{L^2}=\lim_{n\rightarrow +\infty}\|(w^{(n)},\partial_Y\varphi^{(n)})\|_{L^2}\geq C^{-1}>0.
  \end{align*}
   This is a contradiction.
\end{itemize}

This finishes the proof of  the lemma.
\end{proof}

Next  we consider the case when $|\alpha|$ is large enough and $\varepsilon$ has lower bound.

\begin{lemma}\label{prop:epsilonbig-alphabig}
  Let $0<c_3\leq 1$ and $c_3\leq\varepsilon \leq 1$. Let $(1+Y)^2f\in L^2(\mathbb{R}_+)$ and $f\in H^1_0(\mathbb{R}_{+})$. Assume that $\varphi$ is the solution to   \eqref{eq:Orr-som-nonslip}. Then there exists $M_2>0$ sufficiently large, such that for any $|\alpha|\geq M_2$, it holds that
  \begin{align*}
     & \|w\|_{L^2}+\|(\partial_Y\varphi,\alpha \varphi)\|_{L^2}\leq C\|f\|_{L^2},
  \end{align*}
  where the constant $C$  depends on $c_3$.
\end{lemma}

\begin{proof}
  Notice that 
  \begin{align*}
     & \varepsilon\|(\partial_Y^2-\alpha^2)w\|_{L^2}\leq \|Uw\|_{L^2}+\|U''\varphi\|_{L^2}+ \|f\|_{L^2},
  \end{align*}
  which along with Lemma \ref{lemma:nonslip-elliptic} gives
  \begin{align*}
     &c_3\alpha^2\|w\|_{L^2}\leq C\varepsilon\|(\partial_Y^2-\alpha^2)w\|_{L^2}\leq C( \|w\|_{L^2}+\alpha^{-2}\|\alpha^2\varphi\|_{L^2}+\|f\|_{L^2}).
  \end{align*}
  Then we have
  \begin{align*}
     & c_3\alpha^2\|w\|_{L^2}\leq C\big(1+M_2^{-2}\big)\|w\|_{L^2}+C\|f\|_{L^2}.
  \end{align*}
 Choosing $M_2$ sufficiently large so that $c_3M_2^2\geq 2C\big(1+M_2^{-2}\big)$, we obtain
  \begin{align*}
     &  c_3\alpha^2\|w\|_{L^2}\leq C\|f\|_{L^2}.
  \end{align*}
This along with $\|(\partial_Y\varphi,\alpha\varphi)\|_{L^2}\leq |\alpha|^{-1}\|w\|_{L^2}$ gives our result.
\end{proof}

\begin{proposition}\label{prop:epsilonbig-H-1}
Let $\phi$ solve \eqref{eq:OS} with $(f_1,f_2)\in L^2(\mathbb{R}_{+}^2)$. For any fixed $0<c_3\leq 1$, if $c_3\leq\varepsilon \leq 1$, $\alpha\neq 0$, then we have
 \begin{align*}
    &|\alpha|\|(\partial_Y^2-\alpha^2)\phi\|_{L^2} +|\alpha|\|(\partial_Y\varphi,\alpha\varphi)\|_{L^\infty}+ |\alpha|\|(\partial_Y\phi,\alpha\phi)\|_{L^2} \leq C\|(f_1,f_2)\|_{L^2},
  \end{align*}
  where the constant $C$  depends on $c_3$.
\end{proposition}
\begin{proof}
  We decompose $\phi$ as $\phi=\phi_0+\phi_1$, where $\phi_0,\ \phi_1$ solve \eqref{eq:Orr-phi0} and \eqref{eq:Orr-phi1}.
  By Lemma \ref{prop:epsilonbig-alphasmall} and Lemma  \ref{prop:epsilonbig-alphabig}, we have
  \begin{align*}
      & \|(\partial_Y^2-\alpha^2)\phi_1\|_{L^2} +\|(\partial_Y\phi_1,\alpha\phi_1)\|_{L^\infty}+ \|(\partial_Y\phi_1,\alpha\phi_1)\|_{L^2}
   \leq C\|(1+Y)^2\partial_Y(U'\phi_0)\|_{L^2}.
   \end{align*}
  By Lemma \ref{prop:phi0} and the interpolation, we have
   \begin{align*}
      & \|(\partial_Y^2-\alpha^2)\phi_0\|_{L^2} +\|(\partial_Y\phi_0,\alpha\phi_0)\|_{L^\infty} +\|(\partial_Y\phi_0,\alpha\phi_0)\|_{L^2}\\
      &\leq C\big(\varepsilon^{\f13}\|(\partial_Y\phi_0,\alpha\phi_0)\|_{L^2}  +\varepsilon^{\f23}\|(\partial_Y^2-\alpha^2)\phi_0\|_{L^2}\big)\\
      &\leq C|\alpha|^{-1}\|(f_1,f_2)\|_{L^2}.
   \end{align*}
   Summing up the above two inequalities, we finish the proof.
\end{proof}

\section{Nonlinear stability}

This section is devoted to the proof of Theorem \ref{th:main}.

\subsection{Estimate for zero mode}
 In this case, the linearized system \eqref{eq:per-n-mode} can be written as
\begin{align*}
\left\{\begin{aligned}
&\nu\partial_y^2 u_{0,1}=f_{0,1},\quad u_{2,0}=0,\quad\partial_y p_0=f_{0,2},\\
&u_{0,1}|_{y=0}=0.
\end{aligned}\right.
\end{align*}

\begin{proposition}\label{th:linear}
Let $f_0\in L^1(\mathbb{R}_+)$ and $f_{0,1}=\partial_y F_{0,1}$ with $F_{0,1}\in L^1(\mathbb{R}_+)\cap L^2(\mathbb{R}_+)$. Then there exists a unique solution $u_0=(u_{0,1},0)$ to \eqref{eq:per-n-mode} with $\tilde{n}=0$ such that
\begin{align*}
\|u_{0,1}\|_{L^\infty}\leq \nu^{-1}\|F_{0,1}\|_{L^1},
\end{align*}
\begin{align*}
\|\partial_y u_{0,1}\|_{L^2}\leq\nu^{-1}\|F_{0,1}\|_{L^2}.
\end{align*}
Moreover, we have
\begin{align*}
\lim_{y\to+\infty} u_{0,1}=\nu^{-1}\int_0^{+\infty} F_{0,1}dy.
\end{align*}
\end{proposition}

\begin{proof}
Notice that $u_{0,1}$ can be represented as
\begin{align*}
u_{0,1}(y)=-\nu^{-1}\int_0^y\int_{y'}^{+\infty} f_{0,1}(y'')dy''dy'=\nu^{-1}\int_0^y F_{0,1}(y')dy',
\end{align*}
which along with $F_{0,1}\in L^1$ implies that $\|u_{0,1}\|_{L^\infty}\leq\nu^{-1}\|F_{0,1}\|_{L^1}$ and $\lim_{y\to+\infty} u_{0,1}=\nu^{-1}\int_0^{+\infty}F_{0,1}dy$. Moreover, since $\partial_y u_{0,1}(y)=\nu^{-1}F_{0,1}(y)$ and $F_{0,1}\in L^2$, we have $\|\partial_y u_{0,1}\|_{L^2}\leq\nu^{-1}\|F_{0,1}\|_{L^2}$.
\end{proof}

\subsection{Estimates for non-zero modes}

\begin{proposition}\label{thm:main-linear}
There exist positive number $\theta_0$, $\nu_0$ and $\delta_0$ so that the following statements hold.
For any $0<\nu\leq\nu_0$, $(0,\theta_0]\subset \Sigma(U,\nu)$. Moreover, for any $f_n\in L^2(\mathbb{R}_+)$, there holds that 
\begin{enumerate}
\item if $0<|\tilde{n}|\leq \delta_0\nu^{-\f34}$ and $\theta\in(0,\theta_0]$, then
\begin{align*}
|\tilde{n}|^{\f23}\|u_n\|_{L^2}+|\tilde{n}|^{\f13}\nu^{\f12}\|\partial_y u_n\|_{L^2}\leq C\|f_n\|_{L^2}.
\end{align*}
\item if $|\tilde{n}|\geq\delta_0\nu^{-\f34}$ and $\theta\in(0,\theta_0]$, then
\begin{align*}
|\tilde{n}|^{2}\nu\|u_n\|_{L^2}+|\tilde{n}|\nu\|\partial_y u_n\|_{L^2}\leq C\|f_n\|_{L^2}.
\end{align*}
\item if $\theta>\theta_0$ and $\theta\in\Sigma(U,\nu)$,
 \begin{align*}
|\tilde{n}|\|u_n\|_{L^2}+\nu^{\f12}|\tilde{n}|\|\partial_y u_n\|_{L^2}\leq C\|f_n\|_{L^2}.
\end{align*}
\end{enumerate}
\end{proposition}

\begin{proof}
Let $\delta_0\leq \delta_*$ and we take $\theta_0= c_*$, $\nu_1= \delta_*^{\f43}$, where $c_*$ and $\delta_*$ are small positive number in Proposition \ref{prop:os-nonslip-small}.
For any fixed $0<\nu\leq\nu_1$, the the estimates for the case of $\theta\in(0,\theta_0], 0<|\tilde{n}|\leq\delta_*\nu^{-\f34}$ or $\theta>\theta_0, \theta\in\Sigma(U)$ are deduced from the estimates for the Orr-Sommerfeld equation in Section 5. Indeed,  for each given $\tilde{n}\neq 0$, we have
\begin{align}\label{eq:relation-scale}
u_{n,1}=\partial_Y\phi(\frac{y}{\sqrt{\nu}}),\quad u_{n,2}=-\mathrm{i}\alpha\phi(\frac{y}{\sqrt{\nu}})\quad\mathrm{and}\quad f_n(y)=\nu^{-\f12}f(\frac{y}{\sqrt{\nu}}),
\end{align}
where $(\phi,f)$ satisfies \eqref{eq:OS}.

Notice that $\varepsilon|\alpha|^3\leq \delta_*$ and $\varepsilon\leq c_*$ when $\nu\leq\nu_1$, $\theta\leq\theta_0$ and $0< |\tilde{n}|\leq\delta_*\nu^{-\f34}$. Hence, by Proposition \ref{prop:os-nonslip-small} and \eqref{eq:relation-scale}, we deduce that for any $0<\nu\leq\nu_1$, $\theta\in(0,\theta_0]$ and $0< |\tilde{n}|\leq\delta_*\nu^{-\f34}$, there exists a unique solution $u_n\in H^2\cap H^1_0$. Moreover, there holds 
\begin{align*}
\|u_n\|_{L^2}\leq \nu^{\f14}\|(\partial_Y\phi,\alpha\phi)\|_{L^2}\leq C\nu^{\f14}|\tilde{n}|^{-\f23}\|f_n(\sqrt{\nu}\cdot)\|_{L^2}\leq C|\tilde{n}|^{-\f23}\|f_n\|_{L^2},
\end{align*}
and
\begin{align*}
\|\partial_yu_n\|_{L^2}+|\tilde{n}|\|u_n\|_{L^2}\leq \nu^{-\f14}\|(\partial_Y^2-\alpha^2)\phi\|_{L^2}\leq C\nu^{-\f14}|\tilde{n}|^{-\f13}\|f_n(\sqrt{\nu}\cdot)\|_{L^2}\leq \frac{C\|f_n\|_{L^2}}{|\tilde{n}|^{\f13}\nu^{\f12}}.
\end{align*}

For the case of $\theta>\theta_0$ and $\theta\in\Sigma(U,\nu)$, by Proposition \ref{prop:epsilonbig-H-1} and \eqref{eq:relation-scale}, we have
\begin{align*}
\|u_n\|_{L^2}\leq \nu^{\f14}\|(\partial_Y\phi,\alpha\phi)\|_{L^2}\leq C\nu^{\f14}|\tilde{n}|^{-1}\|f_n(\sqrt{\nu}\cdot)\|_{L^2}\leq C|\tilde{n}|^{-1}\|f_n\|_{L^2},
\end{align*}
and
\begin{align*}
\|\partial_y u_n\|_{L^2}\leq \nu^{-\f14}\|(\partial_Y-\alpha^2)\phi\|_{L^2}\leq C\nu^{-\f14}|\tilde{n}|^{-1}\|f_n(\sqrt{\nu}\cdot)\|_{L^2}\leq  C\nu^{-\f12}|\tilde{n}|^{-1}\|f_n\|_{L^2}.
\end{align*}

It remains to prove the second statement. Instead of considering the Orr-Sommerfeld equation, we are back to the original system \eqref{eq:per-n-mode}. Taking the inner product with $u_n$,  we obtain
\begin{align}\label{est:Origial-Re}
\nu\big(\|\partial_y u_n\|^2_{L^2}+\tilde{n}^2\|u_n\|_{L^2}^2\big)\leq \big|\mathbf{Re}\langle u_{n,2}\partial_y U(\frac{y}{\sqrt{\nu}}),u_{n,1}\rangle\big|+\|f_n\|_{L^2}\|u_n\|_{L^2},
\end{align}
and
\begin{align}\label{est:Origial-Im}
   & \tilde{n}\bigg(\Big\langle
    U(\frac{y}{\sqrt{\nu} })u_n,u_n\Big\rangle-\mathbf{Re}\Big\langle(\partial_yU(\frac{y}{\sqrt{\nu} }))\phi_n,\partial_y\phi_n\Big\rangle \bigg)=\mathbf{Im}\langle f_n,u_n\rangle.
\end{align}

We consider two cases: $(i)\ |\tilde{n}|\geq \delta_0^{-1}\nu^{-\f34}$ and $(ii)\ \delta_0\nu^{-\f34}\leq |\tilde{n}|\leq \delta_0^{-1}\nu^{-\f34}$ where we take $\delta_0= \min\big(\delta_*,\|\partial_YU\|_{L^\infty}^{-\f12}/2\big)$.

For the case $(i)$, we notice that $\delta_0= \min\big(\delta_*,\|\partial_YU\|_{L^\infty}^{-\f12}/2\big)$ Then for any $|\tilde{n}|\geq\delta_0^{-1}\nu^{-\f34}$, we have
\begin{align*}
\big|\mathbf{Re}\langle u_{n,2}\partial_y U(\frac{y}{\sqrt{\nu}}),u_{n,1}\rangle\big|\leq \nu^{-\f12}\|\partial_Y U\|_{L^\infty}\|u_n\|_{L^2}^2\leq \nu\tilde{n}^2\delta_0^2\|\partial_YU\|_{L^\infty}\|u_n\|_{L^2}^2\leq \frac{\nu\tilde{n}^2}{4}\|u_n\|_{L^2}^2,
\end{align*}
which along with \eqref{est:Origial-Re} gives
\begin{align*}
\nu(\|\partial_y u_n\|^2_{L^2}+\tilde{n}^2\|u_n\|_{L^2}^2)\leq C\|f_n\|_{L^2}\|u_n\|_{L^2}.
\end{align*}
This implies  that
\begin{align*}
|\tilde{n}|^{2}\nu\|u_n\|_{L^2}+|\tilde{n}|\nu\|\partial_y u_n\|_{L^2}\leq C\|f_n\|_{L^2}.
\end{align*}

For the case $(ii)$, by \eqref{est:Origial-Im}, we obtain
\begin{align}\label{est:Origial-Im-1}
   &\int_{0}^{+\infty}U(\frac{y}{\sqrt{\nu} })(|\partial_y\phi_n|^2+|\tilde{n}|^2|\phi_n|^2)\mathrm{d}y +\f12 \int_{0}^{+\infty}(\partial_y^2U(\frac{y}{\sqrt{\nu} }))|\phi_n|^2\mathrm{d}y=\dfrac{1}{\tilde{n}} \mathbf{Im}\langle f_n,u_n\rangle.
\end{align}
Let $\chi$ be a cut-off function such that $\chi(Y)=1$ for $0\leq Y\leq 1$ and $\chi(Y)=0$ for $Y\geq 2$. Let $Y_3>0$ such that $\partial_YU>0$ for $Y\in[0,4Y_3]$. Then $|\partial^2_yU(\frac{y}{\sqrt{\nu} })|\leq C\nu^{-\f12}\partial_yU(\frac{y}{\sqrt{\nu} })$ for $0\leq Y\leq 2Y_3\nu^{\f12}$, which gives
\begin{align*}
   \left|\int_{0}^{+\infty}(\partial_y^2U(\frac{y}{\sqrt{\nu} }))|\phi_n|^2\mathrm{d}y \right|\leq& C\nu^{-\f12}\int_{0}^{+\infty}(\partial_yU(\frac{y}{\sqrt{\nu} }))\chi\big(\dfrac{ y}{2\sqrt{\nu}Y_3}\big)|\phi_n|^2\mathrm{d}y +C\nu^{-1}\int_{Y_3\nu^{\f12}}^{+\infty}|\phi_n|^2\mathrm{d}y\\
   =& -C\nu^{-\f12}\int_{0}^{+\infty}U(\frac{y}{\sqrt{\nu} })\partial_y\big(\chi\big(\dfrac{ y}{2\sqrt{\nu}Y_3}\big|\phi_n|^2\big)\big)\mathrm{d}y +C\nu^{-1}\int_{Y_3\nu^{\f12}}^{+\infty}|\phi_n|^2\mathrm{d}y\\
   \leq &C\nu^{-\f12}\int_{0}^{+\infty}|U(\frac{y}{\sqrt{\nu} })|\partial_y\phi_n|\phi_n|\mathrm{d}y +C\nu^{-1}\int_{Y_3\nu^{\f12}}^{+\infty}|\phi_n|^2\mathrm{d}y\\
   \leq & \dfrac{C}{\nu^{\f12}|\tilde{n}|}\Big\|\sqrt{U(\frac{y}{\sqrt{\nu} })}\tilde{n}\phi_n\Big\|_{L^2} \Big\|\sqrt{U(\frac{y}{\sqrt{\nu} })}\partial_y\phi_n\Big\|_{L^2} + \dfrac{C}{\nu|\tilde{n}|^2}\Big\|\sqrt{U(\frac{y}{\sqrt{\nu} })}\tilde{n}\phi_n\Big\|_{L^2}^2\\
   \leq & \dfrac{C\nu^{\f14}}{\delta_0}\Big\|\sqrt{U(\frac{y}{\sqrt{\nu} })}\tilde{n}\phi_n\Big\|_{L^2} \Big\|\sqrt{U(\frac{y}{\sqrt{\nu} })}\partial_y\phi_n\Big\|_{L^2} +\dfrac{C\nu^{\f12}}{\delta_0^2}\Big\|\sqrt{U(\frac{y}{\sqrt{\nu} })}\tilde{n}\phi_n\Big\|_{L^2}^2.
\end{align*}
Take $\nu_0= \min\big(\nu_1, (C^{-1}\delta_0/4)^4\big)$. If $0<\nu\leq \nu_0$, we get by \eqref{est:Origial-Im-1}  that
\begin{align}\label{est:Origial-Im-2}
   \Big\|\sqrt{U(\frac{y}{\sqrt{\nu} })}u_n\Big\|_{L^2}^2&=\Big\|\sqrt{U(\frac{y}{\sqrt{\nu} })}\partial_y\phi_n\Big\|_{L^2}^2+|\tilde{n}|^2\Big\|\sqrt{U(\frac{y}{\sqrt{\nu} })}\phi_n\Big\|_{L^2}^2 \\&\leq C|\tilde{n}|^{-1}\|f_n\|_{L^2}\|u_n\|_{L^2}.\nonumber
\end{align}

Using the fact that $\mathrm{i}\tilde{n} u_{n,1}+\partial_y u_{n,2}=0$ and Hardy's inequality,  we can deduce from \eqref{est:Origial-Re}  that
\begin{align}\label{est:Origial-Re-1}
   \nu(\|\partial_yu_n\|_{L^2}^2+\tilde{n}^2\|u_n\|_{L^2}^2)&\leq C\nu^{-\f14}|\tilde{n}|^{\f12}\|u_n\|_{L^2}\Big\|\sqrt{U(\frac{y}{\sqrt{\nu} })}u_n\Big\|_{L^2} +\|f_n\|_{L^2}\|u_n\|_{L^2}.
\end{align}
As in the proof of  \eqref{est:interpolation-1}, we have
\begin{align}\label{est:interpolation-2}
  &\|u_n\|_{L^2}\leq C\nu^{\f14}|\tilde{n}|^{-\f16}\|\partial_yu_n\|_{L^2}^{\f12}\|u_n\|_{L^2}^{\f12} +C|\tilde{n}|^{\f16}\Big\|\sqrt{U(\frac{y}{\sqrt{\nu} })}u_n\Big\|_{L^2}.
\end{align}
Plugging \eqref{est:Origial-Re-1} into \eqref{est:interpolation-2}, we get
\begin{align*}
  \|u_n\|_{L^2} \leq&  C|\tilde{n}|^{-\f16}\|f_n\|_{L^2}^{\f14}\|u_n\|_{L^2}^{\f34} +C|\tilde{n}|^{\f16}\Big\|\sqrt{U(\frac{y}{\sqrt{\nu} })}u_n\Big\|_{L^2}\\
  &+\nu^{-\f{1}{16}}|\tilde{n}|^{-\f{1}{24}}\Big\|\sqrt{U(\frac{y}{\sqrt{\nu} })}u_n\Big\|_{L^2}^{\f14}\|u_n\|_{L^2}^{\f34},
\end{align*}
which gives
\begin{align*}
   \|u_n\|_{L^2} \leq&  C|\tilde{n}|^{-\f23}\|f_n\|_{L^2} +C(|\tilde{n}|^{\f16}+\nu^{-\f{1}{4}}|\tilde{n}|^{-\f{1}{6}})\Big\|\sqrt{U(\frac{y}{\sqrt{\nu} })}u_n\Big\|_{L^2}.
\end{align*}
This along with \eqref{est:Origial-Im-2} shows
\begin{align*}
   \|u_n\|_{L^2}&\leq C(|\tilde{n}|^{-\f23}+ \nu^{-\f12}|\tilde{n}|^{-\f43})\|f_n\|_{L^2}\leq C|\tilde{n}|^{-2}\nu^{-1}\|f_n\|_{L^2}.
\end{align*}
where in the last line we used $|\tilde{n}|\sim \nu^{-\f34}$ in the case $(ii)$.
Putting this inequality into \eqref{est:Origial-Im-2} and \eqref{est:Origial-Re-1}, and using the fact that $|\tilde{n}|\sim \nu^{-\f34}$,  we conclude that 
\begin{align*}
   & |\tilde{n}|^2\nu\|u_n\|_{L^2}+|\tilde{n}|\nu\|\partial_yu_n\|_{L^2}\leq C\|f_n\|_{L^2}.
\end{align*}

The existence of the solution can be proved by using the method of continuity. After replacing $-\nu(\partial_y^2-\tilde{n}^2)$ by $-\nu(\partial_y^2-\tilde{n}^2)+l$ with $l>0$ ,
it is easy to show that there exists a unique solution for any $f_n\in L^2$ and $l$ large enough, and the above priori estimates still hold true for any $l>0$.
\end{proof}

\subsection{Proof of Theorem \ref{th:main}}
With the estimates for the linearized system, the  proof of nonlinear stability is similar to \cite{GM}. For the completeness, we present a sketch.

We firs introduce the functional space
\begin{align*}
X_{\nu,\varepsilon}:=\big\{u:\|u\|_{X_\nu}\leq\varepsilon\nu^{\f12}|\log\nu|^{-\f12}\big\},
\end{align*}
where
\begin{align*}
\|u\|_{X_\nu}=\|u_{0,1}\|_{L^\infty}+\nu^{\f14}\|\partial_y u_{0,1}\|_{L^2}+\sum_{n\neq0}\|u_n\|_{L^\infty}+\nu^{-\f14}\|\mathcal{Q}_0u\|_{L^2}+\nu^{\f14}\|\nabla\mathcal{Q}_0 u\|_{L^2}.
\end{align*}

For $v\in X_{\nu,\varepsilon}$, we define the map $\Psi[v]=u$ as the solution to the system
\begin{align*}
\left\{\begin{aligned}
&U(\frac{y}{\sqrt{\nu}})\partial_x u+(u_2\partial_yU(\frac{y}{\sqrt{\nu}}),0)-\nu\Delta u+\nabla p=-v \cdot\nabla v+f^\nu,\\
&\nabla\cdot u=0,\\
&u|_{y=0}=0.
\end{aligned}\right.
\end{align*}

Notice that
\begin{align*}
-v\cdot\nabla v=-v_{0,1}\partial_x\mathcal{Q}_0 v-(\mathcal{Q}_0 v_2\partial_y v_{0,1},0)-\mathcal{Q}_0 v\cdot\nabla \mathcal{Q}_0 v,
\end{align*}
which implies 
\begin{align*}
\mathcal{P}_0\big(-v\cdot\nabla v+f^\nu)=\partial_y\big(-\mathcal{P}_0(\mathcal{Q}_0v\cdot\nabla\mathcal{Q}_0 v)\big).
\end{align*}
Then we get by Proposition \ref{th:linear} that 
\begin{align}
\|u_{0,1}\|_{L^\infty}\leq \nu^{-1}\|\mathcal{Q}_0 v\|_{L^2}^2,\quad\|\partial_y u_{0,1}\|_{L^2}\leq\nu^{-1}\|\mathcal{Q}_0 v\|_{L^\infty}\|\mathcal{Q}_0v\|_{L^2}.
\end{align}

For non-zero modes, we have
\begin{align*}
\mathcal{Q}_0\big(-v\cdot\nabla v+f^\nu)=-v_{0,1}\partial_x\mathcal{Q}_0 v-(\mathcal{Q}_0 v_2\partial_y v_{0,1},0)-\mathcal{Q}_0(\mathcal{Q}_0 v\cdot\nabla \mathcal{Q}_0 v)+f^\nu.
\end{align*}
Proposition \ref{thm:main-linear} implies that
\begin{eqnarray}\label{eq:nonlinear-1}
\begin{split}
\|\mathcal{Q}_0 u\|_{L^2}\leq& C\big(\|v_{0,1}\|_{L^\infty}\|\nabla\mathcal{Q}_0 v\|_{L^2}
+\|\mathcal{Q}_0 v\|_{L^\infty}\|\partial_y v_{0,1}\|_{L^2}\\
&\quad+\|\mathcal{Q}_0 v\|_{L^\infty}\|\nabla\mathcal{Q}_0 v\|_{L^2}+\|f^\nu\|_{L^2}\big),
\end{split}
\end{eqnarray}
and
\begin{eqnarray}\label{eq:nonlinar-2}
\begin{split}
\|\nabla\mathcal{Q}_0 u\|_{L^2}\leq& C\nu^{-\f12}\big(\|v_{0,1}\|_{L^\infty}\|\nabla\mathcal{Q}_0 v\|_{L^2}+\|\mathcal{Q}_0v\|_{L^\infty}\|\partial_y v_{0,1}\|_{L^2}\\
&\quad+\|\mathcal{Q}_0 v\|_{L^\infty}\|\nabla\mathcal{Q}_0 v\|_{L^2}+\|f^\nu\|_{L^2}\big).
\end{split}
\end{eqnarray}
For the case of  $n\neq0$,  notice that
\begin{align*}
	\mathcal{P}_n(-v\cdot\nabla v+f^\nu)=-v_{0,1}\partial_x\mathcal{P}_n v-(\mathcal{P}_n v_2\partial_y v_{0,1},0)-\mathcal{P}_n(\mathcal{Q}_0v\cdot\nabla\mathcal{Q}_0 v)+\mathcal{P}_n f^\nu.
\end{align*}
Then it  follows from Proposition \ref{thm:main-linear} and the interpoaltion that  for $0<|\tilde{n}|=n/\theta \leq\delta_0\nu^{-\f34}$,
\begin{align*}
	\|\mathcal{P}_nu\|_{L^{\infty}}\leq& C\nu^{-\f14}|\tilde{n}|^{-\f12}\big(\|v_{0,1}\|_{L^\infty}\|\nabla\mathcal{P}_n v\|_{L^2}+\|\mathcal{P}_n v\|_{L^\infty}\|\partial_y v_{0,1}\|_{L^2}\\
	&+\|\mathcal{P}_n(\mathcal{Q}_0v\cdot\nabla\mathcal{Q}_0 v)\|_{L^2}+\|\mathcal{P}_n f^\nu\|_{L^2}\big),
\end{align*}
and for $|\tilde{n}|\geq \delta_0\nu^{-\f34}$,
\begin{align*}
	\|\mathcal{P}_nu\|_{L^{\infty}}\leq& C\nu^{-1}|\tilde{n}|^{-\f32}\big(\|v_{0,1}\|_{L^\infty}\|\nabla\mathcal{P}_n v\|_{L^2}+\|\mathcal{P}_n v\|_{L^\infty}\|\partial_y v_{0,1}\|_{L^2}\\
	&+\|\mathcal{P}_n(\mathcal{Q}_0v\cdot\nabla\mathcal{Q}_0 v)\|_{L^2}+\|\mathcal{P}_n f^\nu\|_{L^2}\big),
\end{align*}
which imply by the Parseval's equality that
\begin{eqnarray}\label{eq:nonliear-3}
\begin{split}
\sum_{n\neq0}\|\mathcal{P}_n u\|_{L^\infty}\leq& C\nu^{-\f14}\big(|\log\nu|^{\f12}\|v_{0,1}\|_{L^\infty}\|\nabla\mathcal{Q}_0 v\|_{L^2}+\sum_{n\neq0}\|\mathcal{P}_n v\|_{L^\infty}\|\partial_y v_{0,1}\|_{L^2}\\
&+|\log\nu|^{\f12}\|\mathcal{Q}_0 v\|_{L^\infty}\|\nabla\mathcal{Q}_0 v\|_{L^2}+|\log\nu|^{\f12}\|f^\nu\|_{L^2}\big).
\end{split}
\end{eqnarray}

By collecting \eqref{eq:nonlinear-1} to \eqref{eq:nonliear-3}, we arrive at
\begin{align*}
\|\Psi[v]\|_{X_\nu}\leq C\nu^{-\f12}|\log\nu|^{\f12}\|v\|^2_{X_\nu}+C\nu^{-\f14}|\log\nu|^{\f12}
\|f{\color{red}^{\nu}} \|_{L^2},
\end{align*}
and
\begin{align*}
\|\Psi[v]-\Psi[v']\|_{X_\nu}\leq C\nu^{-\f12}|\log\nu|^{\f12}(\|v\|_{X_\nu}+\|v'\|_{X_\nu})\|v-v'\|_{X_\nu}.
\end{align*}
Therefore, $\Psi$ is a contraction from $X_{\nu,\varepsilon}$ to itself if $\varepsilon$ and $\|f^\nu\|_{L^2}$ are small enough. Hence, by the fixed point theorem, for any $f^\nu\in L^2$ with $\|f^\nu\|_{L^2}\leq \varepsilon |\log\nu|^{-1}\nu^{\f34}$, there exists a unique solution $u^\nu$ to \eqref{eq:perturbation-NS} in $X_{\nu,\varepsilon}$. Using the elliptic regularity of the Stokes equation, we obtain $\nabla ^2 u^\nu\in L^2(\Omega_\theta)$. The proof of Theorem \ref{th:main} is completed.

\appendix

\section{Some basic inequalities}

 \begin{lemma}\label{lem:ham-bound}
  If $w\in L^2(\mathbb{R}_+)\cap L^1(\mathbb{R}_+;\mathrm{e}^{\alpha Y}\mathrm{d}Y)$, and $\phi\in H^1_0\cap H^2(\mathbb{R}_+)$ satisfies
  \begin{align*}
     & (\partial_Y^2-\alpha^2)\phi=w,\ \alpha > 0,
  \end{align*}
  then it holds that
  \begin{align*}
     &\partial_Y\phi(Y)=e^{-\alpha Y}\int_0^{Y}w(Z)\sinh(\alpha Z)\mathrm{d} Z-\cosh(\alpha Y)\int_{Y}^{+\infty}w(Z)\mathrm{e}^{-\alpha Z}\mathrm{d}Z,\\
     &\alpha \phi(Y)=-\sinh(\alpha Y)\int_{0}^{+\infty} w(Z)\mathrm{e}^{-\alpha Z}\mathrm{d}Z-\mathrm{e}^{-\alpha Y}\int_{0}^{Y}w(Z)\sinh(\alpha Z)\mathrm{d}Z\\
     &\qquad\quad=-\mathrm{e}^{-\alpha Y}\int_{0}^{+\infty}w(Z)\sinh(\alpha Z)\mathrm{d}Z+ \int_{Y}^{+\infty}w(Z)\sinh\big(\alpha(Z-Y)\big)\mathrm{d}Z.
  \end{align*}
Specially, we have
\begin{align*}
   & \partial_Y\phi(0)=-\int_{0}^{+\infty}w(Y)\mathrm{e}^{-\alpha Y}\mathrm{d}Y.
\end{align*}
\end{lemma}
\begin{proof}
 Integration by parts gives
  \begin{align*}
     & \int_{Y}^{+\infty}w(Z)\mathrm{e}^{-\alpha Z}\mathrm{d}Z= \int_{Y}^{+\infty}\big((\partial_Z^2-\alpha^2)\phi(Z)\big)\mathrm{e}^{-\alpha Z}\mathrm{d}Z\\
     &=\int_{Y}^{+\infty}\phi\big((\partial_Z^2-\alpha^2)\mathrm{e}^{-\alpha Z}\big)\mathrm{d}Z+\big((\partial_Z\phi)\mathrm{e}^{-\alpha Z}\big)\big|_{Y}^{+\infty}-\big(\phi\partial_Z\mathrm{e}^{-\alpha Z}\big)\big|_{Y}^{+\infty}\\
     &=-\partial_Y\phi(Y)\mathrm{e}^{-\alpha Y}-\alpha \phi(Y)\mathrm{e}^{-\alpha Y}.
  \end{align*}
  That is,
  \begin{align}\label{est:ham-bound-1}
     \partial_Y\phi(Y)\mathrm{e}^{-\alpha Y}+\alpha \phi(Y)\mathrm{e}^{-\alpha Y}&=-\int_{Y}^{+\infty}w(Z)\mathrm{e}^{-\alpha Z}\mathrm{d}Z.
  \end{align}
  Specially, we have
  \begin{align}\label{est:ham-bound-2}
     & \partial_Y\phi(0)=-\int_{0}^{+\infty}w(Y)\mathrm{e}^{-\alpha Y}\mathrm{d}Y.
  \end{align}
  
  Integration by parts again gives
  \begin{align*}
     & \int_{0}^{Y}w(Z)\mathrm{e}^{\alpha Z}\mathrm{d}Z= \int_{0}^{Y}\big((\partial_Z^2-\alpha^2)\phi(Z)\big)\mathrm{e}^{\alpha Z}\mathrm{d}Z\\
     &= \big(\partial_Z\phi(Z)\mathrm{e}^{\alpha Z}\big)\big|_{0}^{Y}-\big(\phi(Z)\partial_Z(\mathrm{e}^{\alpha Z})\big)\big|_{0}^{Y}=\partial_Y\phi(Y)\mathrm{e}^{\alpha Y}-\alpha \phi(Y)\mathrm{e}^{\alpha Y}-\partial_Y\phi(0).
  \end{align*}
  Along with \eqref{est:ham-bound-1} and \eqref{est:ham-bound-2}, we obtain
  \begin{align*}
     \partial_Y\phi(Y)&=\dfrac{1}{2} \mathrm{e}^{-\alpha Y}\int_{0}^{Y}w(Z)\mathrm{e}^{\alpha Z}\mathrm{d}Z-\dfrac{1}{2}\mathrm{e}^{\alpha Y} \int_{Y}^{+\infty}w(Z)\mathrm{e}^{-\alpha Z}\mathrm{d}Z+\dfrac{1}{2}\mathrm{e}^{-\alpha Y}\partial_Y\phi(0)\\
     &=\mathrm{e}^{-\alpha Y}\int_{0}^{Y}w(Z)\sinh(\alpha Z)\mathrm{d}Z-\cosh(\alpha Y)\int_{Y}^{+\infty}w(Z)\mathrm{e}^{-\alpha Z}\mathrm{d}Z.
  \end{align*}
  Similarly, we have
  \begin{align*}
     \alpha\phi(Y)&=-\dfrac{1}{2} \mathrm{e}^{-\alpha Y}\int_{0}^{Y}w(Z)\mathrm{e}^{\alpha Z}\mathrm{d}Z-\dfrac{1}{2}\mathrm{e}^{\alpha Y} \int_{Y}^{+\infty}w(Z)\mathrm{e}^{-\alpha Z}\mathrm{d}Z-\dfrac{1}{2}\mathrm{e}^{-\alpha Y}\partial_Y\phi(0)\\
     &=-\mathrm{e}^{-\alpha Y}\int_{0}^{Y}w(Z)\sinh(\alpha Z)\mathrm{d}Z-\sinh(\alpha Y)\int_{Y}^{+\infty}w(Z)\mathrm{e}^{-\alpha Z}\mathrm{d}Z\\
     &=-\mathrm{e}^{-\alpha Y}\int_{0}^{+\infty}w(Z)\sinh(\alpha Z)\mathrm{d}Z+ \int_{Y}^{+\infty}w(Z)\sinh\big(\alpha(Z-Y)\big)\mathrm{d}Z.
  \end{align*}
  
\end{proof}

\begin{lemma}\label{lemma:nonslip-elliptic}
Let $\alpha\neq0$. Suppose that $(\partial_Y^2-\alpha^2)w\in L^2(\mathbb{R}_+)$ with $(\partial_Y^2-\alpha^2)\varphi=w,\ \partial_Y\varphi|_{Y=0}=\varphi|_{Y}=0$. Then it holds that
  \begin{align*}
     &|\alpha|^3\|(\partial_Y\varphi,\alpha\varphi)\|_{L^2}+ \alpha^2\|w\|_{L^2}+|\alpha|\|\partial_Yw\|_{L^2}+\|\partial_Y^2w\|_{L^2}\leq C\|(\partial_Y^2-\alpha^2)w\|_{L^2},
  \end{align*}
and
  \begin{align*}
     &\|(\partial_Yw,\alpha w)\|_{L^2}\leq C \|(\partial_Y^2-\alpha^2)w\|_{L^2}^{\f12}\|w\|_{L^2}^{\f12}.
  \end{align*}
\end{lemma}

\begin{proof}
Without loss of generality, we may assume $\alpha>0$. Clearly, $\mathrm{e}^{-\alpha Y}\in H^2(\mathbb{R}_+)$, and $(\partial_Y^2-\alpha^2)\mathrm{e}^{-\alpha Y}=0$.
   Let $(w_{d},\varphi_{d})$ be the solution to the following elliptic equation
   \begin{align*}\left\{\begin{aligned}
     &(\partial_Y^2-\alpha^2)w_{d}=(\partial_Y^2-\alpha^2)w,\\
     & (\partial_Y^2-\alpha^2)\varphi_d=w_d,\ w_d|_{Y=0}=\varphi_d|_{Y}=0.
   \end{aligned}\right.\end{align*}
Using the energy method and the interpolation, it is easy to show that
   \begin{align}\label{est:nonslip-elliptic-1}
      &|\alpha|^{\f52}\|\partial_Y\varphi_d\|_{L^\infty}+ |\alpha|^3\|(\partial_Y\varphi_d,\alpha\varphi_d)\|_{L^2}+ \alpha^2\|w_d\|_{L^2}+|\alpha|\|\partial_Yw_d\|_{L^2}\nonumber\\
      &\quad+\|\partial_Y^2w_d\|_{L^2}\leq C\|(\partial_Y^2-\alpha^2)w\|_{L^2}.
   \end{align}
On the other hand, by  Lemma \ref{lem:ham-bound}, we know that
  \begin{align*}
     & \partial_Y(\partial_Y^2-\alpha^2)^{-1}(\mathrm{e}^{-\alpha Y})|_{Y=0}=-1/(2\alpha).
  \end{align*}
  Then $w=w_d+2\alpha \partial_Y\varphi_d(0) \mathrm{e}^{-\alpha Y}$, which along with \eqref{est:nonslip-elliptic-1} implies that
  \begin{align*}
     & |\alpha|^3\|(\partial_Y\varphi,\alpha\varphi)\|_{L^2}+ \alpha^2\|w\|_{L^2}+|\alpha|\|\partial_Yw\|_{L^2}+\|\partial_Y^2w\|_{L^2}\leq C\|(\partial_Y^2-\alpha^2)w\|_{L^2}.
  \end{align*}
  This gives the first inequality of the lemma.

For the second inequality, we first notice that
  \begin{align*}
     &\langle (\partial_Y^2-\alpha^2)w,-w\rangle =\|(\partial_Yw,\alpha w)\|_{L^2}^2+\partial_Yw(0)\overline{w(0)},\\
     &\|(\partial_Yw,\alpha w)\|_{L^2}^2\leq |\langle(\partial_Y^2-\alpha^2)w,w\rangle|+ \|\partial_Yw\|_{L^\infty}\|w\|_{L^\infty}.
  \end{align*}
  Thanks to $\partial_Yw,w \in L^2$,  we have $\lim_{Y\rightarrow+\infty} w(Y)=\lim_{Y\rightarrow+\infty}\partial_Yw(Y)=0$, hence,
  \begin{align*}
     & \|\partial_Yw\|_{L^\infty}^2\leq 2\|\partial_Y^2w\|_{L^2}\|\partial_Yw\|_{L^2},\quad \|w\|_{L^\infty}^2\leq 2\|\partial_Yw\|_{L^2}\|w\|_{L^2}.
  \end{align*}
  Summing up, we obtain
  \begin{align*}
     \|(\partial_Yw,\alpha w)\|_{L^2}^2\leq & |\langle(\partial_Y^2-\alpha^2)w,w\rangle| +2\|\partial_Y^2w\|_{L^2}^{\f12}\|\partial_Yw\|_{L^2}\|w\|_{L^2}^{\f12}\\
     \leq & \|(\partial_Y^2-\alpha^2)w\|_{L^2}\|w\|_{L^2} +2\|\partial_Y^2w\|_{L^2}^{\f12}\|\partial_Yw\|_{L^2}\|w\|_{L^2}^{\f12},
  \end{align*}
  which implies that
  \begin{align*}
    \|(\partial_Yw,\alpha w)\|_{L^2}^2\leq C\|(\partial_Y^2-\alpha^2)w\|_{L^2}\|w\|_{L^2}.
  \end{align*}
\end{proof}

\begin{lemma}\label{lem:Uineq}
  There exists a positive constant $C>0$, such that for any $z\in\mathbb{C},\ t>0$, it holds that
  \begin{align*}
     & \int_{0}^{t}|z-s|^{\f12}\mathrm{d}s\geq C^{-1}|z|^{\f12}t.
  \end{align*}
\end{lemma}

\begin{proof}
  Let $z_r=\mathbf{Re}(z),\ z_i=\mathbf{Im}(z)$. Let us first claim that
  \begin{align}\label{est:Uineq Re}
     & \int_{0}^{t}|z_r-s|^{\f12}\mathrm{d}s\geq C^{-1}|z_r|^{\f12}t.
  \end{align}
  Once \eqref{est:Uineq Re} holds,  we have
  \begin{align*}
     & \int_{0}^{t}|z-s|^{\f12}\mathrm{d}s\geq C^{-1} \big( \int_{0}^{t}|z_r-s|^{\f12}\mathrm{d}s+ \int_{0}^{t}|z_i|^{\f12}\mathrm{d}s\big)\geq C^{-1}\big(|z_r|^{\f12}t+|z_i|^{\f12}t\big)\geq C^{-1}|z|^{\f12}t.
  \end{align*}
  It remains to prove \eqref{est:Uineq Re}.\smallskip

 \no \textbf{Case 1}. $z_r\leq 0$.  In this case, we have
  $$\int_{0}^{t}|z_r-s|^{\f12}\mathrm{d}s\geq \int_{0}^{t}|z_r|^{\f12}\mathrm{d}s= |z_r|^{\f12}t.$$

\no \textbf{Case 2}. $0\leq z_r\leq t/2$. In this case, we have
  \begin{align*}
     \int_{0}^{t}|z_r-s|^{\f12}\mathrm{d}s&\geq \int_{t/2}^{t}|z_r-s|^{\f12}\mathrm{d}s=  \int_{t/2}^{t}(s-z_r)^{\f12}\mathrm{d}s\geq \int_{t/2}^{t}(s-t/2)^{\f12}\mathrm{d}s\\
     &= \dfrac{2(t-t/2)^{\f32}}{3}=\dfrac{2(t/2)^{\f32}}{3}\geq \dfrac{z_r^{\f12}t}{3}=\dfrac{|z_r|^{\f12}t}{3}.
  \end{align*}

 \no \textbf{Case 3}. $z_r\geq t/2$. In this case, we have
  \begin{align*}
      \int_{0}^{t}|z_r-s|^{\f12}\mathrm{d}s&\geq  \int_{0}^{t/4}|z_r-s|^{\f12}\mathrm{d}s\geq  \int_{0}^{t/4}(z_r-t/4)^{\f12}\mathrm{d}s\\
      &=\dfrac{(z_r-t/4)^{\f12}t}{4}\geq \dfrac{|z_r/2|^{\f12}t}{4},
  \end{align*}
  here we used $z_r-t/4\geq z_r/2=|z_r|/2$. Combining three cases, we conclude our result.
\end{proof}

 The following Hardy's type inequalities come  from \cite{GM}.

\begin{lemma}\label{lem:GM's-hardy}
  (1).Let $\sigma[\cdot]$ be a  linear operator defined by
  \begin{align*}
     &\sigma[f](Y)=\int_{Y}^{+\infty}f(Y_1)\mathrm{d}Y_1,\qquad f\in \mathcal{C}_{0}^{\infty}(\overline{\mathbb{R}_+}).
  \end{align*}
  Then for $1\leq p\leq +\infty$ and $k=0,1,...,$ we have
  \begin{align*}
     &\big\|Y^{k}\sigma[f]\big\|_{L^p_{Y}}\leq C_p\big\|Y^{k+1}f\big\|_{L^p_Y}.
  \end{align*}
  (2).Let $\mathcal{L}[\cdot]$ be a linear operator defined by
  \begin{align*}
     &\mathcal{L}[f](Y)=U(Y)\int_{Y}^{+\infty}\dfrac{f}{U^2}\mathrm{d}Y_1,\qquad f\in \mathcal{C}_{0}^{\infty}(\overline{\mathbb{R}_+}).
  \end{align*}
  Then for $1< p< +\infty$ and $k=0,1,...,$ we have
  \begin{align*}
     &\big\|Y^{k}\mathcal{L}[f]\big\|_{L^p_{Y}}\leq C\big\|Y^{k}(1+Y)f\big\|_{L^p_Y},\\
     &\|\partial_Y\mathcal{L}[f]\|_{L^2_{Y}}\leq C\big(\|f\|_{L^1_{Y}}+\|f\|_{L^2_Y}+\|\partial_Yf\|_{L^2_Y}\big).
  \end{align*}
\end{lemma}

\section{Some estimates of Airy function}\label{Sec:Airy}

 Let $Ai(y)$ be the Airy function, which is a nontrivial solution of $f''-yf=0$.  We denote
\begin{align*}
&A_0(z)=\int_{\mathrm{e}^{{\mathrm{i}}\pi/6}z}^{+\infty}Ai(t)\mathrm{d}t =\mathrm{e}^{{i}\pi/6}\int_{z}^{+\infty}Ai(\mathrm{e}^{{\mathrm{i}}\pi/6}t)\mathrm{d}t.
\end{align*}

The following lemma comes from \cite{CLWZ}.

\begin{lemma}\label{lem:Airy-p1}
There exists $c>0$ and $\delta_0>0$ so that for $\textbf{Im}(z)\le \delta_0$,
\begin{align}
&\left|\f{A_0'(z)}{A_0(z)}\right|\lesssim1+|z|^{\f12},\quad {\rm Re}\f{A_0'(z)}{A_0(z)}\leq\min(-1/3,-c(1+|z|^{\f12})).
\end{align}
Moreover, for ${{\bf Im }z}\le \delta_0$, we have
\beno
 \Big|\f{A_0''(z)}{A_0(z)}\Big|\le C(1+|z|).
 \eeno
\end{lemma}

We denote
\beno
\tilde{A}(Y)=Ai(e^{\mathrm{i}\frac{\pi}{6}}\kappa(Y+\eta))/Ai(e^{\mathrm{i}\frac{\pi}{6}}\kappa\eta),
\eeno
 where $\kappa>0$ and $\mathbf{Im}\eta<0$.  We define $\tilde{\Phi}(Y)$ as the solution of
 \beno
 (\partial_Y^2-\alpha^2)\tilde{\Phi}=\tilde{A},\quad \tilde{\Phi}(0)=0.
 \eeno
Under the assumption $\tilde{A}\in L^2(\mathbb{R}_+)\cap L^1(\mathbb{R}_+;\mathrm{e}^{\alpha Y})$, by Lemma \ref{lem:ham-bound}, we know that
\begin{align*}
   \tilde{\Phi}(Y)&= -\dfrac{\mathrm{e}^{-\alpha Y}}{\alpha}
   \int_{0}^{+\infty}\tilde{A}(Z)\sinh(\alpha Z)\mathrm{d}Z+ \dfrac{1}{\alpha}\int_{Y}^{+\infty}\tilde{A}(Z)\sinh\big(\alpha(Z-Y)\big)\mathrm{d}Z.
\end{align*}
We define the fast decay part of $\tilde{\Phi}(Y)$ as 
\begin{align*}
   &  \tilde{\Phi}_{f}(Y)= \dfrac{1}{\alpha}\int_{Y}^{+\infty}\tilde{A}(Z)\sinh\big(\alpha(Z-Y)\big)\mathrm{d}Z.
\end{align*}
Then $(\partial_Y^2-\alpha^2)\tilde{\Phi}_f=\tilde{A}$.

\begin{lemma}\label{lem:Airy-w}
Let $\kappa>0$ and $\mathbf{Im}\eta< 0$. Then there exists $c>0$ such that
\begin{align*}
&|\tilde{A}(Y)|\leq C\mathrm{e}^{-c\kappa Y\big(1+|\kappa\eta|^{\f12}\big)},\\
&\|Y^\beta\tilde{A}\|_{L^2}\leq C\kappa^{-\f12-\beta}\big(1+|\kappa\eta|)^{-\f14-\f \beta 2},\quad \beta\ge 0,\\
&\|(\partial_Y\tilde{\Phi},\alpha\tilde{\Phi})\|_{L^2}\leq C\kappa^{-\f32}\big(1+|\kappa\eta|)^{-\f34}.
\end{align*}
Moreover, if $c\kappa\geq 2\alpha>0$,  we have 
\begin{align*}
   & |\tilde{\Phi}_f(Y)|\leq C\kappa^{-2}\big(1+|\kappa\eta
   |)^{-1}\mathrm{e}^{-c\kappa Y\big(1+|\kappa\eta|^{\f12}\big)/2},\\
   &|\partial_Y\tilde{\Phi}_{f}(Y)|\leq C\kappa^{-1}\big(1+|\kappa\eta
   |)^{-\f12}\mathrm{e}^{-c\kappa Y\big(1+|\kappa\eta|^{\f12}\big)/2},\\
   &\big\|Y^{\beta}\tilde{\Phi}_f\big\|_{L^2}\leq C\kappa^{-\f{2\beta+5}{2}}\big(1+|\kappa\eta|)^{-\f{2\beta+5}{4}},\quad \beta\ge 0,\\
   &\|(\partial_Y\tilde{\Phi}_f,\alpha\tilde{\Phi}_f)\|_{L^2}\leq C\kappa^{-\f32}\big(1+|\kappa\eta|)^{-\f34},
\end{align*}
and 
\begin{align*}
     &|\tilde{\Phi}(Y)|\leq C\kappa^{-2}(1+|\kappa\eta|)^{-1}\mathrm{e}^{-\alpha Y},\\
     & \big\|Y^{\beta}\tilde{\Phi}\big\|_{L^2}\leq C\kappa^{-2}(1+|\kappa\eta|)^{-1}\alpha^{-\f{2\beta+1}{2}},\quad \beta\ge 0.
  \end{align*}
\end{lemma}

\begin{proof}
 By Lemma \ref{lem:Airy-p1} and Lemma \ref{lem:Uineq}, we have
  \begin{align*}
     &\left|\dfrac{A_0(t+B)}{A_0(B)}\right|= \left|\exp\big(\ln(A_0\big(t+B\big))-\ln(A_0(B))\big)\right| = \left|\exp\bigg(\int_{0}^{t}\dfrac{A_0'\big(s+B\big)}{A_0\big(s+B\big)}\mathrm{d}s\bigg)\right|\\
     &\leq  \exp\bigg(\int_{0}^{t}\mathbf{Re}\dfrac{A_0'\big(s+B\big)}{A_0 \big(s+B\big)}\mathrm{d}s\bigg) \leq \exp\bigg(-\int_{0}^{t}\max\big(1/3,c(1+|s+B|^{\f12})\big)\mathrm{d}s\bigg),
  \end{align*}
which along with Lemma \ref{lem:Uineq} implies
\begin{align}\label{A0tB}
\left|\dfrac{A_0(t+B)}{A_0(B)}\right|\leq \exp\left(-\max\big(t/3,c(1+|B|^{\f12})t\big)\right).
\end{align}
Thanks to $\mathbf{Re}\dfrac{A_0'(z)}{A_0(z)}\leq \min(-1/3,-c(1+|z|^{\f12})\big)<0$, we deduce that $\left|\mathbf{Re}\dfrac{A_0'(z)}{A_0(z)}\right| \geq c(1+|z|^{\f12})$ and
  \begin{align}\label{est:A/A'}
     & \left|\dfrac{A_0(z)}{A_0'(z)}\right|=\left|\dfrac{A_0'(z)}{A_0(z)}\right|^{-1}\leq \left|\mathbf{Re}\dfrac{A_0'(z)}{A_0(z)}\right|^{-1} \leq c^{-1}(1+|z|^{\f12})^{-1}.
  \end{align}

Now we are ready to show the estimates of $\tilde{A}(Y)$.
 Lemma \ref{lem:Airy-p1} gives
    \begin{align*}
     |\tilde{A}(Y)|&=\left|\dfrac{A_0'\big(\kappa(Y+\eta)\big)}{A_0'(\kappa\eta)} \right|= \left|\dfrac{A_0(\kappa\eta)}{A_0'(\kappa\eta}\right| \left|\dfrac{A_0\big(\kappa(Y+\eta)\big)}{A_0(\kappa\eta)}\right| \left|\dfrac{A_0'\big(\kappa(Y+\eta)\big)}{A_0\big(\kappa(Y+\eta)\big)}\right|\\
   & \leq C(1+|\kappa\eta|)^{-\f12}\big(1+|\kappa\eta|+\kappa Y\big)^{\f12} \mathrm{e}^{-c\kappa Y\big(1+|\kappa\eta|^{\f12}\big)}\\
     &\leq C\mathrm{e}^{-c\kappa Y\big(1+|\kappa\eta|^{\f12}\big)}.
  \end{align*}
  Then it is easy to find that for $\beta\ge 0$,
  \begin{align*}
     & \|Y^\beta\tilde{A}\|_{L^2}\leq C\big\|\mathrm{e}^{-c\kappa Y\big(1+|\kappa\eta|^{\f12}\big)}\big\|_{L^2}\leq C\kappa^{-\f12-\beta}\big(1+|\kappa\eta|)^{-\f14-\f \beta 2}.
       \end{align*}

Now we turn to deal with $\tilde{\Phi}(Y)$. By Hardy's inequality, we have
  \begin{align*}
     & \|(\partial_Y\tilde{\Phi},\alpha \tilde{\Phi})\|_{L^2}^2=\big|\langle\tilde{A} ,\tilde{\Phi}\rangle\big| \leq \|Y\tilde{A}\|_{L^2}\left\|\tilde{\Phi}/Y\right\|_{L^2}\leq 2\|Y\tilde{A}\|_{L^2}\|\partial_Y\Phi\|_{L^2},
  \end{align*}
  which gives
  \begin{align*}
     &\|(\partial_Y\tilde{\Phi},\alpha \tilde{\Phi})\|_{L^2}\leq 2\|Y\tilde{A}\|_{L^2}\leq C\kappa^{-\f32}\big(1+|\kappa\eta|)^{-\f34}.
  \end{align*}
  
Thanks to the definition of $\tilde{\Phi}_f$ and $|\tilde{A}(Y)|\leq C\mathrm{e}^{-c\kappa\big(1+|\kappa\eta|^{\f12}\big)Y}$, we have
  \begin{align*}
     |\tilde{\Phi}_f(Y)|&= \alpha^{-1}\left|\int_{Y}^{+\infty}\tilde{A}(Z)\sinh\big(\alpha(Z-Y)\big)\mathrm{d}Z\right|\\
     &\leq C\alpha^{-1}\int_{Y}^{+\infty}\mathrm{e}^{-c\kappa\big(1+|\kappa\eta|^{\f12}\big)Z}\sinh\big(\alpha(Z-Y)\big)\mathrm{d}Z\\
     &= C\kappa^{-1}\big(1+|\kappa\eta|^{\f12}\big)^{-1}\int_{Y}^{+\infty} \mathrm{e}^{-c\kappa\big(1+|\kappa\eta|^{\f12}\big)Z}\cosh\big(\alpha(Z-Y)\big)\mathrm{d}Z\\
     &\leq C\kappa^{-1}(1+|\kappa\eta|)^{-\f12}\int_{Y}^{+\infty} \mathrm{e}^{-c\kappa\big(1+|\kappa\eta|^{\f12}\big)Z}\exp\big(\alpha(Z-Y)\big)\mathrm{d}Z\\
     &\leq C\kappa^{-1}(1+|\kappa\eta|)^{-\f12}\big(c\kappa\big(1+|\kappa\eta|^{\f12}\big) -\alpha\big)^{-1}\exp\big(-\big(c\kappa\big(1+|\kappa\eta|^{\f12}\big)-\alpha\big)Y\big)
  \end{align*}
  Therefore, for $c\kappa\geq 2\alpha>0$, we have $c\kappa\big(1+|\kappa\eta|^{\f12}\big)-\alpha\geq c\kappa(1+|\kappa\eta|^{\f12})/2$ and 
  \begin{align*}
     |\tilde{\Phi}_f(Y)|&\leq C\kappa^{-2}(1+|\kappa\eta|)^{-1}\mathrm{e}^{-c\kappa(1+|\kappa\eta|^{\f12})Y/2}.
  \end{align*}
which implies that for $\beta\geq0 $, 
  \begin{align*}
     &\big\|Y^{\beta}\tilde{\Phi}_f\big\|_{L^2}\leq C\kappa^{-\f{2\beta+5}{2}}\big(1+|\kappa\eta|)^{-\f{2\beta+5}{4}}.
  \end{align*}
  
  Notice that
  \begin{align*}
    \partial_Y\tilde{\Phi}_f(Y) &=-\int_{Y}^{+\infty}\tilde{A}(Z)\cosh\big( \alpha(Z-Y)\big)\mathrm{d}Z.
  \end{align*}
  Similarly, we can obtain
  \begin{align*}
  &|\partial_Y\tilde{\Phi}_f(Y)|\leq C\kappa^{-1}\big(1+|\kappa\eta|\big)^{-\f12}\mathrm{e}^{-c\kappa Y(1+|\kappa\eta|^{\f12})/2},\\
    &\big\|Y^{\beta}\tilde{\Phi}_f\big\|_{L^2}\leq C\kappa^{-\f{2\beta+5}{2}}\big(1+|\kappa\eta|)^{-\f{2\beta+5}{4}},\\
   &\|(\partial_Y\tilde{\Phi}_f,\alpha\tilde{\Phi}_f)\|_{L^2}\leq C\kappa^{-\f32}\big(1+|\kappa\eta|)^{-\f34}.
  \end{align*}

 Thanks to $|\tilde{A}(Y)|\leq C\mathrm{e}^{-c\kappa\big(1+|\kappa\eta|^{\f12}\big)Y}$, we get
  \begin{align*}
      \alpha^{-1}\left|\int_{0}^{+\infty}\tilde{A}(Z)\sinh(\alpha Z)\mathrm{d}Z\right|&\leq C\alpha^{-1}\int_{0}^{+\infty}
     \mathrm{e}^{-c\kappa\big(1+|\kappa\eta|^{\f12}\big)Y}\sinh(\alpha Z)\mathrm{d}Z\\
     &= C\kappa^{-1}\big(1+|\kappa\eta|^{\f12}\big)^{-1}\int_{0}^{+\infty}
     \mathrm{e}^{-c\kappa\big(1+|\kappa\eta|^{\f12}\big)Y}\cosh(\alpha Z)\mathrm{d}Z\\
     &\leq  C\kappa^{-1}\big(1+|\kappa\eta|^{\f12}\big)^{-1}\int_{0}^{+\infty}
     \mathrm{e}^{-c\kappa\big(1+|\kappa\eta|^{\f12}\big)Y}2\exp(\alpha Z)\mathrm{d}Z\\
    & \leq C\kappa^{-1}(1+|\kappa\eta|)^{-\f12}\big(c\kappa\big(1+|\kappa\eta|^{\f12}\big) -\alpha\big)^{-1}.
  \end{align*}
 Then for $c\kappa\geq 2\alpha>0$, we have
  \begin{align*}
     & \alpha^{-1}\left|\int_{0}^{+\infty}\tilde{A}(Z)\sinh(\alpha Z)\mathrm{d}Z\right|\leq C\kappa^{-2}(1+|\kappa\eta|)^{-1}.
  \end{align*}
  Recall that 
  \begin{align*}
     \tilde{\Phi}(Y)&= -\dfrac{\mathrm{e}^{-\alpha Y}}{\alpha}
   \int_{0}^{+\infty}\tilde{A}(Z)\sinh(\alpha Z)\mathrm{d}Z+\tilde{\Phi}_f(Y).
  \end{align*}
 Then we obtain
  \begin{align*}
     &|\tilde{\Phi}(Y)|\leq C\kappa^{-2}(1+|\kappa\eta|)^{-1}\big(\mathrm{e}^{-\alpha Y}+\mathrm{e}^{-c\kappa(1+|\kappa\eta|)^{\f12}Y/2}\big)\leq C\kappa^{-2}(1+|\kappa\eta|)^{-1}\mathrm{e}^{-\alpha Y},
  \end{align*}
  which yields that 
  \begin{align*}
     & \big\|Y^{\beta}\tilde{\Phi}\big\|_{L^2}\leq C\kappa^{-2}(1+|\kappa\eta|)^{-1}\alpha^{-\f{2\beta+1}{2}}.
  \end{align*}
 \end{proof}

\begin{lemma}\label{lem:Airy-bound}
Let $c$ be the constant in Lemma \ref{lem:Airy-w}, $c\kappa\geq 2\alpha>0$ and $\mathbf{Im}\eta<0$. Then it holds that
  \begin{align*}
     &|\partial_Y\tilde{\Phi}(0)|\geq C^{-1}(1+|\kappa\eta|)^{-\f12}(\kappa+3\alpha)^{-1},\\
    & \partial_Y \tilde\Phi_f(0)=-\frac{e^{\mathrm i\frac{\pi}{6}}}{3\kappa Ai(e^{\mathrm i\frac{\pi}{6}}\kappa\eta)}+\mathcal{O}\Big(\frac{\eta}{ Ai(e^{\mathrm i\frac{\pi}{6}}\kappa\eta)}\Big)+\mathcal O\Big(\frac{\alpha^2}{(c\kappa(1+|\kappa\eta|)^{1/2}-\alpha)^3}\Big),\\
    &\tilde\Phi_f(0)=\mathrm i\kappa^{-3}+\mathcal O(\kappa^{-1}\eta)+\mathcal O\Big(\frac{\alpha^2}{(c\kappa(1+|\kappa\eta|)^{1/2}-\alpha)^4}\Big).
  \end{align*}
\end{lemma}

\begin{proof}
  By Lemma \ref{lem:ham-bound}, we have
  \begin{align*}
     -\partial_Y\tilde{\Phi}(0) &=\int_{0}^{+\infty}\tilde{A}(Y)\mathrm{e}^{-\alpha Y}\mathrm{d}Y = \int_{0}^{+\infty}\dfrac{Ai\big( \mathrm{e}^{\mathrm{i}\frac{\pi}{6}}\kappa(Y+\eta)\big)}{ Ai\big(\mathrm{e}^{\mathrm{i}\frac{\pi}{6}}\kappa\eta\big)}\mathrm{e}^{-\alpha Y}\mathrm{d}Y\\
     &= \int_{0}^{+\infty}\dfrac{A_0'\big(\kappa(Y+\eta)\big)}{ A_0'(\kappa\eta)}\mathrm{e}^{-\alpha Y}\mathrm{d}Y\\
     &=-\dfrac{A_0(\kappa\eta)}{\kappa A_0'(\kappa\eta)} -\int_{0}^{+\infty}\dfrac{A_0\big(\kappa(Y+\eta)\big)}{ \kappa A_0'(\kappa\eta)}\partial_Y\big(\mathrm{e}^{-\alpha Y}\big)\mathrm{d}Y,
  \end{align*}
  which along with Lemma \ref{lem:Airy-w} gives
  \begin{align*}
     \kappa|A_0'(\kappa\eta)||\partial_Y\tilde{\Phi}(0)|&\geq \big|A_0(\kappa\eta)\big|- \int_{0}^{+\infty}\big|A_0\big(\kappa(Y+\eta)\big)\big| \big|\partial_Y\big(\mathrm{e}^{-\alpha Y}\big)\big|\mathrm{d}Y\\
     &\geq \big|A_0(\kappa\eta)\big|- \int_{0}^{+\infty} \mathrm{e}^{-\kappa Y/3} \big|A_0(\kappa\eta)\big|\big| \partial_Y\big(\mathrm{e}^{-\alpha Y}\big)\big|\mathrm{d}Y\\
     &= \big|A_0(\kappa\eta)\big|+ \big|A_0(\kappa\eta)\big|\int_{0}^{+\infty} \mathrm{e}^{-\kappa Y/3} \partial_Y\big(\mathrm{e}^{-\alpha Y}\big)\mathrm{d}Y\\
     &= \dfrac{\kappa}{3}\big|A_0(\kappa\eta)\big|\int_{0}^{+\infty} \mathrm{e}^{-\kappa Y/3-\alpha Y}\mathrm{d}Y\\
     &=\big|A_0(\kappa\eta)\big|\dfrac{\kappa}{ \kappa+3\alpha}.
  \end{align*}
   This along with Lemma \ref{lem:Airy-p1} shows that 
     \begin{align*}
     &|\partial_Y\tilde{\Phi}(0)|\geq \dfrac{\big|A_0(\kappa\eta)\big|}{(\kappa+3\alpha)|A_0'(\kappa\eta)|}\geq C^{-1}(1+|\kappa\eta|)^{-\f12}(\kappa+3\alpha)^{-1}.
  \end{align*}
  
  Now we estimate $\tilde{\Phi}_f(0)$. We first have 
  \begin{align*}
  	\partial_Y\tilde{\Phi}_f(0) &=-\int_{0}^{+\infty}\tilde{A}(Y)\cosh(\alpha Y)\mathrm{d}Y=I+II,
  \end{align*}
  where
  \begin{align*}
  	I=-\int_0^{+\infty}\tilde A(Y)\mathrm dY,\quad II=-\int_0^{+\infty}\tilde A(Y)(\cosh(\alpha Y)-1)\mathrm d Y.
  \end{align*}
  For $I$, we notice that by the definition of $\tilde A(Y)$,
  \begin{align*}
  	I=&-Ai^{-1}(e^{\mathrm i\frac{\pi}{6}}\kappa\eta)\int_0^{+\infty}Ai(e^{\mathrm i\frac{\pi}{6}}\kappa(Y+\eta)\mathrm dY=-\kappa^{-1}Ai^{-1}(e^{\mathrm i\frac{\pi}{6}}\kappa\eta)\int_{\kappa\eta}^{+\infty} Ai(e^{\mathrm i\frac{\pi}{6}}Z\mathrm )dZ\\
  	=&-\frac{1}{\kappa Ai(e^{\mathrm i\frac{\pi}{6}}\kappa\eta)}\int_{0}^{+\infty} Ai(e^{\mathrm i\frac{\pi}{6}}Z\mathrm )dZ+\frac{1}{\kappa Ai(e^{\mathrm i\frac{\pi}{6}}\kappa\eta)}\int_{0}^{\kappa\eta} Ai(e^{\mathrm i\frac{\pi}{6}}Z\mathrm )dZ\\
  	=&-\frac{e^{\mathrm i\frac{\pi}{6}}}{\kappa Ai(e^{\mathrm i\frac{\pi}{6}}\kappa\eta)}\int_0^{+\infty} Ai(Y)\mathrm dY+\frac{1}{\kappa Ai(e^{\mathrm i\frac{\pi}{6}}\kappa\eta)}\int_{0}^{\kappa\eta} Ai(e^{\mathrm i\frac{\pi}{6}}Z\mathrm )dZ.
  \end{align*}
  On the other hand, we know that $\int_0^{+\infty} Ai(Y)\mathrm dY=1/3$ and observe that
  \begin{align*}
  	\Big|\int_{0}^{\kappa\eta} Ai(e^{\mathrm i\frac{\pi}{6}}Z\mathrm )dZ\Big|\lesssim\kappa\eta.
  \end{align*}
  Then we obtain 
  \begin{align}\label{eq:I-esi}
  	I=-\frac{e^{\mathrm i\frac{\pi}{6}}}{3\kappa Ai(e^{\mathrm i\frac{\pi}{6}}\kappa\eta)}+\mathcal{O}\Big(\frac{\eta}{ Ai(e^{\mathrm i\frac{\pi}{6}}\kappa\eta)}\Big).
  \end{align}
For $II$, we notice that by Lemma \ref{lem:Airy-w},
  \begin{align*}
  	|II|\leq C\int_0^{+\infty}\alpha^2 Y^2 e^{-c\kappa Y\big(1+|\kappa\eta|^{\f12}\big)}e^{\alpha Y}\mathrm dY\leq C\frac{\alpha^2}{(c\kappa(1+|\kappa\eta|^\f12)-\alpha)^3},
  \end{align*}
  which along with \eqref{eq:I-esi} implies 
  \begin{align*}
  \partial_Y \tilde\Phi_f(0)=-\frac{e^{\mathrm i\frac{\pi}{6}}}{3\kappa Ai(e^{\mathrm i\frac{\pi}{6}}\kappa\eta)}+\mathcal{O}\Big(\frac{\eta}{ Ai(e^{\mathrm i\frac{\pi}{6}}\kappa\eta)}\Big)+\mathcal O\Big(\frac{\alpha^2}{(c\kappa(1+|\kappa\eta|)^{1/2}-\alpha)^3}\Big).
  \end{align*}
  
Now we turn to the estimate of $\tilde\Phi_f(0)$. We notice that
  \begin{align*}
  	\tilde\Phi_f(0)=\alpha^{-1}\int_0^{+\infty}\tilde A(Y)\sinh(\alpha Y)\mathrm dY=I_0+II_0,
  \end{align*}
  where
  \begin{align*}
  	I_0=\alpha^{-1}\int_0^{+\infty}\tilde A(Y)\alpha Y \mathrm dY,\quad II_0=\alpha^{-1} \int_0^{+\infty}\tilde A(Y)(\sinh(\alpha Y)-\alpha Y) \mathrm dY.
  \end{align*}
 Thanks to the definition of $\tilde A(Y)$, we have
  \begin{align*}
  	Y\tilde A(Y)=-\frac{\mathrm i}{\kappa^3}\partial_Y^2\tilde A(Y)-\eta\tilde A(Y).
  \end{align*}
  Then we get
  \begin{align*}
  	I_0=&-\frac{\mathrm i}{\kappa^3}\int_0^{+\infty}\partial_Y^2\tilde A(Y)\mathrm dY-\eta\int_0^{+\infty}\tilde A(Y)\mathrm dY\\
  	   =&\mathrm i\kappa^{-3}\partial_Y\tilde A(0)+\mathcal O(\kappa^{-1}\eta)=\mathrm i\kappa^{-3}+\mathcal O(\kappa^{-1}\eta),
  \end{align*}
  and
  \begin{align*}
  	|II_0|\leq C\alpha^2\int_0^{+\infty} Y^3 e^{-c\kappa Y\big(1+|\kappa\eta|^{\f12}\big)}e^{\alpha Y}\mathrm dY\leq C\frac{\alpha^2}{(c\kappa(1+|\kappa\eta|)^{1/2}-\alpha)^4}.
  \end{align*}
 Thus, we conclude that 
  \begin{align*}
  	\tilde\Phi_f(0)=\mathrm i\kappa^{-3}+\mathcal O(\kappa^{-1}\eta)+\mathcal O\Big(\frac{\alpha^2}{(c\kappa(1+|\kappa\eta|)^{1/2}-\alpha)^4}\Big).
  \end{align*}
  \end{proof}

 \section{The Homogeneous Rayleigh equation}
 
 Here we recall a result about the homogeneous Rayleigh equation from \cite{GM}.

 \begin{proposition}\label{Prop:hom-ray}
	For any $0<\alpha<1$, there exists a function $\varphi\in H^1(\mathbb{R}_+)$ such that
	\begin{align*}
		U(\partial_Y^2-\alpha^2)\varphi- U''\varphi=0,\quad Y>0,
	\end{align*}
	and there holds the following properties: $\varphi=\varphi_0+\varphi_2+\varphi_3$ with
	\begin{align*}
		&\varphi_0=Ue^{-\alpha Y},\quad \varphi_1|_{Y=0}=\frac{\alpha}{U'(0)}+\mathcal{O}(\alpha^2),\\
		&\|\partial_Y\varphi_1\|_{L^2}+\|\varphi_1\|_{L^2}\leq C\alpha,\\
		&\|\partial_Y\varphi_2\|_{L^2}+\alpha\|\varphi_2\|_{L^2}\leq C\alpha^{\f32}.
	\end{align*}
	Here $C$ is independent of $\alpha$. If $\frac{U''}{U}\in L^2(\mathbb{R}_+)$ in addition, then $\varphi_1$ and $\varphi_2$ beong to $H^2(\mathbb{R}_+)$.
\end{proposition}

\end{CJK*}
\end{document}